\documentclass[12pt]{amsart}
\usepackage{amsfonts,latexsym}
\usepackage{amssymb} 
\usepackage{latexsym} 
\usepackage{amsfonts} 

\newtheorem{theorem}{Theorem}[section]

\newtheorem{thm}[theorem]{Theorem}
\newtheorem{prop}[theorem]{Proposition} 
 
\newtheorem{cor}[theorem]{Corollary} 
\newtheorem{corollary}[theorem]{Corollary} 
\newtheorem{properties}[theorem]{Properties} 
 
\newtheorem{consequence}[theorem]{Consequence}

\theoremstyle{definition}
\newtheorem{definition}[theorem]{Definition}

\newtheorem{problem}[theorem]{Problem} 
\newtheorem{problems}[theorem]{Problems}

\theoremstyle{remark}
 
\newtheorem{remark}[theorem]{Remark} 
\newtheorem{remarks}[theorem]{Remarks} 
 
\newtheorem{conjectur}[theorem]{Conjectures} 
\newtheorem{question}[theorem]{Question}

\DeclareMathOperator{\cf}{cf}

\DeclareMathOperator{\COV}{COV}

\newcommand{\brfr}{$\hspace{0 pt}$} 

\begin{document}
\pagespan{3}{}



\title[REGULAR ULTRAFILTERS IN ZFC]{ MORE RESULTS  ON REGULAR \\ ULTRAFILTERS IN ZFC }


\author[Paolo Lipparini]{Paolo Lipparini}
\address
{
Dipartimento di Matematica,
Viale della Cicerca Scientifica,
II Universit\`a di Roma (Tor Vergata),
I-00133 ROME
ITALY
}
\begin{abstract}

We prove, in ZFC alone, some new results on regularity and
decomposability of ultrafilters; among them:

(a)
If $m\geq 1$ and
the ultrafilter  $D$ is
$(\beth_m{(\lambda  ^{+n}}),\beth_m{(\lambda^{+n}}))$-regular then $D$ is
$\kappa  $-decomposable for some $\kappa  $ with
$\lambda  \leq\kappa  \leq2^\lambda  $ (Theorem \ref{4.3}(a$'$)).

(b) If  $\lambda $ is a strong limit cardinal and $D$ is $(\beth_m(\lambda^{+n}
), \beth_m(\lambda^{+n} ))$-regular then either $D$ is $(\cf \lambda, \cf \lambda )$-regular or there are arbitrarily large $\kappa<\lambda $ for
which $D$ is $\kappa $-decomposable (Theorem \ref{4.3}(b)).

(c) Suppose that $\lambda $ is singular, $\lambda<\kappa $, $\cf\kappa
\not=\cf\lambda $ and  $D$ is $(\lambda^+,\kappa )$-regular. Then:

(i)
$D$ is either
 $(\cf\lambda  ,{\rm cf}\lambda  )$-regular, or
$(\lambda',\kappa  )$-regular for some $\lambda'<\lambda $ (Theorem
\ref{2.2}).

(ii) If $\kappa $ is regular then $D$ is either
$(\lambda,\kappa )$-regular, or
$(\omega ,\kappa')$-regular for every $\kappa'<\kappa $ (Corollary
\ref{6.4}).

(iii) If
either
(1) $\lambda$ is a strong limit cardinal and $\lambda^{<\lambda
}<2^\kappa $, or
 (2) $\lambda^{<\lambda }<\kappa $,
  then  $D$ is either $\lambda$-decomposable, or
$(\lambda',\kappa )$-regular for some $\lambda'<\lambda $ (Theorem
\ref{6.5}).

(d) If $\lambda$ is singular,  $D$ is $(\mu,{\rm cf}\lambda )$-regular and
there are  arbitrarily
large $\nu<\lambda $ for which $D$ is $\nu $-decomposable
then
$D$ is $\kappa$-decomposable for some $\kappa$ with
$\lambda\leq\kappa\leq\lambda^{<\mu }$ (Theorem \ref{5.1}; actually,
our result is stronger and involves a covering number).

(e) $D\times D'$ is $(\lambda,\mu )$-regular if and only if
 there is a $\nu$ such that $D$
is $(\nu,\mu )$-regular and $D'$ is $(\lambda,\nu')$-regular for all $\nu'<\nu$
(Proposition \ref{7.1}).

We also list some problems, and furnish applications to
topological spaces and to extended logics  (Corollaries \ref{4.6} and
\ref{4.8}).

\end{abstract}
\maketitle                   






%
%
%
%
%
%
%
%
%
%
%
%

\section{Introduction}
\label{intro}

The notion of a $(\lambda ,\mu )$-regular
 ultrafilter has proven particularly useful
in Model Theory, Set Theory and even General Topology,
see e.g.
\cite{BS, CK, BF, KSV, Lp1, lpndj}; \cite{CN, KM, Ka, Fo, DN};
\cite{KV, Ca, Ko, GF, Sa, Lp2, Lp8}.
 
In this paper we are concerned with theorems of the form ``every $(\lambda,\mu)$-regular ultrafilter is $(\lambda',\mu')$-regular'': many results are
known in
this direction,
but most of them rely on assumptions not decided by ZFC (Zermelo-Fraenkel Set
Theory with Choice):
see \cite{Do, DD} and \cite{FMS,Fo}, \cite[p. 427-431]{Wo} for recent advances.

However, some theorems hold without special assumptions: in this paper we prove
some new results of this kind; moreover we furnish some simplified proofs
or
slight improvements of known results. We (try to) collect
all known results valid in ZFC alone (that is, that do not refer to large
cardinals, inner models or other special assumptions of Set Theory).
We also state some problems, and deal with the closely related notion of
$\kappa$-decomposability.

Apart from the results we prove, and their applications,
we hope to convince the reader that the study of $(\lambda ,\mu )$-regular
ultrafilters in ZFC has some interest in itself, and that many theorems are
still to be discovered.
Even those interested solely in independence
results might find some delight in trying to measure
the exact consistency strength of the failure of some
natural (but false) generalizations of the results provable in ZFC. See, e.g.,
Problems \ref{2.8},  \ref{5.6}, Remarks \ref{2.4} \ref{5.5} and
the comments after Theorems \ref{2.13ex2.11},  \ref{2.15ex2.13},
Question  \ref{3.1}, Problems  \ref{5.2},  \ref{6.8},
Proposition  \ref{6.7} and Definition  \ref{6.9}.

Let us recall the basic notions (see \cite{CK, Lp1, CN}, or \cite{KM} 
for other
unexplained notions).

$S_\mu(X)$ denotes the set of all subsets of $X$
of cardinality $<\mu $,
and $S(X)$ denotes the set of all subsets of $X$.

We shall give the definition of
$(\mu ,\lambda )$-regularity in several equivalent forms.

An ultrafilter $D$ is
{\it $(\mu ,\lambda )$-regular }
if and only if

\smallskip

(FORM I) There is a family of $\lambda $ members of $D$
such that the intersection of any $\mu $ members of the family 
is empty.

\smallskip

The above notion is due to \cite{Kei},
who gave it in the following equivalent form: an ultrafilter $D$
over $I$
is  $(\mu ,\lambda )$-regular
if and only if

\smallskip

(FORM II) There is  a function $f:I\to S_\mu(\lambda )$
such that, for every $\alpha\in\lambda $,
$\{i\in I|\alpha\in f(i) \}\in D $.

\smallskip

The two forma are indeed equivalent. If $f$ is a function as given by
Form II, then define, for $ \alpha \in \lambda $,
$X_ \alpha =\{i\in I|\alpha\in f(i) \}$.
$(X_ \alpha)_{\alpha \in \lambda} $
is then a family witnessing $( \mu, \lambda )$-regularity as given
by Form I (cf. \cite[Lemma 1.2]{Kei}).

Conversely, if 
$(X_ \alpha)_{\alpha \in \lambda} $
is a family witnessing $( \mu, \lambda )$-regularity as given
by Form I, define
$f:I\to S_\mu(\lambda )$
by
$f(i)= \{ \alpha \in \lambda | i \in X_ \alpha \} $.
Then $f$  
witnesses $( \mu, \lambda )$-regularity as given
by Form II.

There is a useful (apparently weaker but actually equivalent) 
version of Form II.

\smallskip

(FORM II$'$) There is  a function $f:I\to S_\mu(\lambda )$
such that 
$|\{\alpha \in \lambda |  \{i\in I|\alpha\in f(i) \}\in D \}|=\lambda   $.

\smallskip

Form II trivially implies Form II$'$; conversely if Form II$'$ holds,
let
$X=\{\alpha \in \lambda |  \{i\in I|\alpha\in f(i) \}\in D \}$,
and let $f'(i)=f(i)\cap X$. Since $|X|=\lambda $, then
$\langle S_\mu(X), \subseteq \rangle $
and
$\langle S_\mu(\lambda ), \subseteq \rangle $
are isomorphic; thus, $f'$ composed with an isomorphism
witnesses the $(\mu ,\lambda )$-regularity
of $D$ as given by Form II.

It is interesting to translate the above equivalent conditions in terms of the
ultrapower
of
$ S_\mu(\lambda )$
taken modulo $D$.
It is immediate (from Forms II, II$'$) to see that an ultrafilter $D$
is $(\mu, \lambda)$-regular
if and only if

\smallskip

(FORM III)  In the ultrapower
$\prod_D \langle  S_\mu(\lambda ),\subseteq, \{\alpha \}
 \rangle_{\alpha\in\lambda }   $
there is an element $x$ such that
$d(\{\alpha \} )\subseteq x$,
for every $\alpha\in\lambda $.

\smallskip
Equivalently, 
\smallskip

(FORM III$'$)  In the ultrapower
$\prod_D \langle  S_\mu(\lambda ),\subseteq, \{\alpha \}
 \rangle_{\alpha\in\lambda }   $
there is an element $x$ such that
$|\{ \alpha\in\lambda | \brfr d(\{\alpha \} ) \subseteq x \}| = \lambda $.

\smallskip

Here and in what follows $d$ denotes the
{\it natural embedding} \cite{CK}.

The definitions given according to Forms III, III$'$
are particularly useful for two reasons: first,
we can work in a model of the form
$\langle S_\mu(\lambda ), \subseteq, \dots \rangle $
and freely use \L o\v s Theorem.
 Second, and more important,
these reformulations allow us to translate arguments concerning
$(\mu,\lambda )$-regularity  of ultrafilters into results about models of the above kind.
This aspect will play no role in the present paper, but we hope
that most results presented here can be generalized to this extended
setting (Problem  \ref{8.4}).
The whole matter is described in detail in \cite[Section 0]{Lp1},
and applications are given in  \cite{Lp1, Lp3, Lp4} (in some of those references
the order of $\lambda $ and $\mu $ is exchanged). See also \cite[Theorem 2]{Lp5}.

In the above definitions we assume that $\mu $ and $\lambda $ are infinite
cardinals. The notion of an $(\alpha,\mu  )$-regular
ultrafilter can be defined even for $\alpha $ an ordinal \cite{BK}, and some results
are indeed theorems in ZFC \cite{Ta}, but we shall not deal with this generalized
notion here. See also \cite[Corollary 5]{Lp5}.

We now briefly discuss two notions closely related to
$(\mu,\lambda )$-regularity:
$\lambda $-descending incompleteness and $\lambda$-decomposability. It turns
out that $\lambda $-descending incompleteness is equivalent to $({\rm
cf}\lambda,{\rm cf}\lambda )$-regularity, so that it is nothing but a
reformulation of $(\lambda,\lambda )$-regularity (for $\lambda$  regular).
However, it is useful since it can be defined in terms of the ultrapower of a
linear order (rather than of the partial order
$ S_\lambda (\lambda )$).

As far as $\lambda$-decomposability is concerned, it is equivalent to
$(\lambda,\lambda )$-regularity for $\lambda$ regular, but it  is a stronger
notion
for $\lambda$ singular. Essentially, $D$ is $\lambda$-decomposable if and only if some
quotient of $D$ is uniform over $\lambda$. The main point in applications of
decomposability is that,
for essentially all purposes,
 it is enough to
consider uniform ultrafilters, and  any uniform ultrafilter {\it must} be uniform on
some set: from the cardinality of such a set we can get information about
regularity properties of the ultrafilter; see Remark  \ref{1.5}(b) below   and, e.g.,
the proofs of
Proposition  \ref{3.3}, Theorem  \ref{5.1}, and Corollary
\ref{5.3}.

Now for the definitions:
$D$ is {\it $\lambda $-descendingly incomplete}
if and only if there is a decreasing sequence $(X_\alpha )_{ \alpha\in\lambda} $ of sets in
$D$ with empty intersection. In terms of ultrapowers, $D$ is
$\lambda $-descendingly incomplete
if and only if in $\prod_D \langle\lambda,< \rangle  $
there is an element $x$ such that $d(\alpha)<x$,
for every $\alpha\in\lambda$.

An ultrafilter $D$ over $I$ is said to be {\it uniform}
if and only if $|X|=|I|$ for every $X\in D$. It is enough to consider uniform ultrafilters
because, were $D$ not uniform, it could be replaced by $D|X$, with $X$ a set
in $D$ of minimal cardinality.  An ultrafilter  $D$ over $I$ is {\it principal}
if and only if there is $i\in I$
such that, for every $X\subseteq I$,
$X\in D$ if and only if $i\in X$. Principal ultrafilters are the trivial ones:
if $D$ is principal and uniform, then $|I|=1$.

If $D$ is over $I$, $D$ is $\lambda $-{\it decomposable }
if and only if there is a partition of $I$
into $\lambda $ classes, the union of
$<\lambda $ classes of which never belongs to $D$.
A partition as above will be called a $\lambda$-{\it decomposition} (of $D$).
If $\Pi$ is 
a partition of $I$, we say
that $\Pi$
{\it has $\kappa $ classes modulo $D$} if and only if $\kappa $ is the least cardinal for
which there is $X\in D$ such that $\Pi$ restricted to $X$
has $\kappa $ classes. Notice that, if this is the case, then $\Pi$ induces a
$\kappa $-decomposition of $D$:
just consider $\Pi'=\Pi_{| X} \cup \{I \setminus X\} $.
It is easy to see that $D$ is $\lambda $-decomposable  if and only if
there is a function $f:I\to \lambda $ such that whenever $X \subseteq \lambda $
and $|X|< \lambda $  then $f ^{-1}(X)\not\in D $. Such an $f$
will be called a $\lambda$-{\it decomposition}, too.
Notice that every ultrafilter is $1$-decomposable, and no ultrafilter is
$m$-decomposable for $1<m<\omega $.

If $D$ is over $I$ and $D'$ is over $I'$, then $D'\leq D$
in the {\it Rudin Keisler} (pre-)order means that
there is  a surjection  $f:I\to I'$ such that $X\in D'$
if and only if $f^{-1}(X)\in D $.
In the above situation, some authors say that
$D'$ is a
\emph{quotient} or a 
\emph{projection} of $D$.

We now recall some facts about the above notions;
most of these facts are trivial or easy, but it is hard to find all of them
collected in a single place
(most of them can be found in  \cite[Section 4]{DJK}).

\begin{properties}
\label{1.1}
Assume that $ \lambda $, $\mu$ and  $ \kappa $ are infinite cardinals.  

(i)
 $(\mu,\lambda)$-regularity is preserved by making $\mu$ larger or $\lambda$
smaller.

(ii) If $D'$ is
 $(\mu,\lambda)$-regular
and $D'\leq D$ then $D$ is
 $(\mu,\lambda)$-regular.

(iii) $D$ is $\lambda$-decomposable
if and only if there is a $D'$ which is uniform on $\lambda$ and
$\leq D$. In particular, every ultrafilter uniform over $ \lambda $
is  $ \lambda $-decomposable.

(iv) $D$ is $\lambda$-descendingly incomplete if and only if it is
$\cf \lambda$-descendingly incomplete.

(v) Every
$({\rm cf}\lambda, {\rm
cf}\lambda)$-regular ultrafilter is
 $(\lambda,\lambda)$-regular.

(vi) Every ultrafilter uniform on $\lambda$ is
$({\rm cf}\lambda, {\rm
cf}\lambda)$-regular
and  $(\lambda,\lambda)$-regular.

(vii) Every $\lambda$-decomposable ultrafilter is
$({\rm cf}\lambda, {\rm
cf}\lambda)$-regular
and  $(\lambda,\lambda)$-regular.

(viii) If $\lambda $ is regular, then every
 $(\lambda,\lambda)$-regular ultrafilter is
$\lambda$-decomposable.

(ix) If $\lambda $ is regular, then an ultrafilter is
 $\lambda$-decomposable
if and only if it is
 $\lambda$-descendingly incomplete.

(x) If $D'$ is $\lambda$-decomposable
and $D'\leq D$ then
 $D$ is $\lambda$-decomposable.

(xi) In particular,
if $\lambda $ is regular, then $D$ is
 $(\lambda,\lambda)$-regular
if and only if $D$ is
 $\lambda$-descendingly incomplete, if and only if
$D$ is $\lambda$-decomposable,
if and only if there is a $D'$ which is uniform over $\lambda$ and
$\leq D$.

(xii) If $D$ is  $(\mu,\lambda)$-regular, $\kappa $ is regular, and
$\mu\leq\kappa\leq\lambda $ then $D$ is $\kappa $-decomposable.

(xiii) If $\mu> \lambda $ then every ultrafilter is $(\mu,\lambda)$-regular.
\end{properties}

\begin{proof}
A proof of (iv) can be found, e.g., in \cite[p. 198--199]{CN}; (v) and (vi) 
come
from \cite[ Lemma 1.3(iv)(iii)]{Kei}; (vii) is immediate from (iii), (vi) and (ii).
See e.g. \cite[p. 179]{KM} for a proof of (viii).
(xi) follows from (iii), (vii), (viii) and (ix). (xii) follows from (i) and
(viii). All other statements are trivial. \end{proof}

As a consequence of  \ref{1.1}, most results on regularity of
ultrafilters have many
equivalent reformulations. For example:

\begin{consequence}
\label{1.2}
If $\kappa $ is a regular cardinal, and
$\mu,\lambda $ are cardinals, then the following are equivalent:

(a) Every ultrafilter uniform on $\kappa $ is $(\mu,\lambda )$-regular.

(b) Every  $\kappa $-decomposable ultrafilter  is $(\mu,\lambda )$-regular.

(c) Every  $(\kappa,\kappa ) $-regular ultrafilter  is $(\mu,\lambda )$-regular.
\end{consequence}

\begin{proof}
The equivalence of (b) and (c) is immediate from  \ref{1.1}(xi),
 since $\kappa $ is assumed to be regular.
 (c)$\Rightarrow  $(a) follows from  \ref{1.1}(vi). Finally, if $D$
 is $\kappa $-decomposable, then by  \ref{1.1}(iii) there is $D'\leq
D$,
 $D'$ uniform on $\kappa$; if (a) holds then $D'$ is
$(\mu,\lambda )$-regular, and $D$, too, is $(\mu,\lambda )$-regular
by  \ref{1.1}(ii).
Thus, (a)$\Rightarrow $(b). \end{proof}

Many results on regularity of ultrafilters have the
form described in   \ref{1.2}, and are usually stated as
in clause (a). However, we believe that clause (c) is
the most convenient way to state the results. Formally, (c) is more natural in
the sense that involves just one notion, regularity, rather than two notions,
regularity and uniformity (or decomposability). For example, a classical result
(see Theorem \ref{2.1}(b)) states, when expressed as in clause (a), 
that

\smallskip

(*) every uniform ultrafilter over $\kappa^+$ is $(\kappa,\kappa )$-regular.

\smallskip

If we state this result as in clause (c), that is

\smallskip

(**) every $(\kappa^+,\kappa^+ )$-regular ultrafilter is
$(\kappa,\kappa )$-regular,

\smallskip

\noindent we immediately get that  every $(\kappa^{++} ,\kappa^{++}  )$-regular ultrafilter is
$(\kappa,\kappa )$-regular, a corollary which is not that obvious, if we keep
the theorem in the form (*).

Moreover, in many
 applications, (c) is what is really used (see
Corollaries \ref{4.6} and
\ref{4.8}, as far as this paper is concerned).

The main advantage of clause (c), however, is that
it naturally lends itself to generalizations. We can take a known result in the
form given by (c): every  $(\kappa,\kappa ) $-regular ultrafilter  is
$(\mu,\lambda )$-regular, and try to see whether it generalizes to: every
$(\kappa,\kappa' ) $-regular ultrafilter  is $(\mu,\lambda' )$-regular, for
appropriate $\kappa'\geq\kappa $ and $ \lambda '\geq \lambda  $.

It turns out that usually such a generalized statement holds, and many examples
are provided in the present paper: the statements of
Theorems  \ref{2.2},  \ref{2.13ex2.11}(ii), \ref{6.5}(a), and Proposition \ref{7.3}
 have all been devised by applying the above described pattern (there are more
possibilities: see Conjecture  \ref{2.16ex2.14}, and Problem  
\ref{2.20ex2.18}(b)).

The following result has a very easy proof (for example, it is the easy part of
\cite[Theorem 1.3]{BK}), but it has many interesting consequences.

\begin{prop}
\label{1.3}
If $\lambda $ is regular and the ultrafilter $D$
is $(\lambda,\kappa )$-regular, then ${\rm cf}(\prod_D \langle
\lambda,<\rangle)>\kappa $.
\end{prop}

The following cardinality result has many consequences, too. It is just a
particular case of \cite[Theorem 2.1]{Kei}.

\begin{prop}
\label{1.4}
If the ultrafilter $D$ is
$(\mu,\lambda )$-regular then $|\prod_D 2^{<\mu }|\geq2^\lambda $.
\end{prop}

\begin{remarks}
\label{1.5}
(a) It is also interesting to note that if $\lambda $ is regular then the
following
are equivalent:
(i) $D$ is $(\lambda,\lambda )$-regular; (ii) {\rm cf}$(\prod_D \langle
\lambda,<\rangle )>\lambda $;
(iii) {\rm cf}$(\prod_D \langle \lambda,<\rangle )\not=\lambda $.

(i)$\Rightarrow$(ii) is an instance of Proposition  \ref{1.3};
(ii)$\Rightarrow$(iii) is trivial,
 and
(iii)
implies that $D$ is $\lambda $-descendingly incomplete, hence
$(\lambda,\lambda )$-regular by \ref{1.1}(xi).

Actually, the above remark is the particular case $\lambda=\kappa $
of Theorem
 \ref{2.13ex2.11}(iii).

(a$'$) If $D$ is over $I$, $\lambda $ is regular, and 
$D$ is $(\lambda,\lambda )$-regular, then $|I| \geq\lambda $. Indeed,
 by  \ref{1.1}(viii), $D$ is $\lambda $-decomposable, and by \ref{1.1}(iii)
 there is $D'\leq D$ uniform on $\lambda $, thus $|I|\geq\lambda $.

Moreover, if $D$ is over $I$, $\mu<\lambda $, and $D$ is $(\mu ,\lambda )$-regular, then $|I| \geq\lambda $.
Indeed, by \ref{1.1}(i), $D$ is $(\lambda',\lambda')$-regular for all
$\lambda'$ with $\mu \leq \lambda' \leq \lambda $. By the preceding
paragraph, we get $\lambda'\leq|I|$ for all regular 
cardinals $\lambda'$ with $\mu \leq \lambda' \leq \lambda $, and this implies $\lambda \leq |I|$.

Notice, however, that if $\mu $ is
singular then
every ultrafilter uniform over {\rm cf}$\mu $ is $(\mu,\mu )$-regular,
by  \ref{1.1}(vi)(v) (take $ \lambda = \cf \mu$).

(b) A subset
$X$ of
$ S_\mu(\lambda )$
is {\it cofinal} if and only if for every $y\in S_\mu(\lambda )$ there is $x\in X$
such that $y\subseteq x$; the minimal cardinality of such an $X$ is the {\it
cofinality} of $ S_\mu(\lambda )$, and is denoted by cf$ S_\mu(\lambda
)$. Notice that in Forms II, II$'$ of the  definition of $(\mu,\lambda )$-regularity 
we can equivalently ask that $f: I \to X$, where $X$ is a cofinal subset of 
$S_{\mu}(\lambda )$. Similarly, in Forms III, III$'$
 it is enough
to refer to the ultrapower of a cofinal subset $X$ of $S_\mu(\lambda )$, assuming,
without loss of generality, that $\{\alpha \}\in X $ for every $\alpha\in\lambda
$.

 Whence if
$D$ is $(\mu,\lambda )$-regular then there is a $(\mu,\lambda )$-regular
quotient of $D$ which is uniform over some $\kappa\leq {\rm cf} S_\mu(\lambda
)$; moreover, if either $\mu=\lambda  $ is regular, or $\mu<\lambda $ then we have
$\kappa\geq \lambda $ by (a$'$). (See also Proposition \ref{6.7}(ii).)

In particular, if $\mu<\lambda $, $D$ is $(\mu,\lambda )$-regular and  {\rm cf}$S_\mu(\lambda )=\lambda $
then $D$  is $\lambda $-decomposable; this applies, for example, when $\lambda^{<\mu }=\lambda $, since $| S_\mu(\lambda )|=\lambda^{<\mu }$; in particular, every $(\omega,\lambda )$-regular ultrafilter is $\lambda $-decomposable. Notice also that if $\mu $ is regular then
{\rm cf}$S_\mu(\mu ^{+n} )=\mu^{+n} $, for every natural number $n$.
\end{remarks}

An ultrafilter $D$ is $\lambda$-{\it complete} if and only if the intersection of any
family of
$<\lambda$ members of $D$ belongs to $D$. It is quite easy to show that
if $\lambda > \omega $, then
$D$ is $\lambda$-complete if and only if  for no infinite $\lambda'<\lambda$ $D$ is
$\lambda'$-decomposable, if and only if  for no infinite $\lambda'<\lambda$ $D$ is
$(\lambda',\lambda')$-regular.

A cardinal $\lambda>\omega  $ is {\it  measurable}  if and only if
there exists a $\lambda $-complete ultrafilter uniform over $\lambda $.
By the preceding remark and  \ref{1.1}(iii)(x), if an ultrafilter
$D$ is
$\kappa $-decomposable for some infinite cardinal $\kappa $, the first such $\kappa $ is
either $\omega $ or a measurable cardinal. Moreover,
if an ultrafilter $D$ is $(\kappa,\kappa )$-regular for some infinite cardinal
$\kappa $, the first such $\kappa $ is either $\omega $ or a measurable
cardinal; this is proved as follows: because of  $(\kappa,\kappa )$-regularity
$D$ is not principal, hence uniform over some infinite cardinal, hence
$\kappa'$-decomposable for some infinite $\kappa' $ (by \ref{1.1}(iii)); the first such $\kappa' $
is either
$\omega $ or a measurable cardinal, by above, and $D$ is
$(\kappa',\kappa' )$-regular by  \ref{1.1}(vii); then $D$ is
$\kappa'$-complete, and $D$ is not
$(\kappa,\kappa )$-regular for $\kappa<\kappa'$, by the remark after the
definition of $\lambda $-completeness, thus $\kappa' $ is also the first $\kappa
$ for which $D$ is $(\kappa, \kappa )$-regular.

The cardinal $\lambda>\omega  $ is {\it  $\kappa $-compact} if and only if there is a
$\lambda $-complete $(\lambda,\kappa) $-regular ultrafilter. $\lambda $ is {\it
strongly compact} if and only if it is $\kappa $-compact for all $\kappa $. It is well
known that the above definitions are equivalent to the more usual ones (see e.g.
\cite[Theorems 5.9 and 5.10]{Ket1}  or \cite[Section 15]{KM}).

We shall sometimes use the following known fact (see e.g. \cite[p. 190]{KM}):
if $\lambda \leq \kappa $ are regular, and $\lambda $ is $\kappa $-compact
then $\kappa^{<\lambda }=\lambda $. 
It follows from the above identity and trivial cardinality arithmetic that
if $\lambda $ is regular, 
$\kappa $ is any cardinal, $\cf \kappa\geq \lambda $ and
$\lambda $ 
is $\kappa' $-compact for all $\kappa'<\kappa $, then 
$\kappa^{<\lambda }=\lambda $ still holds.
Moreover, 
if $\kappa $ is singular, $\lambda $ is regular,
$\lambda $ is $\kappa ^+$-compact and $\cf\kappa<\lambda $ then 
$\kappa^+\leq \kappa ^{\cf\kappa }\leq \kappa^{<\lambda }\leq(\kappa^+)^{<\lambda }=\kappa^+$,
hence $\kappa^{<\lambda }=\kappa^+$.
In particular, by \ref{1.1}(ii) and Remark \ref{1.5}(b), if $\lambda $
is regular, then $\lambda$ is  $\kappa $-compact if and only if there is a
$\lambda $-complete $(\lambda,\kappa) $-regular ultrafilter over $\kappa $
(a fact first stated as Theorem 5.10 in \cite{Ket1}).

We have promised to consider only results in ZFC and,
needless to say, measurable and strongly compact cardinals are large cardinals;
but we shall use them only in order to get counterexamples, that is, in order to
show that certain statements are not theorems of ZFC (as usual, whenever we
mention any such large cardinal, we implicitly assume its consistency).
In  recent developments of set theory the notion of  {\it supercompactness}
has played a very central role (see e.g. \cite{Ka}): 
supercompactness is a stronger property than strong compactness
 and, in certain respects, it is better behaved;
however, in our counterexamples we need only the weaker
notion of strong compactness. The relationship between strong compactness and supercompactness has been analyzed in several recent papers by A. Apter and others; see e.g. \cite{A},
 and further references there.

Clearly, if there is a measurable cardinal, there are largely irregular
ultrafilters. In most cases, from an irregular ultrafilter, it is possible to
construct a model of set theory with a large cardinal.

\begin{thm}
\label{1.6}
 \cite[Theorem 4.5]{Do}  If there is no inner model with a
measurable cardinal then, for every cardinal $\kappa $, every ultrafilter
uniform over $\kappa $ is $(\omega,\kappa')$-regular for every $\kappa'<\kappa
$.
\end{thm}

It is conceivable that the conclusion in Theorem  \ref{1.6} can be
improved to
$(\omega,\kappa)$-regular (the maximum of regularity attainable), but, to the
best of our knowledge, a proof has not been found yet (however, \cite{Do} contains
some more results towards this direction).

If $\lambda $ is a limit cardinal and $D$ is an ultrafilter, we say that there
are arbitrarily large $\nu<\lambda $ such that $D$ is $\nu $-decomposable if and
only if for every   $\nu '<\lambda $ there is
$\nu $ such that $\nu '<\nu <\lambda $ and $D$ is $\nu $-decomposable.

Occasionally, we shall use the following principle.

\begin{definition}
\label{1.7}
If $\lambda $ is a limit cardinal, $U'(\lambda )$ means that for every
ultrafilter $D$, if there are arbitrarily large $\nu<\lambda $ such that $D$ is
$\nu $-decomposable,
 then $D$ is $(\lambda ,\lambda )$-regular.
\end{definition}

Slightly weaker principles have been used in \cite{Lp1, Lp4}. See Definitions
\ref{6.9} and the subsequent discussion for the consistency strength of these principles.

$\lambda^{+\alpha } $ denotes the $\alpha^{\rm th}  $ successor of $\lambda  $:
that is,
if $\lambda=\omega_\beta  $ then $\lambda^{+\alpha } $ is
$\omega_{\beta+\alpha } $. $\lambda^{<\mu } $
is $\sup \{ \lambda^{\mu'}|\mu'<\mu  \}  $.

\section{From successors to predecessors.}
\label{fstp}

As the main result of this section, we will prove a generalization (with a new
proof) of the following known result.

\begin{thm}
\label{2.1}
 (a) If the ultrafilter $D$ is  $(\lambda^+,\lambda^+
)$-regular, then $D$ is either
 $({\rm cf}\lambda  ,{\rm cf}\lambda  )$-regular, or
$(\lambda',\lambda^+ )$-regular for some regular $\lambda'\leq\lambda $.

(b) In particular, every $(\lambda^+,\lambda^+ )$-regular ultrafilter is
$(\lambda,\lambda)$-regular.
\end{thm}

(b) follows from (a) because of  \ref{1.1}(v) and  \ref{1.1}(i).
Notice that (a) is stronger
than  (b) only in the case when $\lambda $ is singular. If $\lambda $ is
regular, then ${\rm cf}\lambda =\lambda  $, hence
$({\rm cf}\lambda  ,{\rm cf}\lambda  )$-regularity
is the same as $(\lambda,\lambda)$-regularity,
so that (b) implies (a), and, actually, (b) implies that the first alternative
holds in the conclusion of (a).

As we remarked in  \ref{1.2}, and the comment below, Theorem  
\ref{2.1} is
usually stated in
some equivalent form.

Under instances of GCH, \cite{Ch} proved  \ref{2.1}(b) 
for $ \lambda $ regular, and a 
slightly 
weaker form of
 \ref{2.1}(a). Without assuming GCH, Theorem  \ref{2.1}
 is proved in  \cite[Theorem 1]{CC} and \cite[Theorem 2.1]{KP}. In 
the particular case when
 $\lambda$ is regular, case (b) can be obtained also as a
consequence of either 
\cite[Corollary 1.8]{BK} or \cite{Jo}, using \ref{1.1}(i) and the characterization of
$(\lambda,\lambda )$-regularity
given in Remark  \ref{1.5}(a). See also \cite[Theorem 8.35, Corollary
8.36,
and p. 203]{CN}. Now we present our generalization of Theorem \ref{2.1}.

\begin{thm}
\label{2.2}
 If $\mu\geq\lambda^+$, $\cf \mu\not=\cf\lambda $ 
and $D$ is a $(\lambda^+,\mu )$-regular ultrafilter, then $D$ is either
 $({\rm cf}\lambda  ,{\rm cf}\lambda  )$-regular, or
$(\lambda',\mu )$-regular for some regular $\lambda'\leq\lambda $.
\end{thm}

\begin{proof} For every $z\subseteq \mu $
with $|z|\leq\lambda $ let
$\phi (z, - ):z\to\lambda $
be an injection; and for every $\beta<\lambda $ let
$F(z, \beta )=\{\alpha \in z|\phi(z,\alpha )<\beta  \} $.
Thus, for every $z \subseteq \mu$ and $\beta < \lambda $,
$|F(z,\beta )|\leq|\beta |<\lambda  $.

Let us work in  ${\bf B} = \prod_D {\bf A} $, where
{\bf A} is an appropriate expansion of $S_{\lambda^+}(\mu ) $.
Suppose that $D$ is $(\lambda^+,\mu )$-regular, and let
$x\in B$ witness it (as given by Form III).

Case (a): there is
$\alpha<\mu $ such that $\phi_{\bf B}(x,d(\alpha))>_{\bf B} d(\beta  ) $
for all $\beta<\lambda $. Then $\phi_{\bf B}(x,d(\alpha))$ witnesses the
$\lambda $-descending incompleteness of $D$, hence $D$ is
$({\rm cf}\lambda  ,{\rm cf}\lambda  )$-regular because of
 \ref{1.1}(iv) and of  \ref{1.1}(xi) (applied with ${\rm cf} \lambda
$ in place of $\lambda
$).

Case (b):
otherwise. For every $\alpha<\mu  $ choose some $\beta_\alpha<\lambda $ such that
$\phi_{\bf B}(x,d(\alpha ))\leq_{\bf B}d(\beta_\alpha) $.

If we show that there is $\beta<\lambda $ such that
$|\{\alpha<\mu| \phi_{\bf B}(x,d(\alpha))\leq_{\bf B} d(\beta  )    \}  | =\mu $
then
$|\{\alpha<\mu|d(\{\alpha \} )\subseteq_{\bf B}F_{\bf B}(x,d(\beta+1))\} |=\mu
$,
and this implies that $D$
is $(|\beta |^+,\mu )$-regular (Form III$'$),
since $F_{\bf B}(x,d(\beta+1)) $
belongs to
$\prod_D S_{|\beta |^+} (\mu )$.

In order to show the existence of a $\beta $ as above,
consider an increasing sequence of ordinals
$(\varepsilon_\delta)_{\delta\in{\rm cf}\lambda }$ cofinal in
$\lambda $, and let
$X_\delta= \{ \alpha<\mu| \beta_\alpha < \varepsilon_\delta  \} $.
It is enough to show that $|X_\delta |=\mu$ for some $\delta\in{\rm cf}\lambda
$, since in this case we can take $\beta=\varepsilon_\delta $.

Since we are in Case (b), $\bigcup_{\delta\in{\rm cf}\lambda } X_\delta = \mu $.
If ${\rm cf}\mu>{\rm cf}\lambda $ then it is trivial that $|X_\delta |=\mu$ for
some $\delta\in{\rm cf}\lambda $.
Otherwise by hypothesis ${\rm cf}\mu<{\rm cf}\lambda $. For every $\nu<\mu $
there is  $X_{\delta_\nu } $ of cardinality $\geq\nu$ (otherwise there is
$\nu<\mu$ such that $\mu\leq\nu\cdot{\rm cf}\lambda <\mu$, absurd). We can
consider a sequence of  ${\rm cf}\mu $-many $\nu$'s converging to $\mu$; then the
$\delta_\nu$'s  are bounded by some  $\delta\in{\rm cf}\lambda$, since
${\rm cf}\mu<{\rm cf}\lambda $, and then $|X_\delta | = \mu$, since
$\delta'<\delta $ implies $X_{\delta'}\subseteq X_\delta $.
\end{proof}

Theorem  \ref{2.1}(a) is the particular case $\mu=\lambda^+$ of 
Theorem  \ref{2.2}.

Theorem  \ref{2.2} strengthens the classical result Theorem  
\ref{2.1} only 
in the case when
$\lambda$ is singular:
if  $\mu\geq\lambda^+$ and $D$ is $(\lambda^+,\mu )$-regular, then $D$ is
trivially $(\lambda^+,\lambda^+)$-regular, by  \ref{1.1}(i), hence
$(\lambda,\lambda)$-regular, by Theorem  \ref{2.1}(b);
if $\lambda $ is regular, $({\rm cf}\lambda  ,{\rm cf}\lambda  )$-regularity is
the same as $(\lambda,\lambda)$-regularity,
and we get the first alternative in the conclusion of Theorem  \ref{2.2}.
Anyway, the proof we have given has the advantage of a greater
simplicity (at least, in our opinion).

\begin{remark}
\label{2.3}
 The assumption
 $\mu\geq\lambda^+$
 is not needed in Theorem  \ref{2.2}, but this is the only 
interesting case. Since
 every ultrafilter is ($\lambda ,\mu$)-regular for  $\mu<\lambda$,
 the theorem is
trivially true for $\mu<\lambda$.

For $\mu=\lambda$ Theorem  \ref{2.2} would be false: any ultrafilter 
is ($\lambda^+,\lambda $)-regular, but any principal ultrafilter is neither
$({\rm cf}\lambda  ,{\rm cf}\lambda  )$-regular, nor
$(\lambda',\lambda  )$-regular for $\lambda'\leq\lambda $ (of course, the case
$\mu=\lambda$  is prevented by the hypothesis
$\cf\mu\not= \cf\lambda $).
\end{remark}

\begin{remark}
\label{2.4}
The dichotomy in the conclusion of Theorem  \ref{2.2} cannot be 
avoided. On one side,
if $\kappa$ is $\kappa^{+\omega+1+\alpha }$-compact then there is a
$(\kappa^{+\omega+1},\kappa^{+\omega+1+\alpha })$-regular
(in fact, 
 $(\kappa,\kappa^{+\omega+1+\alpha})$-regular)
ultrafilter $D$ which is $\kappa$-complete, and hence not $(\omega,\omega
)=({\rm cf}\kappa^{+\omega }, {\rm cf}\kappa^{+\omega } )$-regular.

On the other side, 
\cite{BM} shows that
if it is
consistent to have a $\kappa^+$-compact cardinal $\kappa$ then it is consistent
to have an $(\omega_{\omega+1},\omega_{\omega+1})$-regular ultrafilter 
$D$ which for no $n>0$ is
$(\omega_n,\omega_n)$-regular (see also \cite{AH}).
Hence, by \ref{1.1}(i), for no $\lambda ' <\omega _\omega $ $D$ is
$(\lambda ', \omega _{\omega+1})$-regular.  

We do not know the exact consistency strengths (for each $\alpha>0$) of a $(\lambda^+,
\lambda^{+\alpha })$-regular ultrafilter which is not
$(\lambda, \lambda^{+\alpha })$-regular (cf. also Remark \ref{5.5}).

We do not know whether the hypothesis {\rm cf}$\mu\not={\rm cf}\lambda $ in
Theorem  \ref{2.2} can be
omitted (when $\lambda$ is regular the hypothesis is unnecessary, since we
always get $(\lambda,\lambda )$-regularity from
$(\lambda^+,\lambda^+)$-regularity,
hence from $(\lambda^+, \mu)$-regularity). Another case in which {\rm cf}$\mu\not={\rm cf}\lambda $ is not
necessary is presented in Proposition \ref{8.2}.
\end{remark}

We do not know whether the proof of Theorem  \ref{2.2} can be 
extended in order to
show:

\begin{conjectur}
\label{2.5}
If $\mu\geq\lambda^{+n}$ and $D$ is a
$(\lambda^{+n},\mu )$-regular ultrafilter, then $D$ is either $({\rm
cf}\lambda,{\rm cf}\lambda )$-regular, or $(\lambda',\mu )$-regular
for some regular $\lambda'\leq\lambda $.
\end{conjectur}

If $\lambda $ is regular Conjecture  \ref{2.5} is true: by  
\ref{1.1}(i) we
get
$(\lambda^{+n},\lambda^{+n})$-regularity, hence
 $(\lambda,\lambda)$-regularity, by iterating Theorem  \ref{2.1}(b).

In case $n=0$ Conjecture  \ref{2.5} has an affirmative answer, too, 
by the next proposition. Then Theorem \ref{2.15ex2.13}
implies that Conjecture \ref{2.5} is true also in case when $\mu$
is singular and $\cf \mu < \cf \lambda $.

The next proposition is
 an immediate consequence of
\cite[Theorem 0.20(iv)]{Lp1}.  

\begin{prop}
\label{2.6}
If $\lambda $ is singular then every
$(\lambda,\mu )$-regular ultrafilter is either 
$(\cf\lambda ,\brfr \cf \lambda )$-regular, or $(\lambda',\mu )$-regular
for some  $\lambda'<\lambda $.
\end{prop}

If $\mu\geq\lambda $ and $\lambda $ is regular then every
$(\lambda,\mu )$-regular ultrafilter is
$(\lambda,\lambda )=({\rm cf}\lambda,{\rm cf}\lambda )$-regular, so 
that the hypothesis $\lambda $ singular is not necessary in  \ref{2.6},
if $\mu\geq\lambda $.

The following proposition is a variation (and relies heavily) on 
\cite[Theorem 4.5]{Do}  (stated here as Theorem   \ref{1.6}).

\begin{prop}
\label{2.7}
If there is no inner model with a measurable
cardinal, then, for every cardinal $\lambda  $, every
$(\lambda ,\lambda  )$-regular ultrafilter is both
$({\rm cf} \lambda ,{\rm cf} \lambda  )$-regular and
$(\omega ,\nu   )$-regular, for every $\nu<{\rm cf}\lambda $.
\end{prop}

\begin{proof} If $\lambda $ is regular, the statement in  \ref{2.7}
is equivalent to the statement in Theorem  \ref{1.6}, by  \ref{1.2}.

If $\lambda $ is singular, and $D$ is $(\lambda ,\lambda  )$-regular then, by
Proposition  \ref{2.6}, $D$  is either $({\rm
cf}\lambda,{\rm cf}\lambda )$-regular, or $(\lambda',\lambda  )$-regular
for some  $\lambda'<\lambda $.

In the first case, $D$ is
${\rm cf} \lambda $-decomposable by  \ref{1.1}(viii), and the
conclusion follows from
Theorem  \ref{1.6} (with ${\rm cf}\lambda $ in place of $\kappa $).

In the second case, there is a regular $\mu \geq \lambda'$, with ${\rm cf}
\lambda<\mu<\lambda $ and such that $D$ is $(\mu ,\lambda  )$-regular (by
\ref{1.1}(i)),
 hence $D$ is
$\mu $-decomposable by  \ref{1.1}(xii). By applying Theorem
\ref{1.6} with
$\mu$ in place of
$\kappa $,  we get that $D$ is $(\omega ,{\rm cf}\lambda )$-regular (more than
requested, by  \ref{1.1}(i)).
\end{proof}

Essentially, the proof of Proposition  \ref{2.7}
is implicit in the proof of \cite[Proposition 4.2]{Lp1}.

\begin{problem}
\label{2.8}
Find the exact consistency strength of a
$(\lambda,\lambda)$-regular not $({\rm cf}\lambda,{\rm cf}\lambda )$-regular
ultrafilter.
\end{problem}

Remark  \ref{2.4} gives an example of a $(\lambda,\lambda  )$-regular 
not $(\cf\lambda,\cf\lambda )$-regular ultrafilter.

The following theorem, reformulated here using \ref{1.1}(xi), is proved in \cite[Corollary 7]{lpndj}. 

\begin{theorem} \label{adesso2.9}
If $ \lambda  $ is a singular cardinal
and the ultrafilter $D$ is not $ (\cf  \lambda, \cf \lambda ) $-regular,
then the following conditions are equivalent:

(a) There is $ \lambda ' < \lambda $ such that  $D$ is
$ (\kappa , \kappa )$-regular for all regular cardinals $ \kappa $
with $ \lambda ' < \kappa < \lambda $.

(b) $D$ is $(\lambda^+, \lambda ^+)$-regular.

(c) There is $ \lambda' < \lambda $ such that 
$D$ is $(\lambda',\lambda^+ )$-regular.

(d) $D$ is $(\lambda ,\lambda )$-regular.

(e) There is $ \lambda' < \lambda $ such that 
$D$ is $(\lambda',\lambda)$-regular.

(f) There is $ \lambda' < \lambda $ such that
$D$ is $(\lambda'',\lambda'')$-regular
for every $ \lambda ''$ with $ \lambda' < \lambda ''< \lambda $. 
\end{theorem}

By Proposition  \ref{2.7}, if there is a $(\lambda,\lambda  )$-regular not $({\rm
cf}\lambda,{\rm cf}\lambda )$-regular ultrafilter then there is an inner model
with a measurable cardinal. Using 
Theorem \ref{adesso2.9}, a  much stronger result can be proved.

The principle $\Box_\mu$ has been introduced by R. Jensen  in his
study of the fine structure of $L$, and is now ``ubiquitous in set theory'' \cite{SZ}.
For our purposes here, the exact definition of $\Box_\mu$ is not relevant: we
only need to know the following classical result.

\begin{theorem}\label{box}
If $\Box_\mu$ holds then every
$(\mu^+,\mu^+)$-regular ultrafilter is $(\kappa ,\kappa )$-regular, for every  $\kappa\leq\mu$.
 \end{theorem}

A proof for the case $ \kappa $ regular can be found, e. g., in \cite[p. 219-221]{KM}, using  \ref{1.1}(xi)(ii), of
course. The case $\kappa $ singular then follows by \ref{1.1}(v).

In fact, a stronger result holds: \cite[Theorem 1.4]{Do} implies that
if $\Box^-_{\mu ^+}$ (a principle weaker than $\Box_\mu $)
holds, then every  $(\mu^+,\mu^+)$-regular ultrafilter
is $( \omega ,\mu)$-regular, hence $\kappa $-decomposable for all
$\kappa $ with $ \omega \leq \kappa \leq \mu$, by
\ref{1.1}(i) and Remark \ref{1.5}(b).

\begin{prop}
\label{nuo2.9ora2.11}
Suppose that $\lambda $ is singular, and $D$ is
a $(\lambda,\lambda  )$-regular not $(\cf\lambda,\cf\lambda )$-regular
ultrafilter.
Then 
$\Box_\lambda $ fails. 
\end{prop}

\begin{proof}
By Theorem \ref{adesso2.9}(d)$\Rightarrow $(b), $D$ is 
$( \lambda ^+, \lambda ^+)$-regular. Suppose by contradiction that
$\Box_\lambda $ holds. Then Theorem \ref{box} implies that $D$ is
$(\cf\lambda,\cf\lambda )$-regular, contradiction.
\end{proof} 

As far as we know,  the exact consistency strength of the failure of $\Box_\lambda $
for some singular $\lambda  $ has not been evaluated yet. However,
\cite{SZ} announces that in many cases we get the consistency of many Woodin
cardinals. See \cite{SZ, St} and references there for more details. 
See Proposition \ref{vec2.9} for a strengthening of Proposition \ref{nuo2.9ora2.11}. 

Maybe the following is also true.

\begin{conjectur}
\label{2.12ex2.10}
  If $\lambda<\mu $ and  $D$ is
$(\lambda^+,\mu)$-regular
then $D$ is either $\lambda$-decomposable, or $(\lambda',\mu)$-regular
for some regular $\lambda'\leq\lambda $.
\end{conjectur}

The significant case in Conjecture  \ref{2.12ex2.10} is when $\lambda $ is 
singular. When
$\lambda $ is regular, we get $(\lambda,\lambda )$-regularity from
 \ref{1.1}(i) and Theorem
 \ref{2.1}(b), hence $\lambda $-decomposability from  
\ref{1.1}(viii), so
that the conjecture
always holds when $\lambda $ is regular.

For $\lambda $ singular, Conjecture  \ref{2.12ex2.10} is true under some 
cardinality and cofinality
assumptions: see Theorem \ref{6.5}.

The following are proved in \cite[Theorem A and Corollary 1]{Lp5} 
 (we left (ii) as an open problem in the first
version of the present paper).

\begin{thm}
\label{2.13ex2.11}
 (i) If $D$ is a $(\lambda^+,\kappa )$-regular
ultrafilter then either:

(a) $D$ is $(\lambda,\kappa )$-regular, or

(b) the cofinality of the linear order $\prod_D \langle \lambda,< \rangle $ is
{\rm cf}$\kappa$, and $D$ is $(\lambda,\kappa' )$-regular for all
$\kappa'<\kappa $.

(ii) In particular, by \ref{1.1}(i), every $(\lambda^+,\kappa^+)$-regular
ultrafilter is $(\lambda,\kappa )$-regular.

(iii) If $\lambda\leq\kappa $, $\lambda $ is regular and $D$ is
$(\lambda^+,\kappa )$-regular then the following are equivalent:

(a) $D$ is $(\lambda,\kappa )$-regular;
(b) the cofinality of $\prod_D \langle \lambda,<\rangle$
is $>\kappa $;
(c) the cofinality of $\prod_D \langle \lambda,<\rangle$
is $\not={\rm cf}\kappa $.
\end{thm}

Shortly after \cite{Lp5} was published, we realized that
Theorem  \ref{2.13ex2.11} can be proved also by a slight extension of the 
techniques in
\cite{Pr2}. The equivalence of (a) and (b) in Theorem \ref{2.13ex2.11}(iii)
is due to \cite{BK}.

\cite{FMS} shows that, modulo the consistency of (something less than) a huge
cardinal, for every regular $\lambda$ it is consistent
to have a
$(\lambda^+,\lambda^+ )$-regular not $(\lambda,\lambda^+) $-regular
ultrafilter (see also \cite{Ka2} \cite{Hu}, \cite[p. 427-431]{Wo}).

Even more irregular ultrafilters
have been constructed by M. Foreman \cite{Fo}: modulo some large cardinal
consistency assumption, it is consistent to have a uniform
ultrafilter $D$ over $\omega_2$ such that 
$|\prod_D \omega|=\omega_1$
(\cite[Corollary 3.2]{Fo}); as remarked by Foreman, the above identity implies 
$|\prod_D \omega_1|=\omega_2$, and the same argument gives
$|\prod_D \omega_2|=\omega_3$.
Notice also that if $D$ is uniform over $\omega_2$
and $|\prod_D \omega|=\omega_1$
then 
$D$  is not be $(\omega,\omega_1)$-regular, because of
\cite[Theorem 2.1]{Kei}  (see
Proposition
 \ref{1.4} here); moreover, 
$D$ is not $(\omega_1,\omega_2)$-regular, by Theorem \ref{2.13ex2.11}(ii), or by easy cardinality considerations. On the other hand, 
$D$ is $(\omega_2,\omega_2)$-regular by \ref{1.1}(vi),
$(\omega_1,\omega_1)$-regular and
$(\omega,\omega)$-regular by Theorem \ref{2.1}(b).
Since by Proposition \ref{1.4} 
$\omega_1=|\prod_D \omega| \geq 2^\omega \geq \omega_1 $,
we get 
$2^\omega=\omega_1$,
hence $2^{\omega_1}=\omega_2$,
by Theorem \ref{2.18ex2.16}(ii),
hence also   
$2^{\omega_2}=\omega_3$, again by 
\ref{2.18ex2.16}.

It should be mentioned that the construction of an
irregular ultrafilter is only one among many other important applications of
results in \cite{Fo}.
According to \cite{Fo}, the problem of whether  similar results can be obtained for
$\omega_n$ ($2<n<\omega$) in place of $\omega_2$
{\it looks like only a ``technical problem'', but perhaps not}.
Thus we do not know whether for $n>2$ it is consistent to have
some ultrafilter uniform over $\omega_n$ not
$(\omega,\omega_1)$-regular.

An infinitary generalization of Foreman's result, if possible, probably would be
more than a
``technical problem''!

\begin{problem}
\label{2.14ex2.12}
 Is it consistent to have an ultrafilter $D$ which is $(\omega_n,\omega_n)$-regular, for
every $n<\omega $, but which for no $n<\omega $ is
$(\omega_n,\omega_{n+1} )$-regular? (see the remark after Proposition \ref{7.4})

Can we get $|\prod_D \omega_n|=\omega_{n+1} $ for all $n\in\omega$?
\end{problem}

Be that as it may, the situation changes for singular cardinals.

\begin{thm}
\label{2.15ex2.13}
 (i) \cite[Corollary 2.4]{Ka1} If $\lambda$ is singular
then every $(\lambda^+,\lambda^+)$-regular
ultrafilter is $(\lambda,\lambda^+)$-regular.

(ii) \cite[Corollary B]{Lp5}
 Suppose that $\kappa$ is a singular cardinal, $\kappa>\lambda $ and
either
$\lambda$ is regular, or ${\rm cf}\kappa<{\rm cf}\lambda $. Then every
$(\lambda^{+n},\kappa )$-regular
ultrafilter is $(\lambda,\kappa )$-regular.
\end{thm}

Notice that the conclusion in Theorem  \ref{2.15ex2.13}(i) cannot be 
improved to ``there
is $\lambda '<\lambda$  such that D is
$(\lambda',\lambda ^+)$-regular'',
because of the result from
\cite{BM} mentioned in Remark  \ref{2.4}.

Similarly,
 the conclusion in Theorem  \ref{2.15ex2.13}(ii) cannot be improved to 
``there
is $\lambda '<\lambda$  such that D is
$(\lambda ',\kappa)$-regular'':
if there are $\omega_1$ strongly compact cardinals, say  $\lambda_\alpha $
$(\alpha\in\omega_1)$, $\lambda_\alpha $ increasing, let
$\lambda=\sup_{\alpha\in\omega_1} \lambda_\alpha $ and let $\kappa$
be any cardinal $>\lambda $ (in particular, we can have $\cf \kappa=\omega $).
For each $\alpha\in\omega_1$ there is an ultrafilter $D_\alpha $ which is
$(\lambda_\alpha ,\kappa )$-regular and not $(\nu, \nu)$-regular for
all $\nu<\lambda_\alpha $.
Let D be uniform over $\omega_1$, and consider the sum $D'=\sum_D D_\alpha $
(see Section \ref{rup} for the definition).
By Proposition \ref{7.4}(c) (with $\mu=\nu=\kappa $, using \ref{1.1}(i)(xiii)) $D'$ is
$(\lambda,\kappa )$-regular; by \ref{7.4}(d) $D$ is not $(\lambda', \lambda')$-regular,
for every regular cardinal $\lambda'$ with $\omega_1<\lambda'<\lambda$, hence
not $(\lambda',\kappa )$-regular for
 every $\lambda'<\lambda $, by \ref{1.1}(i).

We expect that just one strongly compact cardinal is sufficient
in order to obtain a
counterexample as above:
if $\lambda $ is
 $\lambda ^{+\alpha}$-compact then it is probably possible to
make $\lambda $
singular (by some variation on Prikry forcing \cite{Pr1})
in such a way that  in the resulting model there
is a
 $(\lambda ,\lambda ^{+\alpha})$-regular ultrafilter which for no $\lambda
'<\lambda $ is
 $(\lambda' ,\lambda ^{+\alpha})$-regular.

Notice also that \ref{2.1}(b) and \ref{2.15ex2.13}(i) imply that if
$\lambda$ is singular and $n>0$ then every $(\lambda^{+n},\lambda^{+n})$-regular
ultrafilter is $(\lambda,\lambda^+)$-regular.

In Theorem  \ref{2.15ex2.13}(ii)
the hypothesis ``$\kappa $ is a singular cardinal''
cannot be weakened to  ``$\kappa $ is a limit cardinal'':
see Remark
\ref{5.5}.

\begin{conjectur}
\label{2.16ex2.14}
 If $\lambda$ is singular and $\lambda<\mu $
then every $(\lambda^+,\mu )$-regular ultrafilter is $(\lambda,\mu )$-regular
 (maybe some assumption on {\rm cf}$\mu$ is necessary, say {\rm
cf}$\mu\not={\rm cf}\lambda$, or some similar condition).
\end{conjectur}

A positive answer to Conjecture  \ref{2.16ex2.14} would encompass (the case 
when $\lambda $ is
singular of) Theorem  \ref{2.2}, in view of Proposition   \ref{2.6}.
Thus, the dichotomy in the conclusion of Theorem  \ref{2.2} might 
depend only on the
dichotomy in the conclusion of  \ref{2.6}. Also, an affirmative 
solution to
Conjecture  \ref{2.16ex2.14} with no assumption on
{\rm cf}$\mu$ would make the assumption {\rm cf}$\mu\not={\rm cf}\lambda$ in
Theorem  \ref{2.2} unnecessary. Theorem  \ref{2.15ex2.13}(ii) shows that if 
Conjecture  \ref{2.16ex2.14} is
true
whenever {\rm cf}$\lambda\leq {\rm cf} \mu $
then it is true for every $\lambda$ and $\mu$. See Theorem  
\ref{2.15ex2.13}, Corollary \ref{6.4}
and Propositions \ref{6.7} and \ref{8.2} for partial answers to
Conjecture  \ref{2.16ex2.14}.

\begin{problem}
\label{2.17ex2.15}
   The proof of Theorem  \ref{2.15ex2.13}(i) uses notions of a large
cardinal
nature, as least functions, weak normality, and the like (see Section 
\ref{lfaa}). Is
there a simpler
proof which makes use of elementary arguments only?
\end{problem}

Also the following results can be proved:

\begin{thm}
\label{2.18ex2.16}
 (a) (\cite[Corollary 2.2]{BK}, \cite[Theorem 1.11]{Ket}) If
$2^\kappa=\kappa^+$ and $2^{\kappa^+}>\kappa^{+ +}$ then every
$(\kappa^+,\kappa^+)$-regular  ultrafilter is $(\kappa,\kappa^+)$-regular.

(b) If
$2^{\kappa^{+n} } =\kappa^{+n+1} $ and $2^{\kappa^{+n+1}}>\kappa^{+n +2}$ then
every $(\kappa^{+n+1},\kappa^{+n+1})$-regular  ultrafilter is
$(\kappa,\kappa^+)$-regular.

(c) \cite[Theorem 7.2.1]{JP} If $\omega_{\omega_1}>2^\omega>\omega_1$ and
$2^{\omega_1}>2^\omega$ then every $(\omega_1,\omega_1)$-regular
ultrafilter is $(\omega,\omega_1)$-regular.
\end{thm}

\begin{proof}[ Proof of (b)]
 Suppose that $D$ is $(\kappa^{+n+1},\kappa^{+n+1})$-regular. By (a)
(with $\kappa^{+n}$ in place of $\kappa $) $D$ is
$(\kappa^{+n},\kappa^{+n+1})$-regular. Now apply Theorem  
\ref{2.13ex2.11}(ii) $n$ times. \end{proof}

For sake of completeness we shall mention also the following result, though it
deals with (moderately) large cardinals.

\begin{thm}
\label{2.19ex2.17}
 \text{\cite{CC}; see also \cite[Theorem 3.1]{Pr}}  If $\kappa $ is inaccessible and not
$\omega $-Mahlo, or is weakly inaccessible and not $\omega $-weakly-Mahlo then every
$(\kappa,\kappa )$-regular ultrafilter is either $(\lambda,\kappa )$-regular for
some $\lambda<\kappa $, or $(\lambda,\lambda )$-regular for all $\lambda<\kappa
$.
\end{thm}

\begin{problems}
\label{2.20ex2.18}
 (a) Does
Theorem  \ref{2.19ex2.17} apply to more inaccessible cardinals (for 
example, to $\kappa $'s
which are not $\omega+1$-Mahlo)? It probably applies to more cardinals (see
\cite{Ha}), but \cite{Shr} imposes limitations.

(b) We do not know whether a two cardinals version of  \ref{2.19ex2.17} 
holds, that is whether
(for $\kappa $ as in the statement of  \ref{2.19ex2.17}, and appropriate 
$\kappa' $) it is
true that every $(\kappa,\kappa' )$-regular ultrafilter is either
$(\lambda,\kappa' )$-regular for some $\lambda<\kappa $, or
$(\lambda,\lambda )$-regular for all $\lambda<\kappa $.
 \end{problems}

\section{Down from exponents (Part I)}
\label{dfe1}

One could ask whether a version of Theorem  \ref{2.1}(b) holds when 
successors
are replaced by exponents; namely whether it can be proved that every
$(2^\lambda,2^\lambda )$-regular
ultrafilter is $(\lambda,\lambda )$-regular. In this form, the problem has a
negative answer: if we start with a model satisfying GCH and with a measurable
cardinal $\mu$, and we add $\mu$ Cohen reals, then in the resulting model
$2^\omega=2^{\omega_1}=\mu $, and there is a $\mu$-decomposable ultrafilter
which is not $\lambda$-decomposable, for every $\lambda$ with
$\omega<\lambda<\mu $ (see e.g. \cite[p. 62]{Ket}; see \cite{Shr} for related results).
Thus, by \ref{1.1}(xi), there is a
$(2^{\omega_1},2^{\omega_1})$-regular
not $(\omega_1,\omega_1)$-regular ultrafilter.

On the other side, if there is no inner model with a measurable cardinal, then
by Proposition  \ref{2.7}
every
$(2^\lambda,2^\lambda )$-regular
ultrafilter is
$(\omega,\nu )$-regular for all $\nu < {\rm cf} 2^ \lambda  $ , hence
$(\lambda,\lambda )$-regular, by  \ref{1.1}(i) and since
$\lambda<{\rm cf}2^\lambda  $.

By the above remarks, the existence of a measurable cardinal is equiconsistent
with the existence, 
for some $ \lambda $,
of a
$(2^\lambda,2^\lambda )$-regular
not $(\lambda,\lambda )$-regular ultrafilter.

In a previous version of the present paper we refined the above problem
to:

\begin{question}
\label{3.1}
 Suppose that $2^\kappa=\lambda $, and that $\kappa$ is the
first
cardinal such that $2^\kappa=\lambda $. Is it true that every $(\lambda,\lambda
)$-regular
ultrafilter is $(\kappa,\kappa )$-regular?
\end{question}

Also \ref{3.1} has a negative answer: start with a model of GCH in 
which $\mu $
is $\mu ^{+\omega_1+1}$-compact, and add
 $\mu^{+\omega_1}$ Cohen reals. Then in the resulting model
$2^{\omega}=\mu^{+\omega_1}$ and
there is a $(\mu,\mu^{+\omega_1+1})$-regular ultrafilter which is
not $\nu$-decomposable, for every $\nu$ with $\omega<\nu<\mu$.
If we put
$\kappa =\omega_1$ and $\lambda =
\mu^{+\omega_1+1}$, we have that $\kappa $ is the first cardinal such that
$2^\kappa =\lambda $, and there exists a $(\lambda ,\lambda )$-regular not
$(\kappa ,\kappa )$-regular ultrafilter.

The above example also shows that \ref{3.1} can be false even if we 
strengthen
the hypothesis of $(\lambda ,\lambda )$-regularity to
$\lambda $-decomposability.

However, we do not know whether it is possible to get a counterexample to
\ref{3.1} starting with something less than a 
$\mu^{+\omega_1+1}$-compact
 cardinal $\mu$. Does a measurable suffice?

In spite of the above counterexamples, we have some positive results, and
we
can actually show that some amount of decomposability
(and hence regularity) can be brought down from exponents.

Let $\beth_n(\lambda )$ denote the $n^{\rm th} $ iteration of the power set of
$\lambda$; that is, $\beth_0(\lambda)=\lambda $, and
$\beth_{n+1}(\lambda)=2^{\beth_n(\lambda)}$.

\begin{thm}
\label{3.2}
(a) If $D$ is $(2^{2^\lambda }, 2^{2^\lambda })$-regular (or just $|\prod_D
2^{2^\lambda }|>2^{2^\lambda }$)
then $D$ is $\kappa $-decomposable for some $\kappa $ with
$\lambda\leq\kappa\leq2^\lambda $.

(a$'$) More generally, if $n\geq1$ and $D$ is $(\beth_n(\lambda ), \beth_n(\lambda
))$-regular (or just if $|\prod_D \beth_n(\lambda )|> \beth_n(\lambda )$)
then $D$ is $\kappa $-decomposable for some $\kappa $ with
$\lambda\leq\kappa\leq2^\lambda $.

(b) Suppose that $\lambda $ is a strong limit cardinal, and that
$D$ is $(\beth_n(\lambda ), \beth_n(\lambda ))$-regular for some $n\geq 0$ (or
just that $|\prod_D \beth_n(\lambda )|>\beth_n(\lambda )$).
Then either $D$ is $({\rm cf} \lambda, {\rm cf} \lambda )$-regular or there are
arbitrarily large $\kappa<\lambda $ for which $D$ is $\kappa $-decomposable.
If in addition $U'(\lambda )$ holds (recall Definition  \ref{1.7}) then $D$ is
$(\lambda, \lambda)$-regular.
\end{thm}

In order to prove Theorem \ref{3.2} we need the following 
proposition, a slight
improvement on \cite[Theorem 1]{AJ}, and  which has independent interest.

\begin{prop}
\label{3.3}
Suppose that $|\prod_D\mu |>\mu $,
and let $\nu $ be the smallest cardinal such that $\mu^\nu >\mu $. Then either:

(i) $D$ is $\kappa $-decomposable for some $\kappa $ with $\nu\leq\kappa\leq\mu
$;
or

(ii) for every $\nu'<\nu$ there is $\kappa $ such that
$\nu'\leq\kappa<\nu$ and $D$ is $\kappa $-decomposable; in addition, $D$ is
$\kappa $-decomposable for some $\kappa $ with $\mu<\kappa\leq2^\nu $; and
moreover $|\prod_D\mu |=\sup_{\nu'<\nu } |\prod_D \nu'| $.

In particular, if $\nu $ is a successor cardinal, then $D$ is $\kappa $-decomposable
for some $\kappa $ with $\nu^-\leq\kappa\leq\mu $ ($\nu^-$ denotes the
predecessor of $\nu $).

\end{prop}

\begin{proof} [Proof of \ref{3.3}.] 
Let $D$ be over $I$.
Recall from the introduction that if $\Pi$ is 
a partition of $I$, $\Pi$
 has $\kappa $ classes modulo $D$ if and only if $\kappa $ is the least cardinal for
which there is $X\in D$ such that $\Pi$ restricted to $X$
has $\kappa $ classes. If this is the case, then $\Pi$ induces a
$\kappa $-decomposition of $D$. 

Any representative $f:I \to \mu$ of an element $f_D\in \prod_D\mu $
induces the partition $\Pi_f=\{ (i,j)| f(i)=f(j)\} $, which has at most $\mu $
classes. If for some $f:I \to \mu$ $\Pi_f$ has $\kappa $ classes modulo $D$, and
$\nu\leq\kappa\leq\mu $,
then (i) holds, so that we can suppose that

\smallskip

(*) for every $f : I \to \mu $, $\Pi_f$
has $<\nu$ classes modulo D (whence every $f_D\in\prod_D\mu$ has a
representative $f$ such that $\Pi_f$ has $<\nu$ classes).

\smallskip

We now find an ordinal $\rho$ and construct a chain of partitions $\Pi_\alpha $
$(\alpha\leq\rho)$
of $I$ according to the following rules:

(a) $\Pi_0$ is the trivial partition;

(b) if $\alpha $ is limit, $\Pi_\alpha $ is the common refinement of the
$\Pi_\beta$'s, for $\beta<\alpha $;

(c) if $\alpha=\beta+1$, there are two cases:

(c$_1$) every element of $\prod_D\mu $ can be represented as $f_D$, for some $f$
such that $f(i)=f(j)$ whenever $i$ and $j$ belong to the same $\Pi_\beta $
class. In this case, take $\rho=\beta $, and the construction ends.

(c$_2$) Otherwise: take an element $f_D$ which cannot be represented in that
way, and choose by (*) a representative $f$ in such a way that $\Pi_f$ has $<\nu$
classes;
then define $\Pi_\alpha $ to be the common refinement of $\Pi_\beta $ and
$\Pi_f$.
Thus, $\Pi_\alpha $ properly refines $\Pi_\beta $.

Notice that

\smallskip

(**) if $\alpha<\nu $ then, by (c$_2$), $\Pi_\alpha $ (if defined)
has at most $\nu^{|\alpha |}\leq\mu^{<\nu }=\mu  $
classes; and that $\Pi_\nu $ (if defined) has at most
$2^\nu $ classes.

\smallskip

$\Pi_\rho $ has at least $\nu $ classes modulo $D$,
 since if it has only $\nu'<\nu $ classes modulo $D$ then 
by (c$_1$) $|\prod_D \mu
|\leq\mu^{\nu'}=\mu $, a contradiction. Whence if $\rho<\nu $ then $\Pi_\rho $
induces a  $\kappa $-decomposition of $D$ for some $\kappa $ with
$\nu\leq\kappa\leq\mu $, and we are in case (i).

Hence, we can suppose $\rho\geq\nu $.

We now show:

\smallskip
{\it Claim. If $\nu'\leq\rho $ then $\Pi_{\nu'}$ has at least $\nu'$ classes
modulo $D$.
}
\smallskip
\begin{proof}[Proof of the Claim] Fix $\nu'\leq\rho $, and suppose that $\Pi_{\nu'}$
 has $\kappa $
classes modulo $D$, witnessed by $X\in D$. For $\alpha\leq\nu' $, let
$\Pi^*_\alpha $ be $\Pi_\alpha $ restricted to $X$; for every $\alpha<\nu' $
$\Pi^*_{\alpha+1}  $ is a proper refinement of $\Pi^*_\alpha $, 
by (c$_2$) and since $f_D=g_D$,
if $f(i)=g(i)$ for every $i\in X$.

We shall define by induction a sequence $(C _ \alpha ) _{ \alpha \leq \nu'} $
of subsets of $X$ in such a way that, 
for $ \alpha \leq \nu' $,   $|C_ \alpha |= | \alpha | $,
and different elements of $C _ \alpha $ belong to different
$\Pi^*_ \alpha $ classes. Let   
$C_0=\emptyset$; $C_\alpha=\bigcup_{\beta<\alpha } C_\beta $ for $\alpha
$ limit; and, if $\alpha=\beta+1 $, let $C_\alpha =C_\beta \cup \{ p\} $, where
$p\in X$ is such that no element of $C_\beta $ is in the same $\Pi^*_\alpha $
class of $p$ (such a $p $ exists since $\Pi^*_\alpha $ properly refines
$\Pi^*_\beta  $).

Thus, $|C _{\nu'}  |= | \nu' | $,
and different elements of $ C _{\nu'}  $  belong to different
$\Pi^* _{\nu'} $ classes;
hence $\Pi^*_{\nu'}$
has at least $\nu' $ classes, that is $\nu'\leq\kappa $, so that the claim is
proved. \end{proof}

Proof of \ref{3.3} (continued). 
Now consider $\Pi_\nu $: the claim shows that $\Pi_\nu $ has at least $\nu $
classes modulo $D$; on the other side, $\Pi_\nu $ has at most $2^\nu $ classes,
so that $D$
is
$\kappa $-decomposable for some $\kappa $ with $\nu\leq\kappa\leq2^\nu $
(whence $\mu<\kappa\leq2^\nu $ if we are not in case (i)).

Moreover, for every $\nu'<\nu $ the claim shows that
$\Pi_{\nu'} $
has at least $\nu' $ classes modulo $D$; by (**) $\Pi _{\nu'}$ has at most $\mu $
classes, so that
$\Pi_{\nu'} $
induces a
$\kappa $-decomposition of $D$ for some $\kappa$ with
$\nu'\leq\kappa\leq\mu $ (whence $\nu'\leq\kappa<\nu $
if we are not in case (i)).

The above proof is essentially taken from \cite{AJ}. The only difference is that \cite{AJ}
applies the arguments in the proof of the claim only to the case $\nu'=\nu$.
Considering the general case $\nu'\leq\nu $ provides a strengthening without
which we could not prove Theorem \ref{3.2}. Notice that statements in \cite{AJ}
talk about descending incompleteness; however, proofs actually deal with decomposability. 

It remains to prove the last identity in (ii). This is easy:
 we mentioned that, if (i) fails, we can suppose that for every $f_D$
$\Pi_f$ has $<\nu $ classes modulo $D$. Then $\prod_D \mu=\bigcup \{\prod_D
x|x\in S_\nu(\mu )\} $; but $|\prod_D \mu |>\mu $ and $|S_\nu(\mu )|=\mu^{<\nu
}=\mu $, so that $|\prod_D\mu|=\sup_{\nu '<\nu } |\prod_D \nu'| $, since
$|x|=\nu'$ implies $|\prod_D x|=|\prod_D\nu'|$.
 \end{proof}

The proof of Proposition \ref{3.3} should be compared with the proof 
of \cite[Lemma 2]{Si}.

\begin{proof}[Proof of \ref{3.2}.]
It is well known (at least for $\mu $ regular)
that if $D$ is $(\mu,\mu )$-regular then
$|\prod_D\mu|>\mu $ (if $\mu $ is singular, use eventually different functions
from $S_\mu(\mu )$ to $\mu $: see \cite[Theorem 0.25]{Lp1}).

Moreover, standard arguments (e.g. \cite[Lemma 4]{AJ}) show that, for
every $\mu $,

(*) $|\prod_D 2^\mu|>2^\mu$ implies $|\prod_D \mu|>\mu$.

Whence, in case (a), $|\prod_D 2^{2^\lambda } |>2^{2^\lambda }$, by hypothesis
and the first remark, hence $|\prod_D 2^{\lambda } |>2^{\lambda }$ by (*).

Arguing in a similar way and iterating (*) $n-1$ times we get
$|\prod_D 2^{\lambda } |>2^{\lambda }$ also from the hypothesis of (a$'$).

Now take $\mu=2^\lambda$ in Proposition \ref{3.3}: by standard 
cardinal arithmetic,
the least $\nu $ such that $(2^{\lambda })^\nu > 2^\lambda  $ is $\geq\lambda^+
$, so that the conclusion of (a) and (a$'$) follows from 
Proposition
\ref{3.3}.

As for (b), arguing as before, we get $|\prod_D \lambda |>\lambda $ in each
case.

If $\lambda $ is regular, then $\lambda $ is the least $\nu$ such that
$\lambda^\nu>\lambda $, since $\lambda $ is supposed to be a strong limit
cardinal.
Hence the conclusion follows from  Proposition \ref{3.3} and
\ref{1.1}(vii).

If $\lambda $ is singular and $D$ is not
$({\rm cf}\lambda,{\rm cf}\lambda )$-regular, then 
$ |\prod _D \lambda | > \lambda $ and  
an easy argument 
(\cite[Lemma 2]{AJ} or  \cite[Lemma 2.1]{Lp1})  show that $|\prod_D
\lambda'|>\lambda $, for some $\lambda'<\lambda $.
If $\lambda'\leq\mu<\lambda $ then $|\prod_D 2^\mu|\geq|\prod_D
\lambda'|>\lambda>2^\mu$, since $\lambda $ is a strong limit cardinal. By (a$'$)
applied with $n=1$ and  $\mu$ in place of $\lambda $, $D$ is $\kappa $-decomposable for
some $\kappa $ with $\mu\leq\kappa\leq 2^\mu<\lambda  $.
By taking  arbitrarily large $\mu $'s $<\lambda $ we get arbitrarily large
$\kappa $'s $<\lambda $ for which $D$ is $\kappa $-decomposable.

As for the last statement, if $D$ is
$({\rm cf}\lambda,{\rm cf}\lambda )$-regular
 then $D$ is $(\lambda,\lambda )$-regular by  \ref{1.1}(v). If $D$ is
not
$({\rm cf}\lambda,{\rm cf}\lambda )$-regular, we have just shown that
there are arbitrarily large $\kappa $'s $<\lambda $ for which $D$ is
$\kappa $-decomposable, hence $D$ is $(\lambda,\lambda )$-regular by applying the definition
of $U'(\lambda )$ (Definition  \ref{1.7}).
\end{proof}

\begin{remarks}
\label{3.4}
 (a) The hypothesis that $\lambda $ is a strong limit
cardinal in Theorem \ref{3.2}(b) is necessary. This is particularly 
evident in the
case $n=0$:
 if $2^\omega>\lambda $, $\cf \lambda>\omega $, and 
$D$ is uniform over $\omega $,  then, 
by Proposition \ref{1.4},
$|\prod_D \lambda |\geq|\prod_D \omega
|=2^\omega>\lambda $, but $\kappa=\omega $ is the only infinite $\kappa $ for which $D$
is $\kappa $-decomposable.

(b) The above counterexample involves the weaker assumption $|\prod_D
\lambda|>\lambda $; indeed, if we assume $(\lambda,\lambda )$-regularity then
case $n=0$ of \ref{3.2}(b) is always true for every limit cardinal 
$\lambda $, because
of Proposition  \ref{2.6} and  \ref{1.1}(i)(viii).

(c) However, case $n=1$ of \ref{3.2}(b) may fail even when $D$ is 
assumed to be
$(2^\lambda,2^\lambda )$-regular, if $\lambda $ is not a strong limit cardinal.
This goes exactly as in the example at the beginning of this section: start with
a model of GCH in which $\mu $ is measurable and add $\mu $ Cohen reals. Take
$\lambda<\mu $ singular with $\omega\not={\rm cf}\lambda $. In the resulting
model $2^\omega=2^\lambda=\mu $, and there is an ultrafilter $D$ which is
$\kappa $-decomposable exactly for $\kappa=\omega $ and $\kappa=\mu $; thus $D$
is not $({\rm cf}\lambda,{\rm cf}\lambda )$-regular, by \ref{1.1}(viii), but $D$ is
$(2^\lambda,2^\lambda )$-regular (actually, $2^\lambda $-decomposable). Hence
the conclusion of \ref{3.2}(b) fails.

(d) We were led to the formulation of  Theorem \ref{3.2} by easy 
results of the
following kind: {\it every $(2^\omega,2^\omega )$-regular
ultrafilter is $(\omega,\omega )$-regular}. This fact can be obtained as a
consequence of   Theorem \ref{3.2}(b), taking $\lambda=\omega $ and 
$n=1$, and since
no ultrafilter is $m$-decomposable, if $1<m<\omega $. However, the following is
a simpler proof: as we mentioned after the definition of measurability,
the first cardinal $\mu$ for which an ultrafilter is $(\mu,\mu )$-regular
is either $\omega $ or a measurable cardinal, and it is well-known that
measurable cardinals are strongly inaccessible, thus $2^\omega <$ the first
measurable cardinal.
\end{remarks}

\begin{problems}
\label{3.5}
(a) Does (i) follow from the hypotheses of Proposition \ref{3.3}? 
Notice that this
would improve
the conclusion of Theorem \ref{3.2}(a)(a$'$) to 
$\lambda<\kappa\leq2^\lambda $, and
would render the hypothesis $U'(\lambda ) $ unnecessary in 
\ref{3.2}(b), in the case
when $\lambda $ is regular.

{\rm \cite[p. 832]{AJ}} asked something slightly weaker, namely whether,
under
the
hypotheses of \ref{3.3}, $D$ is $\kappa $-{\hspace{0 pt}}descendingly incomplete for 
some $\kappa $
with
$\nu\leq\kappa \leq\mu $.

(a$'$) Can we show, at least, that if (i) in \ref{3.3} fails then $D$ 
is
$2^\nu$-decomposable?

(b) A more general problem: find pairs of cardinals $\mu \leq \mu '$ and
$\kappa
<\kappa '$
such that $D$ is $(\kappa ,\kappa ')$-regular whenever $|\prod_D \mu |>\mu '$.
Notice that,  under
particular assumptions on cardinal
arithmetic, \ref{1.1}(xi), Proposition \ref{3.3} and Theorems \ref{2.1}(b) and \ref{2.15ex2.13}(i) 
actually furnish
examples of such pairs; more examples can be found in combination with
Theorems  \ref{2.18ex2.16} 
 \ref{3.2}, \ref{4.3}.
\end{problems}

We notice also the following corollary of Proposition \ref{3.3}
which deals with exponentiation with a larger base.

\begin{corollary}\label{3.6}
If $|\prod_{D} \mu^\lambda |> \mu^\lambda $
 then $D$ is $\kappa $-decomposable for some $\kappa $ with
$\lambda\leq\kappa\leq \mu^\lambda $.
 \end{corollary}

\begin{proof}
If $\nu$ is the first cardinal such that 
$(\mu^\lambda )^\nu > \mu^\lambda $, then 
$\nu \geq \lambda^+$. Then apply 
the last statement in Proposition \ref{3.3}. 
\end{proof}

\section{ Down from exponents (Part II)}
\label{dfe2}

\begin{prop}
\label{4.1}
For every ultrafilter $D$,
and every cardinals $\kappa $, $\nu$,
$|\prod_D \nu^\kappa |
 \leq |\prod_D \nu^{<\kappa }|^{{\rm cf}(\prod_D \langle \kappa ,<
\rangle)}$.
\end{prop}

\begin{proof}
Consider a model {\bf A} of the form
$\langle A, U, <, V, W, F, G, H, J \rangle $, where $U, V, W, J$ are unary predicates, 
\makebox{$\langle U,<\rangle = \langle \kappa,<\rangle  $}, $|J|=\nu$
$V$ is (in a one to one correspondence with) $^{\kappa }\nu$, the set
of all functions from $\kappa $ to $\nu$, and
$W$ is (in a one to one correspondence with)
$\bigcup_{\beta<\kappa } \ ^{\beta }\nu$,
 the set
of all functions from
 $\beta$ to $\nu$, for all $\beta<\kappa $.

$F$ is a function from $W$ to $\kappa $, and we require that if $W(w)$ then $w$ is
(corresponds to) a function from
$F(w)$ to $\nu$.
$G$ and $H$ are functions which represent the functions in $V,W$; namely if
$V(v)$ and $v$ is (corresponds to)
$f:\kappa \to \nu$ then $G(v,\alpha )=f(\alpha )$, for all $\alpha<\kappa $.
Similarly, if $W(w)$ and
$w$ is (corresponds to)
$g:F(w)\to \nu$ then $H(w,\alpha  )=g(\alpha  )$, for all $\alpha <F(w)$.

What matters is that (in {\bf A}) $|V|=\nu^\kappa $ and
$|W|=\nu^{<\kappa }$; we shall get the desired inequality by taking the ultrapower
of {\bf A} modulo $D$, and then computing the cardinality of the unary
predicates $V,W$ there.

Let {\bf B} be any model elementarily equivalent to {\bf A}. The following
formula holds in {\bf A},
hence in {\bf B}:

\[ 
 \forall v \forall u  (V(v) \wedge U( u)   \Rightarrow  \exists w (W(w)
\wedge F(w)= u \wedge   \forall x
( x< u  \Rightarrow  G(v, x )=H(w, x )  )    )   ),
\] 
in words, for every function $f$ in $V$ and for every  element $u$ in $U$ there
is a function $g$ in $W$ such that $g $ is $f$ restricted to the domain
$[0,u)=\{x| U(x)
\wedge x<u\} $. Thus, all ``initial segments'' of functions in $V$ can be found
in $W$; in particular, working in 
{\bf B}, if $\lambda $ is the cofinality of $\langle
U,<\rangle $ and  $u_\gamma \ (\gamma<\lambda )$ is a cofinal sequence, then any
$f$ in $V$ is determined by the functions $f_{|[0,u_\gamma  )}$
($\gamma<\lambda $).

For each $\gamma<\lambda $ there are at most $|W|$ (computed in {\bf B})
functions with domain $[0,u_\gamma )$,
whence in {\bf B} $|V|\leq|W|^\lambda $ (notice that in {\bf A} different
elements in $V$ correspond to different functions, and this can be expressed by
a
first order sentence, using $G$).

Now let ${\bf B} =\prod_D {\bf A} $. By \L o\v s Theorem,
{\bf B} is elementarily equivalent to {\bf A}, and the above argument gives the
result, recalling that in {\bf A}  $\langle U,<\rangle = \langle
\kappa,<\rangle  $.
  \end{proof}

\begin{cor}
\label{4.2}
 Suppose that $\kappa $ and $\lambda $ are infinite cardinals,  $\nu$ is a cardinal, 
 $|\prod_D \nu^\kappa |>\nu^\lambda  $, and $|\prod_D \nu^{<\kappa
}|\leq \nu^\lambda $.

Then
$ \nu^\lambda < 
|\prod_D \nu^\kappa | \leq
|\prod_D \nu^{<\kappa} |^{\cf(\prod_D \langle\kappa,< \rangle)} \leq
\nu^{\lambda \cdot \cf(\prod_D \langle\kappa,< \rangle)}$,
hence
${\rm cf}(\prod_D \langle\kappa,< \rangle)>\lambda   $.

If in addition $D$ is $(\kappa^+,\lambda )$-regular,
then $D$ is $(\kappa, \lambda )$-regular (by Theorem  \ref{2.13ex2.11}(i)).
\end{cor}

The following theorem improves Theorem \ref{3.2}.

\begin{thm}
\label{4.3}
 (a) If $|\prod_D \nu^{\kappa ^{+n}}| > \nu^{\kappa ^{+n}}$
(in particular, if
$D$ is $(\nu^{\kappa^{+n}},\nu^{\kappa^{+n}})$-regular)
then $D$ is
$\mu $-decomposable for some $\mu $ with
$\kappa \leq\mu \leq \nu^\kappa $.
If in addition $\nu^\kappa=\kappa^{+p}$ for some $p<\omega $ then $D$ is
$(\kappa,\kappa )$-regular.

(a$'$) If $m\geq 1$ and
$|\prod_D \beth_m({\kappa ^{+n}})| > \beth_m{(\kappa ^{+n}})$ 
(in particular, if
$D$ is $(\beth_m{(\kappa ^{+n}}),\beth_m{(\kappa ^{+n}}))$-regular)
 then $D$ is
$\mu $-decomposable for some $\mu $ with
$\kappa \leq\mu \leq2^\kappa $.
If in addition $2^\kappa=\kappa^{+p}$ for some $p<\omega $ then $D$ is
$(\kappa,\kappa )$-regular.

(b) Suppose that $\kappa $ is a strong limit cardinal and that
$|\prod_D \beth_m(\kappa^{+n} )|>\beth_m(\kappa^{+n} )$
(in particular, this holds when
$D$ is $(\beth_m(\kappa^{+n} ), \beth_m(\kappa^{+n} ))$-regular).
Then either $D$ is $({\rm cf} \kappa, {\rm cf} \kappa )$-regular or there are
arbitrarily large $\mu<\kappa $ for which $D$ is $\mu $-decomposable.
\end{thm}

\begin{proof}
 (a) Let $q$ be the smallest integer such that $|\prod_D \nu^{\kappa ^{+q}}| >
\nu^{\kappa ^{+q}}$.

If $q=0$ the conclusion follows from Corollary  \ref{3.6}.

Otherwise, $|\prod_D \nu^{\kappa ^{+q-1}}| = \nu^{\kappa ^{+q-1}}$
and $\nu^{\kappa^{+q} }<|\prod_D \nu^{\kappa^{+q}}|$, hence
${\rm cf}(\prod_D \langle\kappa^{+q} ,< \rangle)>\kappa^{+q}    $ by Corollary
\ref{4.2}, by taking $\kappa=\lambda $ there to be $\kappa^{+q}$.

Hence $D$ is $(\kappa^{+q},\kappa^{+q})$-regular by Remark
\ref{1.5}(a), and
$(\kappa^+,\kappa^+)$-regular by an iteration of Theorem  
\ref{2.1}(b), since $q>0$.
Hence $D$ is  $\kappa^+$-decomposable by  \ref{1.1}(viii),
and the conclusion holds with $\mu=\kappa^+$.
.

If $\nu^\kappa=\kappa^{+p}$ for some $p<\omega$, then
the $\mu $ given by the preceding statement satisfies
$\kappa\leq\mu\leq \nu^\kappa=\kappa^{+p}$, hence $\mu=\kappa^{+p'}$ for some
$p'\leq p$. Since $D$ is $\mu $-decomposable, $D$ is
$(\kappa^{+p'},\kappa^{+p'})$-regular by  \ref{1.1}(vii). If $p'=0$
this is what we
want; otherwise it is enough to use Theorem  \ref{2.1}(b) a 
sufficient number of times.

A remark: the reader might observe that in the course of the proof we have
obtained  $(\kappa^+,\kappa^+)$-regularity, which implies $(\kappa,\kappa)$-regularity,
so that the hypothesis $\nu^\kappa=\kappa^{+p}$ might appear to be unnecessary.
However, we get $(\kappa^+,\kappa^+)$-regularity only in the case
$q>0$, while we do not necessarily have it in the case $q=0$, which uses Corollary 
\ref{3.6}.
In fact, the hypothesis
$\nu^\kappa=\kappa^{+p}$ is necessary: as in the example
at the beginning of Section \ref{dfe1}, take $\mu $ measurable, and 
add $\mu $ Cohen
reals. Then $ 2^\lambda=\mu $, for every  $\lambda<\mu $: take $\lambda $
regular with  $\mu>\lambda>\omega $. As in the beginning of Section 
\ref{dfe1}, we have
an ultrafilter $D$ $\mu $-decomposable but not $\lambda $-decomposable; $D$ is
not $(\lambda,\lambda )$-regular, and is $(\mu,\mu )$-regular  (by
\ref{1.1}(vii)),
hence $|\prod_D \mu | > \mu $ (by the first lines in the proof of 
\ref{3.2}).
$\mu=2^{\lambda ^{+n}}$, so that
$|\prod_D 2^{\lambda  ^{+n}}| > 2^{\lambda  ^{+n}}$, but $D$ is not
$(\lambda,\lambda )$-regular, hence the conclusion in the third statement of
\ref{4.3}(a) fails.

(a$'$) From the hypothesis we get
$|\prod_D 2^{\kappa^{+n}}|>2^{\kappa^{+n}}$
 by iterating the fact (already mentioned in the proof 
of \ref{3.2})
that
$|\prod_D 2^{2^\lambda } |>2^{2^\lambda }$
implies $|\prod_D 2^{\lambda } |>2^{\lambda }$.
The conclusion then follows from (a),
with $\nu=2$ and $n=0$.

(b) follows from Theorem \ref{3.2}(b) (case $n=0$) and iterating the fact that 
both $|\prod_D
\lambda^+| >\lambda^+ $ and $|\prod_D 2^\lambda| >2^\lambda $ imply $|\prod_D
\lambda| >\lambda $.
 \end{proof}

Theorems \ref{3.2} and \ref{4.3} can be improved in many ways. For 
example, 
Theorem \ref{4.3}
holds when in place of $\beth_m({\kappa ^{+n}})$
we consider any iteration (in any order)
of any finite number of the $\beth$  and of the successor functions (with at
least one occurrence of $\beth$ in (a$'$)). This is because both $|\prod_D
\lambda^+| >\lambda^+ $ and $|\prod_D 2^\lambda| >2^\lambda $ imply $|\prod_D
\lambda| >\lambda $.

More generally, one can use the formula $|\prod_D \lambda^\mu | \leq |\prod_D
\lambda|^{|\prod_D \mu |}  $, so that, for example, $(\lambda^\mu,\lambda^\mu
)$-regularity implies
$|\prod_D \lambda^\mu |> \lambda ^\mu$, hence
either $|\prod_D \lambda| >\lambda $
or $|\prod_D \mu | >\mu  $: then we can apply the statements of 
\ref{3.2} or \ref{4.3} (or
the methods of proof). Other results can be obtained from Propositions 
\ref{3.3}, \ref{4.1}  and
theorems about cardinalities of ultrapowers (e.g. \cite{Kei}, or 
\cite[Theorem 0.25]{Lp1}). One can also mix in the results of Section \ref{fstp}. We 
leave details to the
reader, since statements become quite involved, and since we do not know
whether Proposition \ref{3.3} is the best possible result.

Notice also that if Problem \ref{3.5}(a) has an affirmative answer, 
then we can
improve the conclusion in Theorem \ref{4.3}(a$'$) to ``{\it for some 
$\mu$ with
$\kappa<\mu\leq2^\kappa $}''.

\begin{problems}
\label{4.4}
 (a) \cite{Lp7}
Suppose that $2^\kappa <\mu <2^{\kappa ^+}$ and $D$ is $\mu $-decomposable.
Is $D$ necessarily $\lambda $-decomposable for some $\lambda $ with $\kappa
\leq\lambda \leq2^\kappa $?
(maybe the assumption $\mu$ regular is necessary)

(b) \cite{Lp6} Is it consistent to have an
$\omega_1$-complete   ultrafilter uniform on some $\mu $ with $\mu^\omega>\mu $?

\end{problems}

Theorems \ref{3.2} and \ref{4.3} have the following consequences for 
topological spaces and
for extensions of first-order logic (see \cite{Ca, Lp2};  
\cite{Ma, Lp1}).

Recall that a topological space is $[\kappa,\mu  ]$-\emph{compact} if and only if every open cover
by $\mu $ many sets has a subcover by $<\kappa $ many sets. A family $\mathcal F $
of topological spaces is {\it productively} $[\kappa,\mu ]$-\emph{compact} if and only if every
product of members of $\mathcal F$ is
$[\kappa,\mu ]$-compact.

In \cite[Theorem 3]{Lp2}  we have shown, improving results from \cite{Ca}:

\begin{thm}
\label{4.5}
 Let $\lambda,\mu $ be infinite cardinals, and $(\kappa_i)_{i\in
I}$ be a family of infinite cardinals. Then the following are equivalent:

(i) Every productively $[\lambda,\mu ]$-compact topological space is
$[\kappa_i,\kappa_i]$-compact for some $i\in I$.

(ii) Every productively $[\lambda,\mu ]$-compact family of topological spaces is
productively $[\kappa_i,\kappa_i]$-compact for some $i\in I$.

(iii) Every $(\lambda,\mu )$-regular ultrafilter is
$(\kappa_i,\kappa_i)$-regular for some $i \in I$.
 \end{thm}

\begin{cor}
\label{4.6}
Any productively $[2^{2^\kappa  },2^{2^\kappa  } ]$-compact family of
topological spaces is productively $[\mu ,\mu  ]$-compact for some $\mu  $ with
$\kappa \leq\mu \leq 2^\kappa  $.

More generally, if $m\geq 1$ then
any productively
$[\beth_m{(\kappa ^{+n}}),\beth_m{(\kappa ^{+n}})]$-compact family of
topological spaces is productively $[\mu ,\mu  ]$-compact for some $\mu  $ with
$\kappa \leq\mu\leq 2^\kappa  $.

If $\kappa $ is a strong limit cardinal, and $U'(\kappa )$ holds, then
any productively $[2^\kappa ,2^\kappa  ]$-compact
(even, every productively
$[\beth_m(\kappa^{+n} ), \beth_m(\kappa^{+n} )]$-compact)
family of topological spaces is productively $[\kappa,\kappa ]$-compact.
\end{cor}

\begin{proof} By \ref{1.1}(vii), Theorems \ref{3.2}, \ref{4.3} and 
\ref{4.5}(iii)$\Rightarrow$(ii) (actually, this
implication in Theorem \ref{4.5} is a corollary of results in 
\cite{Ca}). 
\end{proof}

In what follows, by a {\it logic}, we mean a {\it regular  logic}  in the sense
of
\cite{Eb}. Typical examples of regular logics are extensions
of
first-order logic obtained by adding new quantifiers (e.g., cardinality quantifiers,
asserting ``there are at least $\omega_\alpha $ $x$'s such that \dots"), or by
allowing infinitary conjunctions and disjunctions, and possibly simultaneous
quantification over infinitely many variables (infinitary logics).

Roughly, a logic is regular if and only if it shares all the good properties of
the above typical examples. The reader is invited to look at
\cite{BF} for more
information about logics.

A logic $L$ is $[\lambda,\mu ]$-{\it compact} if and only if for every pair of
sets $\Gamma$ and $\Sigma$ of sentences of $L$, with $|\Sigma|\leq\lambda $, the
following holds: if $\Gamma \cup \Sigma'$ has a model for every $\Sigma'
\subseteq \Sigma$ with $|\Sigma'|<\mu $, then $\Gamma \cup \Sigma$ has a
model.

There is an older (and weaker) notion, called $(\lambda,\mu)$-{\it compactness},
which corresponds to the above definition in the particular case when
$\Gamma=\emptyset$. In a series of papers, J. Makowsky and S. Shelah showed that
the new stronger $[\lambda,\mu ]$-compactness is much better behaved (see \cite{Ma}
for a survey). In particular, Makowsky and Shelah defined what it means for an
ultrafilter to be {\it related} to a logic, and showed that a logic $L$ is
$[\lambda,\mu ]$-compact if and only if there is a $(\mu,\lambda )$-regular
ultrafilter related to $L$. An immediate consequence is:

\begin{cor}
\label{4.7}
Let $\lambda,\mu $ be cardinals, and $(\kappa_i)_{i\in
I}$ be a family of cardinals. If

(i) for every  $(\mu,\lambda )$-regular ultrafilter $D$ there
is $i\in I$
such that $D$ is $(\kappa_i,\kappa_i )$-regular,

then

(ii) for every  $[\lambda ,\mu]$-compact logic $L$ there
is $i\in I$
such that $L$ is $[\kappa_i,\kappa_i]$-compact.
 \end{cor}

The analysis of connections between regularity of ultrafilters and compactness
of logics
can be further elaborated: see \cite[Section 0]{Lp1}, and some references there.
Indeed, we recently proved that conditions (i) and (ii)
in Corollary \ref{4.7} are actually equivalent.
The case  when
all the $\kappa_i $'s are assumed to be regular cardinals
is easier and had been proved in \cite[Theorem 4.1]{Lp1}
(in an equivalent form, by Proposition \ref{7.6}). 
The general case when
the $\kappa_i $'s are allowed to be singular cardinals is proved 
in \cite[Theorem 10]{lpndj}, is
connected with Problem
 \ref{2.8}, and has led to the formulation of the principle 
$U'(\lambda )$ of
Definition  \ref{1.7}. See \cite[Section 4]{Lp1}; see also the end of
Section \ref{lfaa} in the present paper.

Anyway, from Corollary \ref{4.7}, \ref{1.1}(vii) and Theorems  \ref{3.2} and 
\ref{4.3}, we 
get:

\begin{cor}
\label{4.8}
 Every  $[2^{2^\kappa },2^{2^\kappa  } ]$-compact
logic is $[\mu ,\mu  ]$-compact for some $\mu  $ with $\kappa \leq\mu \leq
2^\kappa  $.

More generally, if $m\geq 1$ then
every
$[\beth_m{(\kappa ^{+n}}), \beth_m{(\kappa ^{+n}})]$-compact
logic is $[\mu ,\mu  ]$-compact for some $\mu  $ with $\kappa \leq\mu  \leq
2^\kappa  $.

If $\kappa $ is a strong limit cardinal, and $U'(\kappa )$ holds, then
every  $[2^\kappa ,2^\kappa  ]$-compact logic
(even, every  $[\beth_m(\kappa^{+n} ),\brfr \beth_m(\kappa^{+n} )]$-compact logic) is
$[\kappa,\kappa ]$-compact.
 \end{cor}

Of course, we could use Theorem \ref{4.3}(b) in order to improve the 
last part of
Corollary \ref{4.6} to: {\it
if $\kappa $ is a strong limit cardinal then
any productively $[\beth_m(\kappa^{+n} ), \beth_m(\kappa^{+n} )]$-compact
family $\mathcal F$ of topological spaces either is productively $[{\rm
cf}\kappa,{\rm cf}\kappa ]$-compact, or there are arbitrarily large $\mu <\kappa
$ for which $\mathcal F$ is productively $[\mu ,\mu  ]$-compact.}

A similar remark holds for logics and Corollary \ref{4.8}.

\begin{remark}
\label{4.9}
 The remarks at the beginning of Section \ref{dfe1}, together with
\cite[Theorem 2.5(iv)$\Leftrightarrow$(vi)]{Lp3}, show that it is possible to have a
$(2^\lambda,2^\lambda )$-compact not $(\lambda,\lambda )$-compact logic, thus
solving negatively \cite[Problem III]{Lp4}. Actually, what we get is a
$[2^{\omega_1},2^{\omega_1}]$-compact not $(\omega_1,\omega_1 )$-compact logic.
\end{remark}

\section{ Above limits.}
\label{al}

If $\lambda$ is a limit cardinal and an ultrafilter enjoys some form of
regularity for arbitrarily large  cardinals below $\lambda$ then in some cases
regularity can be lifted up to $\lambda$. In this section we provide
some examples and counterexamples. Again, many problems seem still open.

For $ \lambda \geq \lambda ' \geq \mu $,
let $\COV(\lambda, \lambda ',\mu )$ denote the minimal cardinality of a family of
subsets of $\lambda$, each of cardinality  $< \lambda ' $, such that every subset of
$\lambda$ of cardinality $<\mu $ is contained in at least one set of the family.
In particular, $\COV(\lambda,\mu,\mu) $ is the cofinality of (the partial order)
$S_\mu(\lambda )$. Notice that $\COV(\lambda, \lambda ',\mu )\leq |S_\mu( \lambda )|= \lambda ^{< \mu}$.

\begin{thm}
\label{5.1}
 Suppose that $\lambda$ is a limit cardinal, and that there are  arbitrarily
large $\nu<\lambda $ for which $D$ is $\nu $-decomposable. Then:

(a) $D$ is $\kappa$-decomposable for some $\kappa$ with
$\lambda\leq\kappa\leq\lambda^{{\rm cf}\lambda }$.

(b) In addition, suppose that $\lambda$ is singular
 and $D$ is $(\mu,{\rm cf}\lambda )$-regular. Then
$D$ is $\kappa$-decomposable for some $\kappa$ with
$\lambda\leq\kappa\leq\lambda^{<\mu }$. More generally, for every $\lambda'$ with
$\mu\leq\lambda'<\lambda $ there exists
$\kappa$ such that $D$ is $\kappa$-decomposable and
$\lambda\leq\kappa\leq \COV(\lambda,\lambda',\mu )$.
\end{thm}

\begin{proof} Let $(\nu_\alpha )_{\alpha\in \cf\lambda }$ be a sequence cofinal in
$\lambda$, such that $D$ is $\nu_\alpha $-decomposable, for $\alpha\in cf\lambda
$; and for $\alpha\in {\rm cf}\lambda  $ let $\Pi_\alpha $ be a
$\nu_\alpha $-decomposition.

In order to prove (a) it is enough to consider the common refinement of all the
$\Pi_\alpha $'s (compare with the proofs of \cite[Proposition 2]{Pr}  and
\cite[Theorem 2]{AJ}).

The last statement in (b) implies the statement that precedes it, 
since, trivially, $\COV(\lambda, \lambda ',\mu ) \leq \lambda ^{< \mu} $.
Anyway, we shall give a proof for
the first statement, too, because it is rather simpler.
First, notice that if $ \mu >\cf \lambda $ then (b) follow from (a), hence without loss of generality
we can suppose $ \mu \leq \cf \lambda $.

Let the $\Pi_\alpha $'s be as above, and for $X\subseteq {\rm cf}\lambda $ with
$|X|<\mu $
let $\Pi_X$ be the common refinement of $\{ \Pi_\alpha|\alpha\in X\} $. $\Pi_X$
has $\leq\lambda^{<\mu }$ classes. Let $D$ be over $I$, and let $f:I\to
S_\mu({\rm cf} \lambda )$ witness
the $(\mu,{\rm cf}\lambda )$-regularity of $D$, as given by Form II.
Define $\Pi$ on $I$ by $i\sim j$ if and only if $f(i)=f(j)$ and $(i,j)\in \Pi_{f(i)}$. $\Pi$ has
$\leq\lambda^{<\mu } \cdot ({\rm cf}\lambda )^{<\mu }=\lambda^{<\mu }$ classes.

On the other side, for every $\alpha$, $\Pi$ has more than $\nu_\alpha $ classes
(modulo $D$): for every $\alpha\in {\rm cf}\lambda $, $I_\alpha=\{i\in I|\alpha\in
f(i)\}\in D $, and $\Pi_{|I_\alpha }$ refines
$\Pi_{\alpha \ |I_\alpha }$ (which has $\nu_\alpha $ classes modulo $D$).
Thus, $\Pi$ has $ \kappa $ classes modulo $D$  for some $ \kappa $
with $ \lambda \leq \kappa \leq \lambda ^{< \kappa } $,
and this proves the $ \kappa $-decomposability of $D$.

In order to prove the last statement in (b), let $D$ be over $I$, let $f:I\to
S_\mu({\rm cf} \lambda )$ witness
the $(\mu,{\rm cf}\lambda )$-regularity of $D$
(Form II), and for $\alpha\in {\rm
cf}\lambda $ let
$f_\alpha :I\to \nu_\alpha $ be a $\nu_\alpha $-decomposition.

Let $H$ be a family of subsets of $\lambda$ as given by
$\COV(\lambda,\lambda',\mu )$: thus, for every $i\in I$ there is $h(i)\in H$ such
that $\{ f_\alpha(i)|\alpha\in f(i)\} \subseteq h(i) $, and $|h(i)|<\lambda'$.
Choose $f$, $H$, the $f_\alpha $'s and the
$h(i)$'s
in such a way that the cardinality of $K=\{h(i)|i\in I  \} $ is minimal; this
choice makes the
function $h:I\to K$ a $|K|$-decomposition of $D$: if $h^{-1}(K')\in D$ for some
$K'\subseteq K$ with $|K'|<|K|$ then we could change the values of
$f(i)$, $f_\alpha(i) $ and of
$h(i)$ for $i\not\in h^{-1}(K')$, thus making $K'$ work in place of $K$,
contradicting the minimal cardinality of $K$ (the argument is identical
with the one in the proof of \cite[Lemma 4.7]{Lp1}).

Since $|K|\leq|H|= \COV(\lambda,\lambda',\mu )$,
it remains to show that $|K|\geq\lambda $. 
By $(\mu, \cf \lambda )$-regularity, for every $\alpha\in {\rm
cf}\lambda $, $I_\alpha=\{i \in I|\alpha\in f(i) \} \in D$, whence
$\nu_\alpha=|f_\alpha(I_\alpha )|$, since $f_\alpha $ is a
$\nu_\alpha $-decomposition.
Since $h(i)$ has been chosen in such a way that $f_\alpha(i)\in
h(i)$ when $\alpha\in f(i)$, that is, when $i \in I_\alpha $, we have that $f_\alpha(I_\alpha
)=\{f_\alpha(i)|i \in I_\alpha \}\subseteq \bigcup_{i\in I_\alpha } h(i) $, hence
$|f_\alpha(I_\alpha )|
\leq \lambda'\cdot|h(I_\alpha)|$, since $|h(i)|<\lambda'$, for all $i$.

Putting the above inequalities together, we get:
$\nu_\alpha=|f_\alpha(I_\alpha )|\leq \lambda'\cdot|h(I_\alpha)
|\leq\lambda'\cdot|h(I)|$, for all $\alpha\in{\rm cf}\lambda $. 
Hence $\nu_\alpha \leq |h(I)|=|K|$, for all $\alpha $'s
such that  $\nu_\alpha>\lambda'$. 
Since $ \lambda'< \lambda $, and the $ \nu_\alpha $'s are cofinal
in $ \lambda $, we get
$|K|\geq\sup_{\alpha\in {\rm cf}\lambda }\nu_\alpha =\lambda$.
\end{proof}

Notice that Theorem \ref{5.1}(b) is a common generalization of 
\cite[Lemma 4.9]{DJK}, \cite[Proposition 1]{Pr}  and 
\cite[Lemma 4.7, Remark 4.8]{Lp1}.

We do not know whether Theorem \ref{5.1} can be improved; actually, 
we do not even
know whether:

\begin{problem}
\label{5.2}
 Is the following true? If
$\lambda$ is a limit cardinal, and there are  arbitrarily
large $\nu<\lambda $ for which $D$ is $\nu $-decomposable, then $D$ is
either $\lambda$-decomposable or $\lambda^+$-decomposable.
\end{problem}

We believe that Problem \ref{5.2} has a negative answer, in general 
(and we
believe that the case when $\lambda$ is regular is much easier to falsify).
Be that as it may,
\ref{5.2} is the best we can expect: if $\kappa $ is $\kappa 
^{+\omega+1}$-compact
there is a $(\kappa ,\kappa ^{+\omega+1})$-regular ultrafilter 
which is $\kappa$-complete, and hence not
$(\omega,\omega )$-regular
and not $\kappa^{+\omega }$-decomposable,
by \ref{1.1}(vii), 
but $\kappa^{+n} $-decomposable for
all $n< \omega $, and $\kappa^{+\omega+1}$-decomposable, by \ref{1.1}(xii). 
On the other hand, for every $\lambda$
there is a $(\omega,\lambda )$-regular ultrafilter uniform over $\lambda$
(e. g. \cite{CK}),
 hence
not $\lambda^+$-decomposable by \ref{1.1}(iii), but $\kappa $-decomposable for all
$\kappa\leq\lambda $ by Remark \ref{1.5}(b).

Of course, an affirmative solution of \ref{5.2} would show that 
$U'(\lambda )$
(see Definition   \ref{1.7}) holds
for every limit cardinal $\lambda$. Indeed, \ref{5.2} and $U'(\lambda )$ have 
the same
hypothesis, and the conclusion in $U'(\lambda )$ follows from the conclusion in
\ref{5.2}, since, by  \ref{1.1}(vii), every $\lambda$-decomposable
ultrafilter is
$(\lambda,\lambda )$-regular, and,  by  \ref{1.1}(vii) and Theorem
 \ref{2.1}(b), every
$\lambda^+$-decomposable  ultrafilter is $(\lambda,\lambda )$-regular.

Anyway, Theorem \ref{5.1} shows that Problem \ref{5.2} has an 
affirmative 
answer for many
singular cardinals.

\begin{cor}
\label{5.3}
 Suppose that $\lambda$ is singular, and either

(i)
$\COV(\lambda,\lambda,({\rm cf}\lambda)^+)=\lambda^{+n}$, for some
$n<\omega$; or

(i$'$) $ \lambda ^{\cf \lambda }= \lambda ^{+n}$, for some
$n<\omega$; or 

(ii) ${\rm cf}\lambda=\omega $ and there is no measurable cardinal; or

(iii)
 ${\rm cf}\lambda=\omega $ and  $\COV(\lambda,\lambda,\omega_1)<\lambda^{+\mu_0}$,
where $\mu_0$ is the smallest measurable cardinal; or

(iv)
 $\lambda=\omega_{\alpha+\delta}$, {\rm cf}$\delta=\omega$,
 $\delta<\omega_\alpha$, and
  $\delta<\mu_0$,
where $\mu_0$ is the smallest measurable cardinal.

If $D$ is an ultrafilter, and  there are  arbitrarily
large $\nu<\lambda $ for which $D$ is $\nu $-decomposable, then $D$ is
either $\lambda$-decomposable or $\lambda^+$-decomposable.
\end{cor}

\begin{proof} (i) It is not difficult to show that
there is a $\lambda'<\lambda$ such that
$\COV(\lambda,\lambda,({\rm cf}\lambda)^+)=$
$\COV(\lambda,\lambda',({\rm cf}\lambda)^+)$: see
\cite[Analytical Guide, 6.2, or II, Observation 5.3(10)]{Sh};
the notation for $\COV(\lambda,\mu,\nu)$ is {\rm cov}$(\lambda,\mu,\nu,2)$
in \cite{Sh}: see \cite[II, Definition 5.1]{Sh}.

By Theorem \ref{5.1}(b) and \ref{1.1}(xiii), 
$D$ is $\kappa $-decomposable for some $\kappa $ with
$\lambda\leq\kappa\leq\COV(\lambda,\lambda',({\rm
cf}\lambda)^+)=
\COV(\lambda,\lambda,({\rm
cf}\lambda)^+)=
\lambda^{+n}$. If $D$ is $\lambda $-decomposable we are OK,
otherwise, $D$ is $\lambda^{+m}$-decomposable for some $m>0$. Since successors
are regular cardinals, then \ref{1.1}(xi) and an iteration of Theorem
 \ref{2.1}(b)
imply that $D$ is $\lambda^+$-decomposable. Essentially, (i) had been already
noticed in \cite[Remark 4.8]{Lp1}  (under the assumption
$\COV(\lambda,\lambda',(\cf \lambda)^+)=\lambda^{+n}$).

(i$'$) follows from (i), since $\COV(\lambda,\lambda,({\rm
cf}\lambda)^+)\leq
\lambda^{\cf \lambda }$.

In order to prove (ii) and (iii), let  ${\rm cf}\lambda=\omega $. If $D$ is
 $(\omega,\omega )$-regular, then $D$ is $\lambda $-decomposable because of
Theorem \ref{5.1}(b), since $\lambda ^{< \omega }= \lambda $.

Otherwise, $D$ is $\omega_1$-complete (so (ii) cannot hold), hence
$\mu_0$-complete. Let $\kappa $ be
the smallest cardinal $\geq\lambda $ such that $D$ is $\kappa $-decomposable. By
Theorem \ref{5.1}, at least one such $\kappa $ exists, and, by the 
last statement in
\ref{5.1}(b) and \ref{1.1}(xiii), 
$\kappa\leq \COV (\lambda,\lambda',\omega_1)<\lambda^{+\mu_0}$
for some $\lambda'<\lambda$
(as at the beginning of the proof of (i), we can suppose that
$\COV(\lambda,\lambda,\omega_1)=\COV(\lambda,\lambda',\omega_1)$,
for some $\lambda'<\lambda$).

By  Theorem  \ref{2.1}(b),  \ref{1.1}(xi) and the minimality of 
$\kappa $,
$\kappa $ cannot
be the successor of a regular cardinal.

If $\kappa=\lambda $ the result is proved.
If $\kappa> \lambda $ then $\kappa$ cannot be limit:
all limit cardinals between
$\lambda $ and $\lambda^{+\mu_0}$ are singular
cardinals of cofinality $<\mu_0$; if
$\kappa $ is singular  then $D$ is not
$({\rm cf}\kappa,{\rm cf}\kappa )$-regular,
since ${\rm cf}\kappa <\mu_0$ and $D$ is $\mu_0$-complete.
By  \ref{1.1}(vii) this 
implies
that $D$ is not $\kappa $-decomposable.

The remaining case is when $\kappa $ is a successor of a singular cardinal,
say $\kappa=\nu^+$.
As above, $D$ is not  $({\rm cf}\nu,{\rm cf}\nu )$-regular,
and, by Theorem  \ref{2.1} and  \ref{1.1}(vii), $D$ is
$(\kappa',\kappa )$-regular for some $\kappa'<\kappa $; by
 \ref{1.1}(xii), this contradicts the minimality of $\kappa $, unless
$\kappa=\lambda^+ $.

(iv) Since
 $\lambda=\omega_{\alpha+\delta}$ and
 $\delta<\omega_\alpha$, then, by \cite[IX, Claim 3.7(1) and Theorem 2.2]{Sh},
$\COV(\lambda,\lambda,({\rm cf}\lambda)^+)=$ pp$(\lambda)\leq
\omega_{\alpha+|\delta|^{+4}}$.

 Since
{\rm cf}$\lambda=\omega$, we get from above
$\COV(\lambda,\lambda,\omega_1)\leq
\omega_{\alpha+|\delta|^{+4}}=\lambda^{+|\delta|^{+4}}$,
since $\lambda=\omega_{\alpha+\delta}$ and
$\delta+|\delta| ^{+4}=|\delta| ^{+4} $.
Since $\delta<\mu_0$, and $\mu_0$ is a measurable cardinal,
hence inaccessible,
$|\delta|^{+4}< \mu_0$
and the hypotheses of (iii) apply.
\end{proof}

In (iv) above we have used Shelah's tight bounds on
$\COV(\lambda,\lambda,({\rm cf}\lambda)^+)$
in order to show that Problem \ref{5.2} has
an affirmative answer for a large class of singular cardinals of countable cofinality.
Can the argument be extended in order to cover {\it all}
singular cardinals of countable cofinality?
Can the argument be adapted in some way to singular cardinals of uncountable cofinality?

Having dealt with decomposability, we now turn
to the problem of lifting up regularity.

\begin{prop}
\label{5.4}
 Suppose that $\mu$ is singular, $\lambda<\mu $
and $D$ is $(\lambda,\mu')$-regular
for all $\mu'<\mu $
(equivalently, by \ref{1.1}(i), for a sequence of $\mu'$'s cofinal in $\mu$). Then $D$ is $(\lambda,\mu )$-regular, provided at least one
of the following holds:

(a) $\lambda $ is regular; or

(b) $D$ is $({\rm cf}\lambda,{\rm cf}\mu)$-regular
(in particular, this holds when  ${\rm cf}\lambda>{\rm cf}\mu $, by \ref{1.1}(xiii)); or

(c) the cofinality of
$\prod_D \langle {\rm cf}\lambda,<\rangle$ is $\not ={\rm cf}\mu $; or

(d) ${\rm cf}\lambda\not={\rm cf}\mu $ and $D$ is not $({\rm cf}\lambda,{\rm
cf}\lambda )$-regular.
\end{prop}

\begin{proof}  The easy proof of (a) is given in \cite[Lemma 1.1]{Do}, slightly generalizing
\cite[Proposition 1.1]{BK}.

We now prove (c): by  \ref{1.1}(i) $D$  is $(\lambda^+,\mu')$-regular
for all $\mu'<\mu $, hence we can apply (a),
with $ \lambda ^+$ in place of $ \lambda $,
obtaining
$(\lambda^+,\mu)$-regularity.
Now,
$(\lambda,\mu)$-regularity
 follows from Theorem  \ref{2.13ex2.11}(i),
since the cofinality of 
$\prod_D \langle {\rm cf}\lambda,<\rangle$
equals the cofinality of
$\prod_D \langle \lambda,<\rangle$.

(b) is a consequence of (c) because of Proposition  \ref{1.3}. 
Actually, the arguments in the proof of
\cite[Lemma 1.1]{Do}  give a proof of (b).

Notice that (b) is trivial when ${\rm cf}\lambda>{\rm cf}\mu $: let $D$ be over
$I$, let $\mu_\alpha $ $(\alpha\in {\rm cf}\mu )$ be cofinal in $\mu $, and
for $\alpha\in {\rm cf}\mu $ let $f_\alpha:I\to S_\lambda(\mu_\alpha )$ witness
the $(\lambda,\mu_\alpha  )$-regularity (Form II) of $D$. Then $f$ defined by
$f(i)=\bigcup_{\alpha\in{\rm cf}\mu } f_\alpha(i)$ witnesses the
$(\lambda,\mu )$-regularity of $D$.

As for (d), we have, as in the proof of (c), that $D$ is $(\lambda^+,\mu)$-regular, and the
conclusion follows from Theorem  \ref{2.2}. Alternatively, use (c) and
\ref{1.5}(a) with ${\rm
cf}\lambda $ in place of $\lambda $.
\end{proof}

We do not know whether
 Proposition \ref{5.4} can be proved without any of the assumptions 
(a)-(d).

Of course, a positive answer to Conjecture  \ref{2.16ex2.14} would make the 
assumptions
unnecessary. This is proved as follows: because of \ref{5.4}(a), the only case to be discussed 
is $\lambda $
singular, and,
as in the proof of \ref{5.4}(c),
we get $(\lambda^+,\mu )$-regularity from the hypotheses of 
\ref{5.4}; then Conjecture
 \ref{2.16ex2.14}, if true, would imply $(\lambda,\mu )$-regularity.

See also the remark after Proposition \ref{8.2}.

\begin{remark}
\label{5.5}
 On the contrary, Proposition \ref{5.4} does not generalize to the 
case when $\mu$ is
a regular
limit cardinal. As pointed out in \cite[p. 67]{Ket1}, if $\kappa$ is huge then there exist an
inaccessible cardinal $\mu>\kappa $ and a $\kappa$-complete ultrafilter which is
$(\kappa^+,\mu )$-regular, not $(\kappa,\mu )$-regular, but
$(\kappa,\mu')$-regular
 for all $\mu'<\mu $
(for the last statement use Theorem  \ref{2.13ex2.11}(i), or 
\cite[Theorem 1.3]{BK}).
The example shows that, even for fairly large $\alpha$, it is possible to have a
$(\kappa,\lambda^{+\alpha })$-regular  and $(\lambda,\mu )$-regular
ultrafilter which is not $(\kappa,\mu )$-regular.

As in \cite[Theorem 1.10]{Ket1}, we can collapse $\kappa $ to a {\it regular}
cardinal, still having a $(\kappa^+ ,\mu)$-regular not $(\kappa ,\mu)$-regular ultrafilter.
\end{remark}

\begin{problem}
\label{5.6}
We do not know whether, starting from the above example, it
is possible to
make $\kappa $ {\it singular} by forcing, still having a
$(\kappa ,\mu ')$-regular (for all $\mu '<\mu $),
$(\kappa^+,\mu )$-regular
not
$(\kappa ,\mu )$-regular ultrafilter in the extension. Of course, this would
falsify Conjecture  \ref{2.16ex2.14}.
\end{problem}

Meanwhile, we have proved that
Problem \ref{5.2} has an affirmative answer in the particular case when $ \lambda $
is a singular cardinal, and provided that $D$ is $( \nu, \nu )$-regular
for \emph{all} sufficiently large regular cardinals $\nu < \lambda $.
The proof of the following two results can be found in \cite{lpndj}, along with
some generalizations and connections with Shelah's pcf theory.

  \begin{thm} \label{abst}
Suppose that $ \lambda $ is a singular cardinal,
$ \lambda ' < \lambda $, and the ultrafilter $D$ is
$\nu$-decomposable for all regular cardinals $ \nu $
with $ \lambda ' < \nu < \lambda $.  Then $D$ is either
$ \lambda $-decomposable or $ \lambda ^+$-decomposable
(hence $(\lambda,\lambda )$-regular, by \ref{1.1}(vii)(xi) and Theorem 
\ref{2.1}(b)).
\end{thm} 

\begin{corollary} \label{limit} 
If $ \lambda $ is an infinite cardinal, then 
an ultrafilter is $(\lambda ,\lambda )$-regular
if and only if it is either $ \cf \lambda $-decomposable
or $\lambda^+ $-decomposable. 
\end{corollary}

Corollary \ref{limit} is stated in \cite{lpndj} under the assumption that
$ \lambda $ is a singular cardinal. The case when 
$ \lambda $ is regular (that is $\cf \lambda =\lambda $) is immediate from   \ref{1.1}(xi) and Theorem \ref{2.1}(b).

\section{Least functions and applications.}
\label{lfaa}

The notion of a least function arises directly from the theory of
$\omega_1$-complete ultrafilters, hence from large cardinals. However, there are
consequences holding in ZFC; in particular, we shall find partial answers to
many conjectures, by a slight extension of important and well-known results of
A. Kanamori and J. Ketonen.
Needless to say, our arguments lead to new open problems (we expect that, at
this point, the reader has become  accustomed to this phenomenon).

The next definition gives us the possibility of extending some results from
\cite{Ka1}, and to provide partial answers to Conjectures  \ref{2.16ex2.14} and 
 \ref{2.12ex2.10}.

\begin{definition}
\label{6.1}
 If $D$ is an ultrafilter over some set $I$,
and $\kappa $ is a cardinal, we say that a function $f:I \to \kappa $ is a {\it
$\kappa $-least function for} $D$ if and only if
$\{i\in I|\alpha<f(i)\}\in D$,
 for every $\alpha\in \kappa $, yet for every   $g:I \to \kappa $,
$\{i\in I|g(i)<f(i)\}\in D$ implies that there is $\alpha\in\kappa $ such that
$\{i\in I|g(i)<\alpha \}\in D$.
In other words, in $\prod_D \langle \kappa,< \rangle$, $f_D$ is the least
element larger than all the $d(\alpha)$'s ($\alpha \in \kappa$) (as usual, $f_D$
denotes the class of $f$ modulo $D$).
\end{definition}

When D is over $\kappa $ we get the usual notion of a
{\it least function}. We need the above definition since we will be dealing
with arbitrary $(\lambda,\kappa )$-regular ultrafilters, not necessarily over
$\kappa $ (however, see Proposition \ref{6.7}(ii)).

Notice that if $D$ has a $\kappa $-least function then $D$ is
$ \kappa $-descendingly incomplete, hence if $ \kappa  $ is regular
then $D$ is  $( \kappa , \kappa )$-regular by  \ref{1.1}(xi). The
case
when $ \kappa $ is singular shall be discussed in Remark \ref{8.5}(a).

The following theorem generalizes \cite[Corollary 2.6]{Ka1}  (see also \cite[Section
1]{Ket2}), and is proved in a
very similar way.

\begin{thm}
\label{6.2}
  If $\kappa $ is regular, and the ultrafilter $D$ is
$(\kappa,\kappa )$-regular and has no $\kappa $-least
function, then $D$ is
$(\omega ,\kappa')$-regular for every $\kappa'<\kappa $.
\end{thm}

\begin{proof} The proof goes exactly as in the proofs of \cite[Theorem 2.5 and
Corollary 2.6]{Ka1}; let $D$ be over $I$: just replace everywhere
``$f:\kappa\to\kappa $'' with ``$f:I\to\kappa $'',
and ``$\xi<\kappa $'' with ``$i\in I$''. \end{proof}

The following is proved in \cite[Theorem 1.4]{Ket1}  under the assumption of the
existence of a
least function, but the proof uses only a $\kappa $-least function.

\begin{thm}
\label{6.3}
  Suppose that $\kappa $ is regular, and that the
ultrafilter $D$ over $I$ has a $\kappa $-least function $f$.
Then $D$ is $(\lambda,\kappa )$-regular if and only if $\{i\in I | {\rm cf} f(i) <\lambda \} \in D $.
\end{thm}

From Theorems \ref{6.2} and \ref{6.3} we get the following partial answer to 
Conjecture  
\ref{2.16ex2.14}.

\begin{cor}
\label{6.4}
  If $\kappa>\lambda $, $\kappa $ is regular, $\lambda
$ is singular and
the
ultrafilter $D$ is
$(\lambda^+,\kappa )$-regular,
then $D$ is either
$(\lambda,\kappa )$-regular, or
$(\omega ,\kappa')$-regular for every $\kappa'<\kappa $.
\end{cor}

\begin{proof} Since $\kappa>\lambda $, $D$ is
$(\kappa,\kappa )$-regular, by \ref{1.1}(i).

 If $D$ has no $\kappa $-least function we are done by Theorem 
\ref{6.2}.

Otherwise, let $D$ be over $I$, and let $f$ be a $\kappa $-least function. By
Theorem \ref{6.3}, $\{i\in I | {\rm cf} f(i) <\lambda^+ \} \in D $, 
hence
$\{i\in I | {\rm cf} f(i) <\lambda \} \in D $, since cofinalities are regular
cardinals. Now Theorem \ref{6.3} implies that $D$ is
$(\lambda,\kappa )$-regular. 
\end{proof}

See Theorem  \ref{2.15ex2.13}(ii) for a similar result in the case when $\kappa $ is singular.

Corollary \ref{6.4} puts some restrictions on possible solutions of 
Problem 
\ref{5.6}.

Corollary \ref{6.4} furnishes another proof of Theorem  
\ref{2.13ex2.11}(ii) in 
the particular case
when $\lambda $ is singular: take $\kappa ^+$ in place of $\kappa $ 
in \ref{6.4}, and
observe that, by  \ref{1.1}(i), both $(\lambda,\kappa^+)$-regularity
and
$(\omega,\kappa )$-regularity imply $(\lambda,\kappa )$-regularity.

The case $\lambda $ singular of Theorem  \ref{2.2}, too, can be 
obtained as a
consequence of Corollary \ref{6.4} (and of Proposition  \ref{2.6}): 
if $\mu 
$ in  \ref{2.2} is regular,
take $\kappa =\mu $ in \ref{6.4}; then if $\mu$ is singular, one can use 
an argument similar
to the one used at the end
of the proof of Theorem \ref{6.5} below.

Corollary \ref{6.4} shows that at least the
following weak form of Conjectures  \ref{2.16ex2.14} and  \ref{2.12ex2.10} holds: 
{\it  
if $\mu $ is regular, $\lambda$ is singular, $\lambda<\mu  $, and $D$
is $(\lambda^+,\mu )$-regular then  $D$ is either
$(\lambda,\mu )$-regular or $\lambda $-decomposable}
(since $(\omega,\lambda )$-regularity implies $\lambda
$-decomposability, by Remark  \ref{1.5}(b)).

Actually, we can use Theorems \ref{6.2} and \ref{6.3} in order to 
show that 
Conjecture  \ref{2.12ex2.10}
holds under some assumptions on cardinal arithmetic and cofinalities 
(recall that Conjecture  \ref{2.12ex2.10} is always true when $\mu $ is regular, as remarked after its statement).
Recall \cite{Sh} that  $\kappa\in {\rm pcf} {\bf a}  $ means that ${\bf a} $ is a set
of (distinct) regular cardinals, and there is an ultrafilter $D'$ over
${\bf a} $ such that $\kappa={\rm cf}\prod_{D'} {\bf a} $.

\begin{thm}
\label{6.5}
  Suppose that  $\kappa>\lambda $,
$\lambda $ is singular, ${\rm cf}\lambda\not={\rm cf}\kappa $, the
ultrafilter $D$ is
$(\lambda^+,\kappa )$-regular, and either

(a) $\lambda$ is a strong limit cardinal and 
$\cf S_\lambda(\lambda )<2^\kappa $
(in particular, this holds if $\lambda^{<\lambda }<2^\kappa $), or

(b) $\lambda^{<\lambda }<\kappa $, or

(c)  $\kappa$ is  regular and
   for no ${\bf a}\subseteq \lambda  $ with $|{\bf a} |<\lambda $
and $\sup{\bf a} =\lambda $
it happens
that
 $\kappa\in {\rm pcf} {\bf a}  $.

Then  $D$ is either $\lambda$-decomposable, or
$(\lambda',\kappa )$-regular for some $\lambda'<\lambda $.
\end{thm}

\begin{proof}
First we prove the result in the case when $\kappa $ is  regular.
We shall suppose that $D$ is not $\lambda$-decomposable, and get
$(\lambda',\kappa )$-regularity for some $\lambda'<\lambda $.

Were $D$  $(\omega, \kappa')$-regular for all $\kappa'<\kappa $ then in
particular $D$  would be
$(\omega, \lambda )$-regular (by  \ref{1.1}(i), and since
$\lambda<\kappa $) and trivially $\lambda $-decomposable (because of
Remark  \ref{1.5}(b)).
So, by \ref{1.1}(i) and Theorem  \ref{6.2}, we can suppose that $D$ has a $\kappa $-least 
function $f$, and
that
$D$ is $(\lambda , \kappa)$-regular,
by Corollary \ref{6.4}.

(a) By Proposition  \ref{1.4}, $|\prod_D 2^{<\lambda }|\geq 2^{\kappa
}$, but
$2^{<\lambda }=\lambda $ since $\lambda $ is a strong limit cardinal, so we get
$|\prod_D \lambda|\geq 2^{\kappa  }$.

Let $X$ be cofinal in $S_\lambda (\lambda )$ with
$|X|=\cf S_\lambda(\lambda )$.
Since $D$ is supposed to be not $\lambda $-decomposable, $\prod_D \lambda= \bigcup_{x\in X} \prod_D x$. 
Thus, there is $ x \in
X$ such that $|\prod_D x| \geq \lambda $,
since 
$|X|=\cf S_\lambda(\lambda )<2^\kappa $,
$ 2^\kappa> \kappa >\lambda $
 and $|\prod_D \lambda|\geq 2^{\kappa  }$. 

Arguing as in the last part of the proof of Theorem \ref{3.2}, we get 
that there are
arbitrarily large $\nu$'s $<\lambda  $ for which $D$ is $\nu$-decomposable; were
$D$  $({\rm cf}\lambda,{\rm cf}\lambda )$-regular, then $D$ would be
 $\lambda $-decomposable because of Theorem \ref{5.1}(b) with $\mu=\cf\lambda $
($\lambda^{<{\rm cf}\lambda  }=\lambda $
since $\lambda $ is a strong limit cardinal).

Hence $D$ is not $({\rm cf}\lambda,{\rm cf}\lambda )$-regular, but then Theorem
 \ref{2.2} implies that $D$
is $(\lambda',\kappa )$-regular for some $\lambda'<\lambda $.

(b) Suppose that $D$ is over $I$ and $f$ is a $\kappa $-least function. By 
Theorem \ref{6.3},
$\{i\in I | {\rm cf} f(i) <\lambda \} \in D $. If for some $\lambda'<\lambda $
$\{i\in I | {\rm cf} f(i) <\lambda' \} \in D $ then $D$ is
$(\lambda',\kappa )$-regular again by Theorem \ref{6.3}.

Otherwise, let $g:I\to\lambda $ be defined by $g(i)=\cf f(i)$
if $\cf f(i)< \lambda $ and $g(i)=0$ $\cf f(i)> \lambda $.
Let
$D'=g(D)$ over $I'=g(I)$. Since $f$ is a least function, ${\rm cf}(\prod_D
\langle f(i), <\rangle)=\kappa $, and, trivially, ${\rm cf}(\prod_D \langle
g(i), <\rangle)=\kappa $.
In general, it is not necessarily the case that also
${\rm cf}(\prod_{D'} \langle g(i), <\rangle)=\kappa $,
but we shall prove that this follows from the assumption that $D$
is not $\lambda $-decomposable.

Indeed, let $h_D\in \prod_D g(i)$; since $g(i)={\rm cf}f(i)<\lambda 
$, for every $i$ in a set in $D$, then we can chose a representative $h$ of $h_D$ which is a function from
$I$ to $\lambda $. If $h$ is not a $\lambda $-decomposition, there is
$X_h\subseteq\lambda $, $|X_h|<\lambda $ such that $\{i \in I|h(i)\in X_h  \}\in D $.
We have that $ \{i\in I |  |X_h|<g(i)\} \in D$, since otherwise $\{i\in I | {\rm
cf} f(i) \leq |X_h| \} \in D $, contrary to the first paragraph, since
$|X_h|<\lambda $.

If $i'\in I'$ and  $|X_h|<i'$, define 
$h'(i')=\sup \{h(j)|j \in I \text{ is such that } h(j)\in X_h \text{ and } 
g(j)=i' \}$. Notice that $h'(i')<i'$, since $|X_h|<i'$,
and $i'$ is a regular cardinal.
 If  $|X_h|\geq i'$, then the definition
of $h'(i')$ is arbitrary and irrelevant, since
$ \{i\in I |  |X_h|\geq g(i)\} \not \in D$ hence $ \{i'\in I' |
|X_h|\geq i'\} \not \in D'$.

Trivially, $\prod_{D'} \langle g(i), < \rangle $ is (can be identified with) a substructure of
$\prod_{D} \langle g(i), < \rangle $; since the above argument shows that for every
$h_D\in \prod_D g(i)$ there is $h'_D\in \prod_{D'} g(i)$ such that 
$h_D\leq h'_D$ (mod $D$),
we have that the two ultraproducts above have the same cofinality, hence
${\rm cf}(\prod_{D'} \langle g(i), <\rangle)=\kappa $.

Moreover, if $g$ is not a $\lambda $-decomposition, we can assume without loss
of generality that $|I'|=|g(I)|<\lambda $, but then $\kappa=\cf \kappa \leq |\prod_{D'}
g(i)|\leq \lambda^{|g(I)|}$ contradicts
the hypothesis  $\lambda^{<\lambda }<\kappa $.

(c) The arguments used in the proof of (b) give also (c), just letting  ${\bf a}=g(I)$:
the proof shows that if the conclusion fails then
$|{\bf a} |<\lambda $,
$\sup{\bf a} =\lambda $, and
$\kappa={\rm cf}\prod_{D'} {\bf a} $, contradicting the hypothesis of (c).

Now let $\kappa $ be singular.

In case (a), if ${\rm cf}\kappa<{\rm cf}\lambda $ then we get from Theorem
 \ref{2.15ex2.13}(ii) that $D$ is $(\lambda,\kappa )$-regular, and we can 
proceed as in the case $\kappa$ regular,
 since in the proof of (a) the existence of a least function has been used only in order to get
$(\lambda,\kappa )$-regularity.

The following argument, however, covers all cases when $\kappa $ is singular.

Having already proved the theorem in the case $\kappa$ regular,
we prove the case $\kappa$ singular by applying the case 
$\kappa$ regular to a set of 
 sufficiently large regular cardinals $\kappa'<\kappa$.

First notice that, in case (a), if $\lambda<\kappa'$, then
$\cf S_\lambda(\lambda )\leq \lambda^{<\lambda }\leq 2^{\lambda}\leq 2^{\kappa'}  $;
moreover, since $\cf S_\lambda (\lambda )< 2^\kappa $,
there is no $\kappa''$ with $\lambda<\kappa''<\kappa $ such that
$ 2^{\kappa'}= \cf S_\lambda(\lambda ) $ for all $\kappa'$ with
$\kappa''<\kappa'<\kappa $. Indeed, $\kappa$ singular,
together with classical cardinal arithmetic (see e.g. \cite[p. 257]{KM}), would imply $2^{\kappa} = \cf S_\lambda(\lambda )$, contradicting the
hypothesis. Hence, $ 2^{\kappa'}>\cf S_\lambda(\lambda ) $,
for large enough $\kappa'<\kappa $.

By what we have proved, the theorem is true for every regular and sufficiently
large $\kappa'<\kappa $ (in (b) it is enough to take $\kappa' >\lambda^{<\lambda
}$, and this is possible since $\kappa$ is singular $>\lambda^{<\lambda
}$).
Hence, if $D$ is not $\lambda $-decomposable, then
  for
every sufficiently large regular $\kappa'<\kappa $ there is $\lambda_{\kappa'}
<\lambda $ such that $D$ is $(\lambda_{\kappa'},\kappa' )$-regular.

Choose a set $C$ of sufficiently large cardinals, 
$C \subseteq \kappa $, $C$ cofinal in $\kappa $, with $|C|=\cf \kappa $, and
consider $C'=\{ \lambda_{\kappa'}|\kappa'\in C\} $.
If $\cf \kappa < \cf \lambda $ let $\lambda'=\sup C'$
and notice that $\cf \kappa < \cf \lambda $
implies that $\lambda'<\lambda $.
If $\cf \kappa > \cf \lambda $, then
there is $\lambda'<\lambda $ such that
$|\{\kappa'\in C|\lambda_{\kappa'}<\lambda'\} | \geq \cf \kappa  $,
since if
$|\{\kappa'\in C|\lambda_{\kappa'}<\lambda'\} |<\cf\kappa  $
for all $\lambda'<\lambda$, then 
$|C|< \cf \kappa  $, contradicting
the assumption that $C$ is cofinal in $\kappa $.

In both cases, if $\lambda'$ is as above, then by  \ref{1.1}(i), $D$ is
$(\lambda',\kappa')$-regular
for cf$\kappa $-many $\kappa' \in C $, hence for a set
of $\kappa' $ unbounded in $\kappa $ and, again by  \ref{1.1}(i),
for all $\kappa'<\kappa $.
Without loss of generality, we can choose $\lambda'$ regular, so that 
Proposition \ref{5.4}(a)
implies that $D$ is $(\lambda',\kappa )$-regular.
\end{proof}

The main argument in the proof of \ref{6.5}(a), in the  case when  
$\kappa $ is
regular, is taken from \cite[Theorems 11 and 12]{Pr}, where a slightly less general
result is obtained for the particular case $\kappa=\lambda^+$. The arguments for
cases (b) and (c), as well as for the case $\kappa$ singular, seem to be new.
See \cite{Lp10,lpndj}  for other applications of Shelah's pcf-theory to decomposability and regularity of ultrafilters. 

The results in \cite{Ka1}, as well as the generalization given in Theorem 
\ref{6.2}, suggest
the following problems. The problems are quite natural, but can 
hardly be found
in the literature; see, anyway, \cite[p.231]{Ket2}, \cite[Remark 6.9]{Sh1}.

\begin{problem}
\label{6.6}
  Are the following theorems of ZFC?

(A)  If $\kappa $ is regular, and the ultrafilter $D$ is
uniform over $\kappa $ and has no least function, then $D$ is
$(\omega ,\kappa)$-regular.

(B)  If $\kappa $ is regular, and the ultrafilter $D$ is
$(\kappa,\kappa )$-regular and has no $\kappa $-least function, then $D$ is
$(\omega ,\kappa)$-regular.

\end{problem}

Clearly, \ref{6.6}(B) implies \ref{6.6}(A), by \ref{1.1}(vi). We 
state below two other 
consequences of \ref{6.6}(B).

\begin{prop}
\label{6.7}
  Suppose that \ref{6.6}(B) holds. Then:

(i) If $\mu $ is regular and $\lambda $ is singular then every
$(\lambda^+,\mu)$-regular ultrafilter is $(\lambda ,\mu)$-regular
(that is, Conjecture  \ref{2.16ex2.14} holds for $\mu $ regular).

(ii) If $\lambda\leq\kappa $, $\kappa $ is regular and $D$ is $(\lambda ,\kappa )$-regular then there
exists a $(\lambda ,\kappa )$-regular $D'\leq D$ (in the Rudin-Keisler
order),
$D'$ uniform over $\kappa $.
\end{prop}

\begin{proof}  (i) is proved as Corollary 
\ref{6.4}, 
using \ref{6.6}(B) in place of Theorem
\ref{6.2}.

(ii) If $D$ over $I$ is
 $(\omega,\kappa )$-regular,
 then there is a function
$f:I\to S_\omega (\kappa  )$ witnessing it. But $|S_\omega (\kappa
)|=\kappa^{<\omega }=\kappa  $ so that $D'=f(D)$ is uniform on $\kappa $ and
$(\omega,\kappa )$-regular.

If, on the contrary, $D$ is not  $(\omega,\kappa )$-regular, then
 by \ref{6.6}(B) there is a $\kappa $-least function 
$f:I\to\kappa $. So
$D'=f(D)$ is uniform on $\kappa $ and
the identity function $id$ is a least
function for $D'$ (such an
ultrafilter is usually called {\it weakly normal}).
By Theorem \ref{6.3},
$\{i\in I | {\rm cf} f(i) <\lambda \} \in D $, hence
$\{\alpha\in\kappa  | {\rm cf} id(\alpha) <\lambda \} \in D' $, so,  by
Theorem \ref{6.3} again, the least function $id$ proves the 
$(\lambda,\kappa 
)$-regularity of
$D'$.  
\end{proof}

\ref{6.7}(ii) is a quite natural requirement: it imports that, for 
most purposes,
it is enough to consider only $(\lambda ,\kappa )$-regular ultrafilters over
$\kappa $.

Notice that the assumption $\kappa $ regular is necessary in 
\ref{6.7}(ii):
every ultrafilter uniform over $\omega $ is $(\omega_\omega , \omega_\omega )$-regular, by \ref{1.1}(vi)(v). A more significant example, in which $\kappa>\lambda $,
is the following:
if $\lambda $ is a strongly compact cardinal, then there is an
$\omega_1$-complete
$(\lambda ,\lambda^{+\omega } )$-regular  ultrafilter, but no such ultrafilter
can be uniform over $\lambda^{+\omega }$, otherwise it would be
$(\omega, \omega ) $-regular 
by
 \ref{1.1}(vi), and this
contradicts $\omega_1$-completeness.

However, in the above example,   the ultrafilter can be chosen to be uniform over
$\lambda^{+\omega+1} $, so that, as far as we know, the following might be a
theorem of ZFC:

\smallskip

{\bf \ref{6.7}(ii)*}  If $\lambda<\kappa $ and
$D$ is $(\lambda ,\kappa )$-regular then there
exists a $(\lambda ,\kappa )$-regular $D'\leq D$
such that $D'$ is uniform either over $\kappa^+ $ or over $\kappa $.

\smallskip

Notice also that \ref{6.7}(ii) is true when $\kappa=\lambda^{+n} $ 
and $\lambda $ is
regular, since we can
always get a $D'\leq D$ over {\rm cf}$S_\lambda(\kappa )$, by
\ref{1.5}(b), and it is easy
to show that
{\rm cf}$S_\lambda(\lambda^{+n}  )=\lambda^{+n} $, if $\lambda$ is regular.

Of course, in case \ref{6.6}(B) turned out to be unprovable in ZFC, 
we might ask
the problem whether \ref{6.7}(ii) or \ref{6.7}(ii)* are theorems 
of 
ZFC.

We suspect that \ref{6.7}(ii)* and \ref{6.7}(ii) are equivalent. 

Is it true that \ref{6.7}(ii) 
implies 
that 
\ref{6.6}(A) and \ref{6.6}(B)
are equivalent?

We now turn to another kind of problem.

\begin{problem}
\label{6.8}
 For which sets $K$ of  infinite cardinals is there  an ultrafilter
which is $\kappa $-decomposable exactly for those $\kappa\in K$?

In other words, if $D$ is an ultrafilter, let $K_D=\{\kappa\geq\omega| D$ is
$\kappa $-decomposable$\}$. Which are the possible values $K_D$ can take?
\end{problem}

Many constraints are known on $K_D$: first, every $\kappa $-decomposable
ultrafilter is ${\rm cf} \kappa $-decomposable, by 
\ref{1.1}(vii)(viii); and the least $\kappa $ for
which an ultrafilter is $\kappa $-decomposable  is always a measurable cardinal
or $\omega $, as mentioned shortly after  \ref{1.5}. Moreover, if
$\kappa $ is regular, then every $\kappa^+$-decomposable
ultrafilter is $\kappa $-decomposable, by \cite{CC}, \cite{KP}, or Theorem 
2.1 here, and
 \ref{1.1}(xi). By the
same results,
if $\kappa $ is singular, then every $\kappa^+$-decomposable
ultrafilter is either ${\rm cf} \kappa $-decomposable
or $\kappa' $-decomposable for all sufficiently large regular $\kappa'<\kappa $
(use \ref{1.1}(xi)(xii)).

Further constraints on $K_D$ are given by Theorems \ref{adesso2.9}, \ref{box}  \ref{2.19ex2.17}, 
\ref{3.2}, 
\ref{4.3}, \ref{5.1}, \ref{abst} \ref{6.5},
Propositions \ref{3.3}, \ref{6.10}, \ref{8.1}, \ref{ex5.10} and  Corollary 
\ref{5.3}. See also
Propositions \ref{7.1}, \ref{7.6}, and Problems \ref{3.5}, \ref{4.4}, 
\ref{5.2}, 
\ref{7.2}.

We now consider in detail the cases when $|K_D|\leq 2$.

If $|K_D|=1$, say $K_D=\{\mu \} $, then $\mu $ is either $\omega $,
or a measurable cardinal.

There are many open problems already for the case $|K_D|=2$. The case when
 $K_D=\{\omega,\kappa  \} $ has originally been studied by J. Silver \cite{Si};
in this case, $D$  is usually called {\it
indecomposable}, and only the following possibilities can occur,
because of  \ref{1.1} and Theorem
 \ref{2.1}:
 either $\kappa $ has cofinality $\omega $,
or $\kappa $ is weakly inaccessible, or
$\kappa $ is the successor of a cardinal of cofinality $\omega $.

If $\mu $ is a measurable cardinal and $D$ is a $\mu $-complete uniform
ultrafilter over $\mu $, then $D$ is $\mu $-decomposable, but not
$\lambda $-decomposable 
for all $\lambda<\mu $, and, trivially, not $\lambda $-decomposable 
for 
all $\lambda>\mu $.
If $D'$ is uniform over $\omega$, then
$D'$ is $\omega$-decomposable, by \ref{1.1}(iii). Thus, by Proposition \ref{7.1},
$K_{D \times D'}=\{\omega, \mu \}$. If 
$\mu $ is measurable, 
and  $\mu $ is made singular by Prikry forcing \cite{Pr1}, then in 
the resulting model
there is an ultrafilter which is $\kappa $-decomposable exactly for
$\kappa=\omega $ and $\kappa=\mu $.

We can get a similar example without using forcing, but starting with $ \omega $
measurable cardinals.
Let $D$ be uniform over $ \omega $, and let
$(\mu_n) _{n \in \omega } $ be a strictly increasing sequence of measurable cardinals,
and take $D_n$ to be a $\mu_n$-complete uniform ultrafilter
over $\mu_n$.
Set $\mu=\sup \{\mu_n |n \in \omega \}  $ and
$D^*=\sum_D D_n$ (see Section \ref{rup} for the definition).
It is easy to show that
$K_{D^*}= \{\omega,\mu  \} $.

The example can be modified in order to get:
$K_{D^{**}}= \{\omega,\mu_0, \mu_1, \mu_2, \dots, \mu_n, \dots \mu  \} $:
just take
$D'_n=D_0 \times D_1 \times \dots \times D _{n-1}\times D_n $, and
$D ^{**} =\sum_D D'_n$.

In the model constructed in \cite{BM}
there is an ultrafilter $D$ such that  
$K_D=\{\omega,\omega_{\omega}, \omega _{ \omega + 1} \}$
($\omega_{\omega}$-decomposability
follows from GCH and Theorem \ref{6.5}(a) with 
$ \kappa = \lambda ^+$). By taking a projection along some
$\omega_{\omega}$-decomposition, we get 
an ultrafilter $D'$ such that  
$K _{D'}  =\{\omega,\omega_{\omega}\}$. See also 
the result by H. Woodin stated in \cite[Theorem 1.5.6(iv)]{Ma}.
The constructions given in \cite{BM, AH} might shed further light on
the problem
of the possible values $K_D$ can take.

\cite{Shr} has constructed an ultrafilter for which $K_D=\{\omega,\kappa  \} $, where
$\kappa $ is inaccessible and not weakly compact. However, it is not exactly known
(\cite[p. 1007]{Shr}) for which inaccessible cardinals $\kappa $ we can have
$K_D=\{\omega,\kappa  \}$; by Theorem  \ref{2.19ex2.17} and 
\ref{1.1}(xi)(xii), 
$\kappa $
must be $\omega $-Mahlo, but it is not known whether we can have, say,
$\kappa $ not $\omega+1$-Mahlo.

As far as the remaining possibilities are concerned in the case 
$|K_D|=2$, we do not know whether we can have:

(a) $K_D=\{\omega,2^{\omega_1}\} $ and $2^{\omega}<2^{\omega_1}$, or

(b) $K_D=\{\omega,\omega_{\omega+1}\} $, or, more generally,

(c) $K_D=\{\omega,\kappa^+\} $ with ${\rm cf}\kappa=\omega $, and $ \kappa \not = \omega $.

An affirmative answer to Conjecture  \ref{2.12ex2.10} would prevent (b) and 
(c), by \ref{1.1}(xi)(xii). If
Problem \ref{3.5}(a) has an affirmative answer then (a) cannot hold: 
this is proved as follows. If $D$ is
$2^{\omega_1 }$-decomposable, then $D$ is
$(2^{\omega_1},2^{\omega_1})$-regular by  \ref{1.1}(vii); now,  take $\kappa=\omega $ and $m=n=1$ in 
the statement of
Theorem \ref{4.3}(a$'$), and recall that if Problem \ref{3.5}(a) has an 
affirmative answer then
the conclusion of \ref{4.3}(a$'$) can be improved to $\kappa<\mu\leq 
2^\kappa $, so that,
in the present case, $D$ is $\mu $-decomposable for some $\mu $ with
$\omega<\mu\leq 2^\omega <2^{\omega_1}$.

Apparently, the case when $|K_D|=2$ and $\inf K_D>\omega $ has never been
studied. There are  trivial cases, e.g., if $D$ is $\mu $-complete 
and uniform over
$\mu $, and $D'$ is $\mu' $-complete and uniform over $\mu' $, then
 $K_{D\times D'}=\{\mu,\mu'\} $.
As another example, if $\kappa$ is $ \kappa ^+$-compact, then there 
is  an ultrafilter
$D$ which is $\kappa$-complete, $( \kappa , \kappa ^+)$-regular, and uniform over $ \kappa ^+$, hence $K_D=\{\kappa,\kappa^+\}$, by \ref{1.1}(xii).

In case $K_D$ is infinite 
Shelah's {\bf pcf} theory \cite{Sh}
deeply influences the possible values of $K_D$ (\cite{Lp10,lpndj}).

The possibility that $K_D$ is an interval can always occur: if $D$ is uniform
over $\lambda $ and $(\omega,\lambda )$-regular, then
$K_D=[\omega,\lambda ]$, by \ref{1.1}(i) and Remark  \ref{1.5}(b).
If there is no inner model with a measurable cardinal, then by Donder's Theorem
  \ref{1.6}, if $D$ is uniform
over $\lambda$ then $K_D$ is always equal to $[\omega,\lambda ]$, since $D$ is $\lambda' $-decomposable for all $\lambda'<\lambda $ by 
\ref{1.1}(i) and Remark
  \ref{1.5}(b); moreover, $D$ is
$\lambda $-decomposable by  \ref{1.1}(iii).

If $\kappa $ is $\lambda $-compact and $\lambda $ is regular then there is
a $\kappa $-complete $(\kappa,\lambda )$-regular ultrafilter uniform over $\lambda $,
hence $K_D= \{ \mu| \kappa \leq \mu \leq \lambda \text{ and } \cf\mu \geq \kappa \}$,
by \ref{1.1}(xii) in the case $\mu$ regular; then apply Theorem 
\ref{5.1}(b) in order to get the case $\mu$ singular 
(since $\mu^{<\kappa }=\mu$, by the
result we mentioned shortly after the definition of strong compactness).
If $\kappa \leq \mu \leq \lambda $
and $\cf \mu < \kappa $ then $\mu \not \in K_D$ by
\ref{1.1}(vii). 

Hence, a general solution of Problem \ref{6.8} appears to be quite 
difficult.

Problem \ref{6.8} appears to be connected also with some variations 
on the principle we
have denoted by $U'(\lambda )$ (see Definition  \ref{1.7}).

\begin{definition}
\label{6.9}
  Let $\lambda $ be a limit cardinal.

(i) \cite[Definition 4.4]{Lp1}  $U^*(\lambda )$ means that for every ultrafilter $D$,
if there are arbitrarily large regular $\kappa<\lambda $ such that  $D$ is
$\kappa $-decomposable, then $D$ is $(\lambda ,\lambda )$-regular.

(ii) $U(\lambda )$ means that for every ultrafilter $D$, if there
is $\lambda'<\lambda $ such that  $D$ is
$\kappa $-decomposable for all regular
$\kappa $ with $\lambda'<\kappa<\lambda $, then $D$ is
$(\lambda ,\lambda )$-regular.
\end{definition}

In \cite[p. 132]{Lp4} $U(\lambda )$ has been stated in the following equivalent
form: ``for every ultrafilter $D$, if there
is $\lambda'<\lambda $ such that  $D$ is
$(\kappa, \kappa ) $-regular for all
$\kappa $ with $\lambda'<\kappa<\lambda $, then $D$ is
$(\lambda ,\lambda )$-regular''. The equivalence follows from
\ref{1.1}(xi) and Theorem \ref{2.1}(b).  

Theorem \ref{abst} shows that  if $ \lambda $ is a singular cardinal, then
$U( \lambda )$ holds.  

Clearly, for every limit cardinal $ \lambda $, $U'(\lambda) $ implies $U^*(\lambda)$, which implies $U(\lambda) $.
\cite{Lp1, Lp4} contain applications of the above principles to abstract logics. See also \cite{lpndj} for improved results, whose proofs
implicitly use $U( \lambda )$. 

As far as we know, it is possible that
$U^*(\lambda )$
 is a theorem of ZFC, for every singular cardinal $\lambda $
(our guess is that
it is not a theorem).
For sure, the failure of 
$U^*(\lambda)$ for $\lambda $ singular
 has a quite large consistency strength; see \cite[Proposition 4.5]{Lp1}, Corollary \ref{5.3} and the next proposition. 
We originally
thought that the above principles were quite similar in strength, but now we
know that it is much harder to make $U(\lambda )$ fail. 
In fact, if $U(\lambda )$ fails then $ \lambda $ is weakly inaccessible,
by Theorem \ref{abst}.
 
The following Proposition gives a condition
equivalent to the failure of $U^*(\lambda )$.

\begin{prop}
\label{6.10}
 Let $\lambda $ be a singular cardinal.
Then
$U^*(\lambda )$ fails if and only if 
 there is an ultrafilter $D$ such that:

(a) $D$ is not $\cf \lambda $-decomposable;

(b) there are arbitrarily large regular $\kappa<\lambda $ for which $D$ is
$\kappa $-decomposable;

(c) there are arbitrarily large regular $\kappa<\lambda $ for which $D$ is not
$\kappa $-decomposable.
\end{prop}

\begin{proof} (i) By definition, if $U^*(\lambda )$ fails there is $D$ which satisfies
(b) and is not $(\lambda  ,\lambda  )$-regular.

$D$ satisfies (a) since every  $\cf \lambda $-decomposable ultrafilter is
$(\lambda ,\lambda )$-regular, by \ref{1.1}(xi)(v).

Were (c) false, there would be $\lambda'<\lambda $ such that $D$ is
$\kappa $-decomposable for all regular $\kappa $ with $\lambda'<\kappa<\lambda  $, but
then  Theorem \ref{abst} would imply that $D$ is
$(\lambda ,\lambda )$-regular, a contradiction.

Conversely, let $D$ satisfy (a), (b), and (c). Because of (b), if $D$ is not $(\lambda,\lambda )$-regular then $U^*(\lambda )$ fails. By Proposition  
\ref{2.6} and \ref{1.1}(xi), every
$(\lambda ,\lambda )$-regular ultrafilter is either
 ${\rm cf}\lambda  $-decomposable  or $(\lambda',\lambda )$-regular for some
$\lambda'<\lambda $. But this is impossible: the first possibility cannot occur
because of (a), and the second possibility cannot occur because of (c), since every
$(\lambda',\lambda )$-regular ultrafilter is $\kappa $-decomposable for all
regular $\kappa $ with $\lambda'\leq\kappa\leq\lambda $, by \ref{1.1}(xii).
\end{proof}

Proposition \ref{6.10} still holds if we replace everywhere 
$U^*(\lambda )$ by
$U'(\lambda )$ and we delete the word ``regular'' in condition (b).

Notice that the example mentioned shortly before Definition \ref{6.9}
can be used in order to provide a singular cardinal $\lambda $
and a $(\lambda, \lambda )$-regular not $\cf\lambda $-decomposable ultrafilter  
$D$ such that there are arbitrarily large regular $\kappa<\lambda $ for which $D$ is
$\kappa $-decomposable and there are arbitrarily large singular $\kappa<\lambda $ for which $D$ is not
$\kappa $-decomposable.
Just take $\lambda $ singular with
$\omega<\cf\lambda<\kappa $.

We expect to be able to find more applications of Kanamori and Ketonen's
results, as well as  of their generalizations Theorems \ref{6.2} and 
\ref{6.3}.

\section{ Regularity of products.}
\label{rup}

Given certain regular ultrafilters, we sometimes can ``sum'' their
regularities by taking products. This is because the regularity properties of
$D\times D'$ are determined by the regularity properties of $D$ and of $D'$
(Proposition \ref{7.1}).

In this section we shall present some examples, and, more generally, we shall
consider $D$-sums; some similar results  appeared in \cite[Section 5]{Ket1} under
much stronger assumptions, such as $\omega_1$-completeness.

The product $D\times D'$ of two ultrafilters $D$ and $D'$ (over $I$, $I'$,
respectively) is the ultrafilter over $I\times I'$
defined by:
$X\in D\times D'$ if and only if $ \{i\in I| \{i'\in I'|(i,i')\in X \}\in D'\}\in D $.

The following proposition is useful and has a simple proof, but might be
new.

\begin{prop}
\label{7.1}
  For $D$, $D'$ ultrafilters,  the following are
equivalent:

(a) $D\times D'$ is $(\lambda,\mu )$-regular;

(b) there is a cardinal $\nu$ such that $D$ is $(\nu,\mu )$-regular and $D'$ is
$(\lambda,\nu')$-regular for all $\nu'<\nu$.

Thus, in particular, $D\times D'$ is $(\lambda,\lambda  )$-regular
if and only if either
$D$ is $(\lambda,\lambda  )$-regular
or
$D'$ is $(\lambda,\lambda  )$-regular.
\end{prop}

\begin{proof} We believe that the proof can be best viewed in model-theoretical terms,
using Form III of the definition of regularity
(at least, we discovered the result by reasoning in model-theoretical terms).
An alternative proof of \ref{7.1} using Form I can be obtained from
 the proof of  \ref{7.4} below (which is a result stronger than \ref{7.1}).

It is well known that, for every model {\bf A},
$\prod_{D\times D'} {\bf A} \cong \prod_{D} \prod_{D'} {\bf A} $
(throughout, we shall use the same names for corresponding elements of
$\prod_{D\times D'} {\bf A} $ and of $  \prod_{D} \prod_{D'} {\bf A} $).

Now, suppose that there exists a $\nu$ as in statement (b) of the Proposition.
Then
the following is true: whenever $X\subseteq\mu $ and $|X|<\nu$ then there is
$x_{X} \in \prod_{D'} \langle S_\lambda(\mu ), \subseteq, \{\alpha \}\rangle_{\alpha<\mu}$ such that $
d(\{\alpha \}) \subseteq x_{_X} $ for every $\alpha\in X$ 
(by Form III, since if $|X|=\nu'<\nu$
then $D'$ is
$(\lambda,\nu')$-regular, and $S_\lambda(X) \subseteq S_\lambda(\mu)$ is isomorphic to $S_\lambda(\nu')
$). 
Further, we can have
$X = \{ \alpha \in \mu | d({\alpha })\subseteq x_X\} $:
just replace $x_X$ by
$x'_X(i')=X \cap x_X (i')$.

In other words,
$\prod_{D'} \langle S_\lambda(\mu ), \subseteq,\{\alpha \} \rangle_{\alpha<\mu}$
contains a ``copy'' of
 $\langle S_{\nu}(\mu), \subseteq, \{\alpha \}\rangle_{\alpha<\mu}$, whence, by the
$(\nu,\mu )$-regularity of $D$, in
$\prod_D \prod_{D'} \langle S_\lambda(\mu ), \subseteq,\{\alpha \} \rangle_{\alpha<\mu}
\cong \prod_{ D
\times D'} \langle S_\lambda(\mu ), \subseteq,\{\alpha \} \rangle_{\alpha<\mu} $ there is an element $y$
such that
$ d(\{\alpha \}) \subseteq y $ for every $\alpha\in\mu $, so that $D\times D'$
is $(\lambda,\mu )$-regular.

If one needs the actual definition of a
$(\lambda,\mu )$-regularizing function, 
this goes as follows: for every $X\subseteq\mu $
with $|X|<\nu$ let $x_{_X}$ be as above.
Thus, $x_{_X}$ is (the equivalence class modulo $D$ of)
a function $x_{_X}:I'\to S_ \lambda (X) \subseteq S_ \lambda (\mu)$
such that if $ \alpha \in X$ then
$\{i'\in I'| \alpha \in  x_{_X}(i')\}\in D' $.
Now, let $g:I\to S_\nu(\mu )$ witness
the $(\nu,\mu )$-regularity of $D$, as given by Form II.
Thus, for every $ \alpha \in \mu $, $\{i\in I| \alpha \in g(i)\}\in D $.
Then $f:I\to \prod_{D'}  S_\lambda(\mu )$
defined by
$f(i)=x_{g(i)}$
witnesses
the $(\lambda,\mu )$-regularity of
$D\times D'$, since, for every $ \alpha \in \mu$,
$\{(i,i')| \alpha \in f(i)(i')\}=
 \{(i,i')| \alpha \in x _{g(i)} (i')\} \in D \times D'$, since
$\{i \in I |\{i' \in I'| \alpha \in x _{g(i)} (i')\}\in D' \} \supseteq
\{i \in I | \alpha \in g(i)\}  \in D $.

Having proved that (b)$ \Rightarrow $(a), let us prove
(a)$ \Rightarrow $(b).
 Suppose that
$D\times D'$ is $(\lambda,\mu )$-regular,
so that, by Form III, in
$\prod_{ D \times D'} \langle S_\lambda(\mu ), \subseteq, \{\alpha \}\rangle_{\alpha<\mu} $ there is an
element $x$ such that
$ d(\{\alpha \}) \subseteq x $ for every $\alpha\in\mu $.
Thus, $x$ is in $\prod_D \prod_{D'} S_\lambda(\mu )$, and this means that there
is a function $f:I\to \prod_{D'}  S_\lambda(\mu )$ such that for every
$\alpha\in\mu $ 
$\{  i\in I| \prod_{D'} S_\lambda(\mu)\models d(\{\alpha \} )\subseteq f(i)\} \in D$.

Define $g:I\to S(\mu )$
by $g(i)= \{\alpha\in \mu| \prod_{D'} S_\lambda(\mu)\models d(\{ \alpha \} ) \subseteq f(i) \} $. Since $\alpha\in g(i)$
if and only if
$\prod_{D'} S_\lambda(\mu)\models d(\{\alpha \} )\subseteq f(i)$, then
for every $\alpha\in\mu $ $\{  i\in I| \alpha\in g(i)\} \in D$.

Let $\nu=\sup_{i\in I}(|g(i)|^+) $; thus, $g:I \to S_\nu(\mu)$ makes $D$ $(\nu,\mu )$-regular, according to Form II. By
the definition of $\nu $, for every $\nu'<\nu $ there is $i\in I$  such that
$|\{\alpha\in \mu|\alpha\in g(i) \}|\geq\nu' $.

Given any $\nu'<\nu $, fix some $i$ as above.
Then
$|\{\alpha\in \mu |\prod_{D'} S_\lambda(\mu)\models d(\{ \alpha \} ) \subseteq f(i) \}|\geq \nu' $.
Choose
$X\subseteq \{\alpha\in \mu |\prod_{D'} S_\lambda(\mu)\models d(\{ \alpha \} ) \subseteq f(i) \}$
with $|X|=\nu'$, and, for $i'\in I'$, define $f'(i')=X\cap f(i)(i')$.
Since $S_\lambda(\nu')$ is isomorphic to
$S_\lambda (X)$, $f':I'\to S_\lambda(X)$ witnesses
the $(\lambda,\nu')$-regularity of $D'$, as given by Form II, since
if $ \alpha \in X$ then $\prod_{D'} S_\lambda(\mu)\models d( \{ \alpha \}  ) \subseteq f(i)$, that is,
$\{i'\in I'| d( \{ \alpha \}  ) \subseteq f(i)(i')\} \in D' $, thus
$\{i'\in I'| d( \{ \alpha \}  ) \subseteq f'(i')\} \in D' $.

As for the last statement in the Proposition, the if-part follows from  \ref{1.1}(ii)
and the fact that both
$D\leq D\times D'$
and $D'\leq D\times D'$.

On the other side, if
$D\times D'$ is $(\lambda,\lambda  )$-regular, then, by the equivalence of (a) and (b),
there is a cardinal $\nu$ such that $D$ is
 $(\nu,\lambda)$-regular and $D'$ is
$(\lambda,\nu')$-regular for all $\nu'<\nu$.
Thus, by  \ref{1.1}(i), if $\nu\leq \lambda $ then $D$ is
$( \lambda ,\lambda)$-regular, and if
 $\nu>\lambda $ then $D'$ is
$( \lambda ,\lambda)$-regular
\end{proof}

Thus, for example, if $D$ is $(\nu^+,\mu )$-regular and $D'$ is
$(\lambda,\nu)$-regular, and neither $D$ nor $D'$ is $(\kappa,\kappa )$-regular, then $D\times
D'$ is
$(\lambda ,\mu )$-regular and not $(\kappa,\kappa )$-regular (see also
Proposition \ref{7.6}).

As another example, if $D$ is not
$(\lambda^+,\lambda^+)$-regular and if $D'$ is not $(\lambda,\lambda^+)$-regular, then $D\times D'$ is not
$(\lambda,\lambda^+)$-regular.

On the contrary,
if $D$ is $(\lambda^+,\lambda^+)$-regular, and
$D'$ is $(\lambda,\lambda)$-regular then $D\times D'$ is
 $(\lambda,\lambda^+)$-regular (this improves \cite[Theorem 5.8]{Ket1});
in particular, Theorem  \ref{2.1}(b) implies that if $D$ is
$(\lambda^+,\lambda^+)$-regular then $D\times D$ is $(\lambda,\lambda^+)$-regular. More generally, 
if $D$ is $(\lambda^{+n},\lambda^{+n})$-regular
then $D\times D \times \dots \times D$ ($n+1$ factors) is  $(\lambda,\lambda^{+n})$-regular. 
Moreover, if $D$ is $(\lambda^{+n+1},\lambda^{+2n+1})$-regular
then $D\times D$ is  $(\lambda,\lambda^{+2n+1})$-regular, since
if $D$ is $(\lambda^{+n+1},\lambda^{+2n+1})$-regular then $D$
is $(\lambda,\lambda^{+n})$-regular by iterating Theorem  
\ref{2.13ex2.11}(ii)
$n+1$ times; then apply Proposition \ref{7.1} with $D=D'$
and $\nu=\lambda^{+n+1}$.

Notice that Proposition \ref{7.1}, together with  \ref{1.1}(xi), implies
that, if $\kappa $
is {\it regular}, then $D \times D'$ is $\kappa $-decomposable if and only if either $D$ or
$D'$ is $\kappa $-decomposable. This is not necessarily true when $\kappa $ is
singular: let $D$ be uniform over $\omega $, and suppose that $\kappa $ is
$\kappa^{+\omega }$-compact; thus, there is an $\omega_1$-complete
$(\kappa, \kappa^{+\omega })$-regular ultrafilter $D'$, which is not
$\kappa^{+\omega } $-decomposable,
by  \ref{1.1}(vii), but which is
 $\kappa^{+n}$-decomposable for all $n<\omega $, by
 \ref{1.1}(xii). Now,  $D\times
D'$ is $(\omega, \omega )$-regular and $\kappa^{+n}$-decomposable for
all $n<\omega $, by  \ref{1.1}(vi)(xi) and the last statement in
Theorem \ref{7.1}; then Theorem \ref{5.1}(b) implies that   $D\times D'$ 
is
$\kappa^{+\omega }$-decomposable, since 
$(\lambda^{+\omega })^{<\omega }= \lambda^{+\omega }$.
However,
neither $D$ nor $D'$ is  $\kappa^{+\omega
}$-decomposable.

We do not know whether we have a counterexample as above in which
$D=D'$.

\begin{problem}
\label{7.2}
Can there be a $D$ not $\kappa $-decomposable such that
$D\times D$ is $\kappa $-decomposable?
\end{problem}

As an application of Proposition \ref{7.1} we can get a
generalization (with a simpler proof) of a result by Ketonen. \cite[Theorem 1.1]{Ket2} 
is the particular case $\lambda = \kappa $ of the next proposition.

\begin{prop}
\label{7.3}
  If $\kappa $ is regular, $\lambda \geq \kappa $, and the ultrafilter $D$
is $(\kappa, \lambda  )$-regular and has no $\kappa $-least function then $D\times
D$ is $(\omega, \lambda )$-regular.
\end{prop}

\begin{proof} Immediate from \ref{1.1}(i),  Theorem \ref{6.2} and Proposition \ref{7.1}. 
\end{proof}

There is a version of Proposition \ref{7.1} for sums of ultrafilters.

If $D$ is an ultrafilter over $I$, and for every $i\in I$ $D_i$ is an
ultrafilter over some set $I_i$, the {\it $D$-sum} $\sum_D D_i$ of the $D_i$'s
modulo $D$  is the ultrafilter
over
$ \{(i,j)|i\in I , j\in I_i  \}  $
defined by
$X\in \sum_D D_i$ if and only if $ \{i\in I| \{j \in I_i| (i,j)\in X    \} \in D_i       \} \in D $ (cf.
e.g. \cite[Definition 0.4]{Ket1}).

In the particular case when all the $D_i$'s are equal (to, say, $D'$)  we get
the product $D\times D'$, thus the following proposition generalizes Proposition \ref{7.1}.

\begin{prop}
\label{7.4}
 (a) $\sum_D D_i$ is $(\lambda,\mu )$-regular if and only if  there is a
function $g:I\to S(\mu)$
such that
$\{ i\in I |\alpha\in g(i)\}\in D $ for every $\alpha \in \mu $ and such that 
for every $i\in I$ $D_i$ is $(\lambda,|g(i)|)$-regular (equivalently, we can just ask
that
$\{i\in I| D_i$ is $(\lambda,|g(i)|)$-regular$\}\in D$).

(a$'$) $\sum_D D_i$ is $(\lambda,\mu )$-regular if and only if  there is a
family $(X_\alpha )_{\alpha\in\mu }$
of elements in $D$ such that
for every $i\in I$ $D_i$ is
$(\lambda, |\{\alpha\in\mu|i\in X_\alpha \} |)$-regular
(equivalently, we can just ask  that
$\{i\in I| D_i \text{ is }
(\lambda, |\{\alpha\in\mu|i\in X_\alpha \} |)\text{-regular}\}\in D$).

(b)
If $\sum_D D_i$  is $(\lambda ,\mu )$-regular 
then for every cardinal $\nu$
either $D$ is $(\nu,\mu )$-regular, or
$\{i\in I| D_i \text{ is } \brfr (\lambda,\nu)\text{-regular}\}\in D $.

(c)
If $D$ is $(\nu^+,\mu )$-regular and 
$\{i\in I| D_i$ is $(\lambda,\nu)$-regular$\}\in D $,
then
$\sum_D D_i$ is $(\lambda,\mu )$-regular.

(d) $\sum_D D_i$ is $(\lambda,\lambda )$-regular if and only if
either $D$ is $(\lambda,\lambda )$-regular or
$\{i\in I| D_i \text{ is } (\lambda, \lambda )\text{-regular}\}\in D$.
\end{prop}

\begin{proof}
(a) can be proved in a way similar to the proof of Proposition \ref{7.1}, noticing that,
for every model ${\bf A}$, if
$E=\sum_D D_i$ then $\prod_E {\bf A}\cong \prod_D\prod_{D_i} {\bf A}  $. A direct proof of (a$'$) is given below.
However, the two proofs are interchangeable, since it is immediate
to see that (a) and (a$'$) are equivalent. Indeed, if  $g$ is a function
as given by (a), then define, for each $\alpha\in \mu $,
$X_\alpha = \{ i\in I|\alpha\in g(i)\} $: the $X_\alpha$'s then satisfy (a$'$).
Conversely, given $X_\alpha$'s as in (a$'$), let
$g(i)= \{\alpha\in\mu |i\in X_\alpha \} $: then $g$
satisfies (a) (indeed, this is nothing but the usual proof for
the equivalence of Forms I and II in the definition of $(\lambda,\mu )$-regularity).

Notice that both in (a) and in (a$'$) the condition inside the parenthesis
is equivalent to the condition outside,  since if $X\not\in D$
then $\sum_D D_i$ does not  depend on the $D_i$'s ($i\in X$).

Now, let us prove (a$'$).
Suppose that $\sum_D D_i$ is $(\lambda,\mu )$-regular, that is, by Form I,
there exist sets $(Z_\alpha)_{\alpha\in\mu} $ in $\sum_D D_i$
such that the intersection of any $\lambda $ of them is empty.
For every $\alpha\in\mu $, let
$X_\alpha=\{i\in I|\{ j\in I_i|(i,j)\in Z_\alpha \}\in D_i \} $.
Thus, for every $\alpha\in\mu $, $X_\alpha\in D$, since
$Z_\alpha \in \sum_D D_i$.

If $i\in X_\alpha $, let
$Y_{\alpha i}=\{j\in I_i|(i,j)\in Z_\alpha \} $. Thus,
$Y_{\alpha i}\in D_i$.
We claim that, for each $i\in I$,
the family $\{ Y_{\alpha i}|\alpha \in \mu$ is
such that $i\in X_\alpha\}$ witnesses the
$(\lambda, |\{\alpha\in\mu|i\in X_\alpha \} |)$-regularity of $D_i$.
If not, for some $i$ there is $B\subseteq \{ \alpha\in\mu|i\in X_\alpha \}$
with $|B|=\lambda $ such that
$\bigcap_{\alpha\in B}Y_{\alpha i} \not= \emptyset$.
If  $j\in \bigcap_{\alpha\in B}Y_{\alpha i} $,
then $(i,j)\in \bigcap_{\alpha\in B}Z_\alpha $, absurd,
since $(Z_\alpha)_{\alpha\in\mu} $ was supposed
to be a $(\lambda,\mu )$-regularizing family for $\sum_D D_i$.

Conversely, suppose that
there are 
$(X_\alpha )_{\alpha\in\mu }$ as given in (a$'$), and
for every $i\in I$ let
$\{ Y_{\alpha i}|\alpha\in \mu$ is
such that $i\in X_\alpha\}$ witness the
$(\lambda, |\{\alpha\in\mu|i\in X_\alpha \} |)$-regularity
of $D_i$.
For $\alpha\in\mu$, let $Z_ \alpha = \{(i,j)|i\in X_ \alpha \ {\rm and}\ j\in Y _{ \alpha i } \}$.
Thus, $Z_ \alpha \in\sum_D D_i$.

We claim that the $Z_ \alpha $'s witness the
$( \lambda , \mu )$-regularity of $\sum_D D_i$.
Indeed, if by contradiction $\bigcap _{ \alpha  \in B} Z_ \alpha  \not = \emptyset $
for some $B$ with $|B|= \lambda $, say
 $(i,j)\in\bigcap _{ \alpha   \in B} Z_ \alpha
$, then
$j\in \bigcap _{ \alpha  \in B} Y_{\alpha i}$,
and this contradicts the assumption that
$\{ Y_{\alpha i}|\alpha\in\mu$ is
such that $i\in X_\alpha\}$ witnesses the
$(\lambda, |\{\alpha\in\mu|i\in X_\alpha \} |)$-regularity
of $D_i$.

Hence, we have proved (a) and (a$'$).

(b) Since $\sum_D D_i$ is $(\lambda,\mu)$-regular, there is a function $g:I\to S(\mu)$ as given by (a). If 
$\{ i\in I | |g(i)|< \nu\} \in D$ then $D$ is
 $(\nu,\mu)$-regular by Form II since
we can change the values of $g$ for a set not in $D$ 
hence, without loss of generality, we can suppose that
$g:I\to S_\nu(\mu)$.

Otherwise, $\{ i\in I | |g(i)| \geq \nu\} \in D$, hence 
$\{i\in I| D_i\text{ is } (\lambda,\nu)\text{-regular}\}\in D $, by (a) and \ref{1.1}(i).

(c) If $D$ is $(\nu^+,\mu )$-regular, then, by Form II, there is
$g:I\to S_{\nu^+}(\mu)$ witnessing it. Thus, $|g(i)|\leq \nu$,
for every $i \in I$, hence, by \ref{1.1}(i) and (a),  
$\sum_D D_i$ is $(\lambda,\mu )$-regular.

(d) The only-if part is immediate from (b) with $\mu=\nu=\lambda $.

Conversely, if 
$\{i\in I| D_i \text{ is } (\lambda,\lambda )\text{-regular}\}\in D $
then (c) with $\mu=\nu=\lambda $ implies that 
$\sum_D D_i$ is $(\lambda,\lambda  )$-regular, by \ref{1.1}(xiii). 
On the other side, if $D$ is $(\lambda, \lambda )$-regular, then
$\sum_D D_i$ is $(\lambda,\lambda  )$-regular, by \ref{1.1}(ii),
since trivially $D \leq \sum_D D_i$.

Notice that in (b) and (d) we cannot replace
``$\{i\in I| D_i$ is $(\lambda,\nu)$-regular$\}\in D $'' by ``for every $i\in I$ $
D_i$ is $(\lambda,\nu)$-regular''. This is because if $X\not\in D$ then $\sum_D
D_i$ does not really depend on the ultrafilters $D_i$ ($i\in X$), hence the
$D_i$'s ($i \in X$) can be chosen arbitrarily.
\end{proof}

For example, suppose that, in the same model of
Set Theory, for every $n<\omega $ there is an ultrafilter $D_n$ which is
$(\omega_n,\omega_n )$-regular, and which for no  $m<n$ is
$(\omega_m,\omega_{m+1} )$-regular.
Then, if we take $D$  uniform over $\omega $, the sum $\sum_D D_n$ is
$(\omega_n,\omega_n )$-regular, for every $n<\omega $, but for no $n<\omega $
$\sum_D D_n$ is $(\omega_n,\omega_{n+1} )$-regular (compare with Problem
 \ref{2.14ex2.12}). Notice that, without loss of generality, each $D_n$ can 
be chosen
uniform over $\omega_n $, by  \ref{1.1}(viii)(iii)(ii); in this
case, $\sum_D D_n$ is uniform over
$\omega_\omega $.

When $D$ has a least function, the condition in Proposition \ref{7.4} can be
simplified.

\begin{cor}
\label{7.5}
  Suppose that $\mu $ is regular,  the ultrafilter $D$
is $(\mu,\mu)$-regular and  has a $\mu$-least function $f$. Then
$\sum_D D_i$ is $(\lambda,\mu )$-regular
if and only if $\{i\in I|D_i$ is $(\lambda, {\rm cf}f(i))$-regular$\}\in D $.
\end{cor}

\begin{proof} If
$\sum_D D_i$ is $(\lambda,\mu )$-regular
then 
there is a function $g:I\to S(\mu)$ as given by Proposition \ref{7.4}(a),
such that
for every $i\in I$ $D_i$ is $(\lambda,|g(i)|)$-regular.

For $i\in I$, let $f'(i)=\sup(g(i)\cap f(i))$. Thus, $f'(i)\leq f(i)$ for every $i\in I$.
Since  for every $\alpha<\mu $ both $\{i\in I|\alpha<f(i)\} $ and $\{i\in I|
\alpha\in g(i) \} $ belong to $D$, we have that $\{i\in I| \alpha < f'(i)\}\in D
$. Since $f$ is a $\mu$-least function, $\{i\in I|f(i)=f'(i)  \}\in D $, hence
$\{i\in I|{\rm cf} f(i)\leq |g(i)|  \}\in D $, and we are done by
\ref{1.1}(i).

For the converse, suppose that $\{i\in I|D_i$ is
$(\lambda, {\rm cf}f(i))$-regular$\}\in D $. 
The
arguments in the proof of  \cite[Theorem 1.3]{Ket1}  show that there is a $g:I\to
S(\mu)$
such that $\{i\in I\mid |g(i)|=
{\rm cf}  f(i)\}\in D $ and for every  $\alpha \in \mu$
$\{i \in I| \alpha\in g(i)\}\in D $. Then the conclusion is immediate from Proposition
\ref{7.4}(a).
\end{proof}

Corollary \ref{7.5} improves \cite[Theorem 5.6]{Ket1}.

Notice that if $D$ has a $ \mu$-least function,
 it is not necessarily the case that $D\times D$
 has a $ \mu$-least function (see Remark \ref{8.5}(b))

It is natural to ask whether Proposition \ref{7.4}  can be improved (and simplified)
to:
{\it ``$\sum_D D_i$ is $(\lambda,\mu )$-regular if and only if
there is a cardinal $\nu$ such that $D$ is $(\nu,\mu )$-regular and for every
$\nu'<\nu$  $\{i\in I| D_i \text{ is } (\lambda,\nu')\text{-regular}\}\in D$''}.
The next example shows that the above statement is false.

Let $D$ be $(\omega_\omega ,\omega_{\omega+1})$-regular, but not
$(\omega_n,\omega_n)$-regular for $n>0$ (as we mentioned,
\cite{BM}
produced such an ultrafilter).
By Theorem \ref{6.2} and \ref{1.1}(i), $D$ has an $\omega_{\omega+1}$-least function 
$f$ and, by Theorem
\ref{6.3} and \ref{1.1}(i), for every $n<\omega $ $\{i\in I| {\rm cf} f(i)>\omega_n 
\}\in D $.
For every $i\in I$, if ${\rm cf} f(i)=\omega_n$ and $n>0$, let $D_i$ be an
$(\omega,\omega_{n-1})$-regular ultrafilter over $\omega_{n-1}$; notice that  $\{ i\in I|
\cf f(i)=\omega \}\not\in D $, so that if $\cf f(i)=\omega $ then $D_i$ can be
chosen arbitrarily. Corollary \ref{7.5} implies that
$\sum_D D_i$ is not $(\omega,\omega_{\omega+1})$-regular, but for every
$n<\omega $ $\{i\in I|D_i$ is $(\omega,\omega_n)$-regular$\}\in D $, since
$\{i\in I| {\rm cf} f(i)>\omega_{n+1} \}\in D $.

Notice that, in the above example, 
$\sum_D D_i$ is $(\omega_\omega, \omega_{\omega+1})$-regular,
$(\omega,\omega_n)$-regular for all $n<\omega $,
hence $(\omega,\omega_\omega )$-regular by Proposition \ref{5.4},
but not $(\omega_n, \omega_{\omega+1})$-regular for $n<\omega $,
again by Corollary \ref{7.5}.

The following is a generalization of \cite[Theorem 5.9]{Ket1}.

\begin{prop}
\label{7.6}
 For every cardinals $\lambda,\mu, \chi$, and for every set $\mathcal K$ of cardinals,
the following are equivalent:

(a) There is a $\chi$-complete $(\lambda,\mu )$-regular ultrafilter which for no $\kappa\in \mathcal
K$ is $(\kappa,\kappa )$-regular.

(b) For every $\nu$ with $\lambda\leq\nu\leq\mu$
there is a $\chi$-complete $(\nu,\nu )$-regular  ultrafilter which for no $\kappa\in \mathcal K$ is
$(\kappa,\kappa )$-regular.

(c) For every $\nu$ with $\lambda\leq\nu\leq\mu$
there are an $n\in \omega $ and a $\chi$-complete 
$(\nu^{+n},\nu^{+n} )$-regular  ultrafilter which for no $\kappa\in \mathcal K$ is
$(\kappa,\kappa )$-regular.
\end{prop}

\begin{proof} (a)$\Rightarrow$(b) is trivial, by \ref{1.1}(i).

First, we prove the converse in the case $\chi=\omega $.
Let $D_\nu$ ($\lambda\leq\nu\leq\mu$) be  ultrafilters as
given by (b). Construct inductively,
for each $\nu$ with $\lambda\leq\nu\leq\mu$, a chain of models ${\bf A}_\nu $ as
follows.

${\bf A}_\lambda  $ is
$\prod_{D_\lambda }\langle  S _\lambda(\mu ), S_\kappa(\kappa) ,
\subseteq , \{ \alpha \} \rangle 
_{  \kappa \in {\mathcal K},  \alpha \in \mu \cup \sup{\mathcal K}}$;

${\bf A}_{\nu^+} $ is
$\prod_{D_{\nu^+}} {\bf A_\nu} $; and

${\bf A}_\nu $ is
$\prod_{D_{\nu}}
\left(
 \lim_{\nu'<\nu }{\bf A}_{\nu'}
\right)$,
 if $\nu$ is limit, where
$ \lim_{\nu'<\nu }{\bf A}_{\nu'}$
denotes
 the direct limit of the ${\bf A}_{\nu'}$'s with respect to the
natural embeddings.

Iterating the arguments in the proof of Proposition \ref{7.1}
it can be shown by induction on $\nu$
($\lambda\leq\nu\leq\mu$) that whenever $X\subseteq\mu $
and $|X|\leq\nu$ then there is
$x_{_X} \in {\bf A}_\nu $
 such that
$d(\{\alpha \}) \subseteq x_{_X} $ for every $\alpha\in X$.
In other words, for every $\nu$, ${\bf A}_\nu $
contains a copy of $S_{\nu^+}(\mu)$.

The base $\nu= \lambda $ of the induction is just Form III of the definition
of  $( \lambda , \lambda )$-regularity, since 
if $|X|=\lambda $ then $S_\lambda (X)$
is isomorphic to $S_\lambda (\lambda )$.

The successor step is dealt exactly as in the proof of \ref{7.1}:
by the inductive hypothesis,
${\bf A}_\nu $ contains a copy of $S_{\nu^+}(\mu) $, hence, by the
$( \nu^+ , \nu^+)$-regularity of $D _{\nu^+} $,
${\bf A}_{\nu^+} =\prod_{D_{\nu^+}} {\bf A_\nu} $
contains a copy of $S_{\nu ^{++}} (\mu) $.

The case
$\nu$ limit is similar:
$ \lim_{\nu'<\nu }{\bf A}_{\nu'}$
contains a copy of 
$S_{\nu'}(\mu)$
for each $\nu'<\nu$,
hence a copy of
$\bigcup_{\nu'<\nu}S_{\nu'}(\mu)=S_\nu(\mu)$, thus
$\prod_{D_{\nu}}
\left(
 \lim_{\nu'<\nu }{\bf A}_{\nu'}
\right)$
contains a copy of
$S _{\nu^+} (\mu)$,
by the $(\nu, \nu)$-regularity of $D_\nu$.

Thus, in  the final model ${\bf A}_\mu
$ there is an element $x$ such that
 $ d(\{\alpha \}) \subseteq x $ for every $\alpha\in \mu$;
now, ${\bf A}_\mu $
 is a {\it complete extension} (see \cite[Section 6.4]{CK})
of $\langle S_\lambda(\mu ), \subseteq \rangle$,
and by \cite[Theorem 6.4.4]{CK}
there is an ultrafilter $D$ such that $\prod_D \langle S_\lambda(\mu ),
\subseteq \rangle $ is embeddable in ${\bf A}_\mu  $ and $x$ is in the range of
the embedding. Thus, $D$ is $(\lambda,\mu )$-regular (Form III).

Let $\kappa\in \mathcal K$.
Since no $D_\nu$ is $(\kappa,\kappa )$-regular, we have that $D$ is not
$(\kappa,\kappa )$-regular, again by using the arguments in the proof of
Proposition \ref{7.1}. 
Indeed, at no stage of the construction of the $\mathbf A_{\nu}$'s
there can appear an element witnessing $(\kappa, \kappa )$-regularity.
A fortiori, no such element can be in $\prod_D S_\kappa (\kappa )$.
(This is the reason why
we have included the $S_\kappa(\kappa )$'s in our models).

Having proved
(b)$\Rightarrow$(a)
in the case $\chi=\omega $,
let now
$\chi$ be arbitrary. Since, as we mentioned in the introduction,
an ultrafilter $D$ is $\chi$-complete if and only if for no $\chi'<\chi$ $D$ is 
$(\chi',\chi')$-regular, then the result for $\chi $-complete ultrafilters follows
from the case $\chi=\omega $, by appropriately extending the set
$ \mathcal K$.

Thus, (b)$\Rightarrow$(a) is proved.

(c) is equivalent to (b) by Theorem \ref{2.1}(b). 
\end{proof}

Less direct proofs of Proposition \ref{7.6} can be obtained from the proof of 
\cite[Theorem 3]{Lp2} (stated here as Theorem \ref{4.5}) or from the proof of
\cite[Theorem 10]{lpndj}  (cf. also \cite[Theorem 7]{Lp9} ; notice that there the order of
$\lambda $ and $\mu $ is exchanged, in the definition of regularity).

If $I$ is a finite set, we can generalize Proposition \ref{7.6} to the effect that
there is a $\chi$-complete ultrafilter which is
$(\lambda_i,\mu_i )$-regular for every $i\in I$ and
not $(\kappa,\kappa )$-regular
for $\kappa\in \mathcal K$ if and only if for every $i\in I$ and for every $\nu $
with $\lambda_i\leq\nu\leq\mu_i$ there is a $\chi$-complete ultrafilter which is
$(\nu,\nu )$-regular
and
not $(\kappa,\kappa )$-regular
for $\kappa\in \mathcal K$.

This is because for each $i \in I$ Proposition \ref{7.6} gives
a $(\lambda_i,\mu_i )$-regular not $(\kappa,\kappa )$-regular ultrafilter $D_i$, and then, letting $I= \{i_1, \dots, i_n\} $, $D_{i_1}\times
D_{i_2}\times\dots \times D_{i_n}$ is the desired ultrafilter, by Proposition \ref{7.1}.

The above statement is not true when $I$ is infinite,
already in the case $\lambda_i=\mu_i$.
Suppose that GCH holds, and that $(\mu_i)_{i\in\omega } $ 
is a strictly increasing sequence of measurable
cardinals, and let $\kappa =\sup \mu_i$. Then for every $i<\omega $ there is a
$(\mu_i,\mu_i )$-regular not $(\kappa ,\kappa  )$-regular ultrafilter; but GCH,
\ref{1.1}(vii)(xi)  and
Theorems \ref{5.1}(a) and  \ref{2.1}(b) imply that any ultrafilter which 
is
$(\mu_i,\mu_i )$-regular for every $i<\omega $ is $(\kappa ,\kappa )$-regular.

\section{Further remarks.}
\label{fr}

In this section we add a few disparate remarks (of course, we state some
problems, too, in order  to keep on with the tradition).

\begin{prop}
\label{8.1}
  Suppose that $ \lambda, \mu, \kappa $ are regular cardinals, 
and that there is a sequence $(f_\alpha
)_{\alpha\in\kappa }$ of functions from $\lambda $ to $\mu $ which is increasing
modulo eventual dominance (that is, for every $\alpha<\beta<\kappa $ there is
$\gamma <\lambda $ such that $f_\alpha(\delta )<f_\beta(\delta )$, for every
$\delta>\gamma $), and suppose that there is no function from $\lambda $ to $\mu
$ which eventually dominates all the $f_\alpha $'s.
Then every $( \kappa, \kappa) $-regular ultrafilter is either
$(\lambda,\lambda )$-regular or
$(\mu,\mu )$-regular.
\end{prop}

\begin{proof} Consider a model {\bf A}
 with unary predicates $U, V, W$ representing
$\kappa,\lambda,\mu $ respectively, with a binary predicate 
$<$ representing the
well orders of $\kappa,\lambda,\mu $, and a ternary relation $R$ such that for
$\alpha\in\kappa $
$R(\alpha,-,-)$ represents the diagram of $f_\alpha $.
Thus, {\bf A} satisfies

$\forall xy (U(x) \wedge U(y) \wedge x<y
\Rightarrow \exists z (V(z) \wedge \forall z'>z
(V(z') \Rightarrow \forall w,w' (R(x,z',w) \wedge R(y,z',w')
\Rightarrow w<w'))))$

Let $D$ be an ultrafilter, consider the ultrapower of the above model, and recall that, by
 \ref{1.1}(xi), if $ \kappa $ is regular, then $\kappa $-descending
incompleteness
is equivalent to $( \kappa, \kappa) $-regularity.
If $D$ is $\kappa $-descendingly incomplete, then in $\prod_D U$ there
is an element $x$ greater than all $d(\alpha)$'s ($\alpha\in\kappa $). Now,
$R(x,-,-)$ is the diagram of a function $g$ from $\prod_D V$
to $\prod_D W$,
and, by the above-displayed formula and \L o\v s Theorem,
 $g$ eventually dominates all the functions with diagram given
by $R(d( \alpha ) ,-,-)$ ($ \alpha \in \kappa $).

If $D$ is not
$(\lambda,\lambda )$-regular, that is, not
$\lambda  $-descendingly incomplete, then the $z$ whose existence
is asserted by (*) is bounded by some $d( \gamma )$ with
$ \gamma < \lambda $, hence, without loss of generality, we can assume
$z=d( \gamma )$.
In particular, for every $ \alpha \in \kappa$ there is $ \gamma_ \alpha  \in \lambda $
such that, from $ \gamma_ \alpha $ on, $g$ dominates the function with diagram given
by $R(d( \alpha ) ,-,-)$.

If $D$ is not 
$(\mu,\mu )$-regular, that is, not
$\mu  $-descendingly incomplete, define,
for $\gamma\in \lambda $,
$g'(\gamma)=\inf \{\delta\in\mu | g(d(\gamma ))\leq d(\delta ) \}$.
Thus, $g': \lambda \to \mu$
and $g_{|\lambda }\leq g'$ pointwise.
Translating (in {\bf A}) the fact that,
from $ \gamma_ \alpha $ on, $g$ dominates the function with diagram given
by $R(d( \alpha ),-,-)$, we have that,
 from $ \gamma_ \alpha $ on, $g'$ dominates $f_ \alpha $,
thus  $g'$ dominates all the $f_ \alpha $'s, contradicting our assumption.
\end{proof}

The above proposition  might be relevant to the problems discussed in 
Section \ref{dfe1}. See
\cite{BK}, \cite[p. 180]{KM} and \cite[Theorem 0.25]{Lp1}  for connections 
between regularity of
ultrafilters and the existence of families of eventually different functions.
Probably, the argument in Proposition  \ref{8.1} can be elaborated further.

In some cases, we can prove Theorem  \ref{2.2} without the hypothesis 
${\rm
cf}\mu\not={\rm cf}\lambda $.
The simplest case is when ${\rm cf}\lambda={\rm cf}\mu=\omega  $, $\lambda $ has
the form $\nu^{+\omega }$ for some $\nu$, while $\mu=\sup_{n\in\omega } \mu_n$,
where the $\mu_n$'s can be chosen to be limit cardinals.
This is a consequence of the next proposition (it is case $\alpha=2$). In order to prove the
general form, we need a definition.

If ${\rm cf}\mu=\omega $, define as follows the relation ``the {\it order } of
$\mu $ is $\geq\alpha $'', for $\alpha \not=0$ an ordinal.

The order of every $\mu $ of cofinality $\omega $ is $\geq1$;

If $\alpha>1$, the order of $\mu $ is $\geq\alpha $ if and only if for every $\beta<\alpha $
$\mu $ is a limit of some sequence of cardinals, each of cofinality $\omega $ and of order $\geq\beta
$.

Of course, we could define the order of $\mu$ to be the least $\alpha $ such
that the order of $\mu$ is $ \geq\alpha $, but we shall not actually need this.

\begin{prop}
\label{8.2}
  Suppose that $\lambda $ and $\mu$ are infinite cardinals, and

(a) $\alpha>0 $ and $\alpha \leq$ the first weakly inaccessible cardinal (or there is
none); and

(b) $\mu $ has cofinality $\omega $ and order $\geq\alpha $; and

(c) $\lambda<\mu $, and $\lambda=\nu^{+\gamma }$, for some $\gamma<\omega^\alpha
$ (ordinal exponentiation), and some regular cardinal $\nu$.

If $D$ is a $(\lambda^+,\mu )$-regular ultrafilter, then $D$ is $(\lambda,\mu
)$-regular. Moreover, $D$ is
either
 $(\omega,\omega )$-regular, or
$(\nu,\mu )$-regular.
\end{prop}

\begin{proof} If ${\rm cf} \lambda > \omega $
then $D$ is $( \lambda , \mu )$-regular by
Theorem  \ref{2.15ex2.13}(ii).

 If ${\rm cf} \lambda = \omega $ and
$D$ is $(\omega,\omega )$-regular then $D$ is
$(\lambda,\mu )$-regular by Theorem  \ref{2.13ex2.11}(iii) and  
\ref{1.5}(a), since $\cf \prod_D \lambda = \cf \prod_D \cf \lambda $.

We shall prove by induction on $\alpha $ that the proposition is true for every
ultrafilter which is not $(\omega,\omega ) $-regular.

The case $\alpha=1$ is a particular case of Theorem  \ref{2.15ex2.13}(ii),
since in this case $\lambda=\nu^{+n}$ for some $n<\omega $,
and $\nu$ is regular.

Suppose the statement of the proposition true for all
$\beta<\alpha $, and let $\lambda,\mu, \nu, \gamma  $ and $D$ be given.
Let $\delta $ be the smallest ordinal such that $D$ is
$(\nu^{+\delta },\mu )$-regular. Notice that 
$\delta \leq \gamma +1$, since $D$ is $(\nu^{+\gamma+1}, \mu)$-regular,
thus $\nu^{+\delta }<\mu$.
Notice also that $\gamma+1<\omega^\alpha $,
since $\alpha>0$, hence $\omega^\alpha $ is limit.

We want to show that $\delta=0$. If not,
by Theorem  \ref{2.15ex2.13}(ii), either $\delta $ is limit, or
$\delta=\varepsilon+1$ and $\cf\varepsilon=\omega  $.

Since $D$ is not $(\omega,\omega )$-regular, $D$ is $\omega_1$-complete, hence
$\mu^*$-complete, where $\mu^*$ is the first measurable cardinal, hence for all
$\kappa <\mu^*$ $D$ is not $(\kappa,\kappa )$-regular.
We now show that $\delta $ is not limit.
If $\delta $ is limit, then,
since $\delta\leq\gamma+1<\omega^\alpha $, we  necessarily have
$\cf\delta <\sup\{\alpha, \omega_1\}\leq $ the first weakly inaccessible cardinal, hence
 $\delta $ cannot be limit by Proposition  \ref{2.6}, since
then $\nu^{+\delta }$ would be a singular cardinal, $\delta $ being smaller than
the first weakly inaccessible cardinal,
and since $D$ is not $(\cf \nu^{+\delta },\cf \nu^{+\delta })$-regular,
$\cf \nu^{+\delta }$ being smaller than the first measurable cardinal.

So, let $\delta=\varepsilon+1$ and ${\rm cf}\varepsilon=\omega  $.
Since
$\varepsilon<\delta\leq\gamma+1<\omega^\alpha $, then,
by expressing $ \varepsilon $ in normal form, we get that
$\varepsilon $ has the
form
$\varepsilon'+\omega^\beta $, for some $\beta<\alpha $, $\beta>0$.
By hypothesis, $\mu $ has order $\geq \alpha $, so that $\mu=\sup_{n\in\omega }
\mu_n$ for certain $\mu_n$'s of cofinality $\omega $ and order $\geq\beta $.
Without loss of generality, we can assume that $\mu_n>\nu^{+\varepsilon }$,
for every $n\in \omega $, since $\varepsilon \leq \gamma $, 
$\sup_{n\in\omega } \mu_n = \mu > \lambda = \nu^{+\gamma }$,
and $\mu$ is limit. 

Let us fix $n$. By the definition of $\delta $, $D$ is
$(\nu^{+\delta },\mu )$-regular, that is, $(\nu^{+\varepsilon+1 },\mu
)$-regular,
hence, by \ref{1.1}(i), $(\nu^{+\varepsilon+1},\mu_n^+ )$-regular, and
$(\nu^{+\varepsilon
}, \mu_n)$-regular by Theorem  \ref{2.13ex2.11}(ii). By Proposition  
\ref{2.6}, and since $D$ is not $(\omega, \omega )$-regular, 
there is
$\eta<\varepsilon= \varepsilon'+ \omega^\beta   $
 such that $D$ is $(\nu^{+\eta },\mu_n )$-regular;
hence either $\eta\leq \varepsilon'$ or
$\eta= \varepsilon' + \eta'$ with $\eta'< \omega^ \beta  $.
In both cases, 
$D$ is   $(\nu^{+\varepsilon'+1 },
\mu_n)$-regular: if 
$\eta\leq \varepsilon'$, by \ref{1.1}(i);
if $\eta= \varepsilon' + \eta'$ with $\eta'< \omega^ \beta  $
by the
inductive hypothesis
with 
$\beta $ in place of $\alpha $, $\mu_n$ in place of $\mu$,
$\eta$ or $\eta-1$ in place of $\gamma $, 
$\nu ^{+ \varepsilon' +1} $ in place of $\nu$,
and $\nu ^{+n} $ in place of $ \lambda $), and since
$\nu^{+\varepsilon'+1 }$ is a regular cardinal.

Since $\varepsilon'$ does not depend on $n$, then $D$ is
$(\nu^{+\varepsilon'+1},\mu)$-regular by Proposition \ref{5.4}(a), 
but this
contradicts the choice of $\delta $, so
$\delta=0$ and $D$ is
$(\nu,\mu )$-regular.
\end{proof}

Proposition  \ref{8.2} can be used in order to improve
 Proposition \ref{5.4} 
 a bit.
Anyone of the assumptions (a)-(d) in the hypotheses
of Proposition \ref{5.4} can be replaced by the conjunction
of the assumptions (a)-(c) stated in the hypothesis
of Proposition \ref{8.2}. We do not know whether 
the hypothesis ${\rm cf}\lambda\not={\rm 
cf}\kappa $ in Theorem \ref{6.5} (cases (a)(b)), too, can be
replaced by the conjunction
of the assumptions (a)-(c) of Proposition \ref{8.2}.

We do not know whether a result similar to Proposition \ref{8.2} can be proved for
cardinals $\lambda,\mu $ of cofinality $>\omega $. The argument in the proof of  \ref{8.2} breaks
since if $D$ is not  $({\rm cf\lambda },{\rm cf\lambda } )$-regular, nonetheless
$D$ might be $(\lambda',\lambda')$-regular for some
$\lambda'<{\rm cf\lambda } $. In other words, we can prove a result analogous
to \ref{8.2} only if we assume $({\rm cf}\lambda )^+$-completeness.
The situation is similar to the proof of Corollary \ref{5.3}.

We turn to another kind of problem.

\begin{problem}
 \label{8.3}
Are the regularity properties of an ultrafilter $D$
determined by the function $\lambda \to {\rm
cf}(\prod_D\langle\lambda,<\rangle)$?

More precisely, is it always true that if $D$ and $D'$ are ultrafilters
and ${\rm cf}(\prod_D\langle\lambda,<\rangle) = {\rm
cf}(\prod_{D'}\langle\lambda,<\rangle)$ for every cardinal $\lambda $ then for
every pair of cardinals $\nu,\mu $
$D$ is $(\mu,\nu)$-regular if and only if $D'$ is $(\mu,\nu)$-regular?

\end{problem}

Theorem  \ref{2.13ex2.11} suggests that the answer to Problem \ref{8.3} might be 
affirmative.

\begin{problem}
\label{8.4}
  Which results of the present paper generalize to the notion
$(\lambda,\mu )\Rightarrow(\lambda',\mu')$ of \cite[Definition 0.12]{Lp1} or \cite[p. 139]{easter5}. 
See also \cite[Definition 1.2]{Lp3} and \cite{nuotop}.
\end{problem}

We suspect that most results generalize; the problem is whether we need the
parameter $\kappa $ of \cite[0.21(c)]{Lp1} and, if this is necessary, to determine the smallest
possible value of this parameter. Most results in Section \ref{al} do 
not generalize,
unless in the definition of
$(\lambda,\mu )\Rightarrow(\lambda',\mu')$ one considers only {\it simple}
extensions (that is, extensions generated by a single element: in order to make
sense, one has to deal only with models having Skolem functions).

\begin{remarks}
\label{8.5}
  (a) The notions of a least function and
of a $ \kappa $-least function
(Definition \ref{6.1}) are interesting only in the case
$ \kappa $ regular. Indeed, an
ultrafilter $D$ has a $ \kappa $-least function if and only if
it has a $\cf \kappa $-least function,
as we are going to show.

Let $ \kappa $ be singular, and let $C=( \beta _ \alpha ) _{ \alpha \in \cf\kappa } $ be
a sequence of order type cf$ \kappa $ closed and unbounded in $ \kappa $.

If $D$ is over $I$, and $g:I \to \kappa $ is
a $ \kappa $-least function, let
$h:I \to  \kappa $ be defined by
$h(i)=\sup \{ \beta _ \alpha \in C| \beta _ \alpha \leq g(i)\} $:
for every $i$, $h(i)$ belongs to $C$, since $C$ is closed.
Moreover,
$h(i)\leq g(i)$. 
Notice that, for every $\alpha \in \cf \kappa $, $\beta_\alpha \leq h(i)$
if and only if $\beta_\alpha \leq g(i)$.
It cannot be the case
that $h$ is bounded mod $D$ by some $ \beta < \kappa  $: were this the case,
take $ \beta _ \alpha  \in C$ such that $ \beta _ \alpha \geq \beta $;
then $ \{i\in I | \beta < h(i)\}\supseteq
\{i\in I | \beta _ \alpha < h(i)\} \in D $.
Since $g$ is a $ \kappa $-least function,
then $ \{i\in I|h(i)=g(i)\}\in D $. 
This implies that the function $h'$
defined by $h'(i)=$ the $ \alpha \in \kappa $ such that
$ \beta_ \alpha =h(i)$ is a $\cf \kappa $-least function.

Conversely, suppose that $h:I \to {\rm cf} \kappa  $ is a
$\cf \kappa $-least function.
Without loss of generality, we can suppose that
$\{i \in I| h(i) \text{ is limit}\} \in D$. 
Define $g:I \to \kappa  $ by $g(i)= \beta _{h(i)} $.
We want to show that $g$ is a
$ \kappa $-least function.
Indeed, $g$ is clearly unbounded in  $\kappa \pmod D$.
If 
$f:I\to \kappa$,
$f \leq g \pmod D$ and
 $f$ is unbounded in  $\kappa \pmod D $,
let  $f':I \to \cf \kappa$
be defined by 
 $f'(i)=$
the least 
$\alpha\in \cf \kappa $ such that $f(i)\leq \beta_\alpha $.
$f'$ is unbounded in $\cf\kappa \pmod D $, since
were $f'$ bounded $\pmod D $ by $\alpha\in \cf \kappa $,
then $f$ would be bounded $\pmod D $ in $\kappa $ by $\beta_\alpha $,
absurd.
Since $h$ is a $\cf \kappa $-least function, then $h \leq f' \pmod D $;
thus $g(i)=\beta_{h(i)} \leq \beta_{f'(i)}$ for $i$ in a set in $D$.
Thus, for a set in $D$, 
$f(i) \leq g(i) \leq \beta_{f'(i)}$.
Since the image
of $g$ is contained in $C$,
 by the definition of $f'$,
we have $g(i)=\beta_{f'(i)}$.
Again by the definition of $f'$,
and since $g(i)$ is a limit point of 
$C$, we get $f(i)=g(i)$.
Thus, $f=g \pmod D$, hence $g$
is a $\kappa $-least function.

(b) It may happen that $D$ has a least function
 but $D\times D$
 has no least function.
\cite{FMS} constructed an ultrafilter  $D$
uniform over $ \omega _1$, hence
$( \omega _1, \omega _1)$-regular,  which is not
$( \omega, \omega _1)$-regular, hence
$D$ has an $ \omega_1 $-least function
(equivalently, a least function, since $D$
is uniform over $ \omega _1$)
by 
\cite[Section 2]{BK}, \cite[Theorem 2.3]{Ka1}.
However $D\times D$ is
$( \omega, \omega _1)$-regular by Theorem \ref{7.1}, and cannot have
a least function by
\cite[Theorem 2.4]{Ket1}  (see Theorem \ref{6.3}).

(c) The above remarks lead to the following definition.
If $\langle X, \leq \rangle$ is a linear order, let us say
that an ultrafilter $D$ over $I$ has an 
{\it $\langle X, \leq \rangle$-least function}
(or, simply, an {\it 
$X$-least function}, if the order on $X$ is understood)
if and only if there exists
a function $f:I \to X$ such that
$\{i\in I|x<f(i)\}\in D$,
 for every $x \in X$, yet for every $g:I \to X$,
$\{i\in I|g(i)<f(i)\}\in D$ implies that there is $x \in X$ such that
$\{i\in I|g(i)<x\}\in D$.
In other words, in $\prod_D \langle X,\leq \rangle$, $f_D$ is the least
element larger than all the $d(x)$'s ($x \in X$).

Thus, a $ \kappa $-least function as defined in  \ref{6.1}
is the same as a $\langle \kappa ,\leq \rangle$-least
function in the above sense.

Of course, there are orders $X$ for which the existence
of an $X$-least function is forbidden, for example, take $X$ to be any order
in which every element has an immediate predecessor (by elementarity,
this is true in $\prod_D X$, too, hence no least function is possible).

Notice that if $D$ is the ultrafilter given in (b),
then
$D$ has an $ \omega_1 $-least function
but not an
 $ X $-least function, where
$X=\{x\in \prod_{D} \omega_1 |$ for some $\alpha<\omega_1, x<d(\alpha 
)  
\} $.

However, the arguments in (a) give the following. Suppose that
$X$ is a linear order, and $C\subseteq X$ is a well order cofinal in
$X$ of type $\kappa $
($\kappa$ a regular cardinal) and
such that

(*) whenever $H\subseteq C$ is nonempty,  $\sup H$ computed
in $X$ exists and is the same as computed in $C$.

Then
an ultrafilter $D$ has  an $X$-least function if and only if it
has a $\kappa $-least function.

(d) The remarks in (b) and (c) above are connected as follows. We
get a condition under which
 $D \times D'$ has a $\kappa $-least function.

Suppose that $D$,  $D'$ are ultrafilters, and let
$X=\{x\in \prod_{D'} \kappa |$ for some $\alpha<\kappa, x<d(\alpha )  
\} $.
Thus, $X$ has cofinality $\kappa $.
If there exists a
well order $C$ cofinal in
$X$ and satisfying (*) above, then $D$ has a $\kappa $-least function
if and only if 
 $D \times D'$ has a $\kappa $-least function: just 
apply
(c) with $X$ as above,
recalling that $\prod_{D \times D'} \kappa
\cong \prod_D \prod_{D'} \kappa$.
\end{remarks}

In Proposition \ref{nuo2.9ora2.11}
we showed that if $D$ is a 
 $(\lambda,\lambda  )$-regular
 non $({\rm cf}\lambda,{\rm cf}\lambda )$-regular
ultrafilter then $\Box_\lambda$ fails. 
We can show that $\Box_\mu$ fails for many more cardinals.

\begin{prop}
\label{vec2.9}
  Suppose that $\lambda $ is singular, and $D$ is
a $(\lambda,\lambda  )$-regular
 not $(\cf\lambda,\cf\lambda )$-regular
ultrafilter.
Then there is $\mu<\lambda $ such that  for every
$\kappa$ with $\mu \leq \kappa \leq \lambda$
 $D$ is
$(\kappa^+,\kappa^+)$-regular and $\Box_\kappa$ fails.
Moreover,  either

(a) $\mu$ is singular, or

(b) $\lambda =\mu^{+\omega }$, 
$\mu$ is a regular limit $\omega $-weakly-Mahlo
cardinal,
 and $D$ is
$(\mu,\lambda  )$-regular.
\end{prop}

\begin{proof}
We shall show that there exists some
$\mu<\lambda $ such that $\mu>{\rm cf}\lambda $ and $D$ is
$(\kappa^+,\kappa^+)$-regular 
for every
$\kappa$ with $\mu \leq \kappa \leq \lambda$.
 This necessarily implies that
$\Box_\kappa$ fails.
Suppose by contradiction that
$\Box_\kappa$ holds
for some 
$\kappa$ with $\mu \leq \kappa \leq \lambda$.
Then Theorem \ref{box}  implies that $D$
is $(\kappa' ,\kappa' )$-regular for every
$\kappa'\leq\mu$, but this contradicts the assumption that 
$D$
is not $({\rm cf}\lambda,{\rm cf}\lambda )$-regular, since
$\mu>{\rm cf}\lambda$.

Now we show how to find $\mu$ as above.
By Proposition  \ref{2.6}, $D$ is
 $(\lambda',\lambda  )$-regular for some $\lambda'<\lambda$.
By Corollary \ref{limit} and \ref{1.1}(xi), $D$ is
$(\lambda^+, \lambda^+)$-regular.

Suppose that $\lambda$ has not the form $\nu^{+\omega}$ for some $\nu$.
Then there is a singular cardinal $\mu$ with $\lambda'<\mu<\lambda$
and $\mu>{\rm cf}\lambda $, hence by  \ref{1.1}(i)
$D$ is
$(\kappa^+,\kappa^+)$-regular 
for every
$\kappa$ with $\mu \leq \kappa < \lambda$,
and we fall in case (a).

Hence we can suppose that
$\lambda$ has the form $\nu^{+\omega}$ for some $\nu$;
hence $\cf \lambda= \omega$; moreover, 
without loss of generality, we can suppose that $\nu$ is a limit cardinal.
$D$ is $(\nu^+,\lambda )$-regular 
by \cite[Corollary B]{Lp5}, stated here as Theorem  \ref{2.15ex2.13}(ii),
hence, by \ref{1.1}(i), 
$D$ is
$(\kappa^+,\kappa^+)$-regular 
for every
$\kappa$ with $\nu \leq \kappa < \lambda$.

If $\nu$ is singular, just take $\mu=\nu$, and
 we are in case (a). Hence, suppose
that $\nu$ is a regular cardinal.
Again by \cite[Corollary B]{Lp5}, $D$ is $(\nu,\lambda )$-regular.
Since $\cf \lambda = \omega$
and $D$ is not $(\cf\lambda, \cf\lambda)$-regular,
then necessarily $\nu > \omega$.
If $\nu$ is
$\omega $-weakly-Mahlo, then (b) holds with
$\mu= \nu$.
By  \ref{1.1}(i),
$D$ is $(\nu,\nu)$-regular, and 
if $\nu$ is not
$\omega $-weakly-Mahlo, then the result from \cite{CC} stated here 
as Theorem  \ref{2.19ex2.17}
implies that $D$ is
either $(\mu,\nu )$-regular for
some $\mu<\nu$, or $(\mu,\mu)$-regular for all $\mu<\nu$. By
\ref{1.1}(i),
and since $\nu$ is a limit cardinal of uncountable cofinality, we can choose
$\mu$ to be a singular cardinal in such a way that
 $\mu>{\rm cf}\lambda= \omega $. Hence we are in case (a).
\end{proof}

Parts of the above proof of Proposition  \ref{vec2.9} essentially 
appeared, in a 
somewhat
hidden form, in the course of the proof of \cite[Proposition 0.22]{Lp1}.

The next proposition appeared in a previous version of this paper,
where it  has been used in order to 
show
that if  $\mu $ is a singular cardinal of cofinality $ \omega $, $\lambda^ \omega <\mu $ and  the ultrafilter
$D$ is
$(\lambda,\mu)$-regular, then $D$ is either $\mu $-decomposable or
$\mu^+$-decomposable, a result now subsumed by Theorem \ref{abst},
via \ref{1.1}(i).

However, the next proposition  appears to be of independent interest;
the main idea of its proof is probably due to R. Solovay
(see \cite[p. 74]{Ket1}).

\begin{prop}
\label{ex5.10}
  If $D$ is $(\lambda,\mu )$-regular and
$\kappa $-complete then $D$ is
$((\lambda^{<\kappa })^+,\mu^{<\kappa } )$-regular. Actually, if $\nu^{<\kappa
}<\lambda^{<\kappa }$ for every $\nu<\lambda $ then $D$ is
$(\lambda^{<\kappa },\mu^{<\kappa } )$-regular.
\end{prop}

\begin{proof} Let $X_\alpha $
 $(\alpha\in\mu )$ witness the $(\lambda,\mu )$-regularity
of $D$ (Form I).
 For $y\subseteq\mu $ with $|y|<\kappa $ let $ Z_y=\bigcap_{\alpha\in y}
X_\alpha$. Each $Z_y$ is in $D$, since $D$ is $\kappa $-complete. If
$x\subseteq\mu $ with $|x|<\lambda$ then 
$|\{Z_y|y\subseteq  x, |y|<\kappa  \}| \leq|x|^{<\kappa }$.
 Thus, if $Y\subseteq S_\kappa(\mu )$ and $|Y|>\nu^{<\kappa }$
for every $\nu<\lambda $, then
 $|\cup Y|\geq\lambda $; hence $\bigcap_{y\in Y}
Z_y=\bigcap_{\alpha\in \cup Y} X_\alpha = \emptyset$.
Since $|S_\kappa(\mu)|=\mu^{<\kappa}$
 the $Z_y$'s
witness the $((\lambda^{<\kappa })^+, \mu^{<\kappa })$-regularity
 of $D$, respectively,  the
$(\lambda^{<\kappa }, \mu^{<\kappa })$-regularity 
of $D$ if the assumption in the
second statement holds.
 \end{proof}

The following interesting corollary
is hardly mentioned in the literature.

\begin{corollary}\label{dopoex5.10}
If $\kappa $ is $\mu$-compact, $\mu$
is singular and $\cf \mu< \kappa $
then $\kappa $ is $\mu^+$-compact.
 \end{corollary}

\begin{proof}
By $\mu$-compactness there is a $\kappa $-complete
$(\kappa, \mu)$-regular ultrafilter $D$.
By Proposition \ref{ex5.10}
with $(\cf\mu)^+$ in place of $\kappa $,
$D$ is
$(\kappa^{\cf\mu}, \mu^{\cf\mu})$-regular,
since $\kappa $ is strongly inaccessible.
Since  $\mu^{\cf\mu}>\mu$,
then $D$ is $( \kappa , \mu^+)$-regular by
 \ref{1.1}(i), hence $\kappa $ is $\mu^+$-compact.
\end{proof}

Some of the results presented in this paper have been obtained in 
1995, while
the author was visiting the University of Cagliari. A preliminary version of
this paper has been circulating since 1996. That version contained essentially
all the results proved in Sections \ref{fstp}, \ref{dfe1} and 
\ref{al} here. The introduction, too, had been written
in 1996 (at that time, one had the feeling that independence results were taken
in much more consideration than ZFC results; now things are rapidly changing).

Slightly less general versions of the results in Section \ref{dfe2} have been announced in the abstract 
\cite{Lp7}.

We have announced further results about regularity of ultrafilters in the
abstracts \cite{Lp6} and \cite{Lp7};
 however we have found a gap in a proof, so that some
statements in \cite{Lp6, Lp7} have to be considered as problems, or conjectures, so far.

This work has been performed under the auspices of GNSAGA (CNR).

\begin{problem}
\label{ex8.7}
  Which results about regularity of ultrafilters
(in particular, which results of the present paper) hold assuming just the
Prime Ideal Theorem, rather than the Axiom of Choice?
\end{problem}

We wish to express our warmest gratitude
to an anonymous referee for a careful reading of the manuscript,
for a great deal of suggestions that helped improve the exposition and for
detecting some inaccuracies. Last but not least, we appreciate
encouragement in our efforts to ``keep this neglected area of
set theory alive''.

\end{document}